\documentclass{amsart}

\usepackage[T1]{fontenc}
\usepackage[utf8]{inputenc}
\usepackage[english]{babel}
\usepackage{amsfonts}
\usepackage{amsmath}

\usepackage{amsthm}
\usepackage{paralist}
\usepackage{amssymb}
\usepackage{mathrsfs}
\usepackage{mathabx} 
\usepackage{comment} 
\usepackage{color} 
\usepackage{xcolor}
\usepackage[a4paper,top=3cm,bottom=3cm]{geometry}

\usepackage[colorlinks=true, urlcolor=blue, linkcolor=blue, citecolor=blue]{hyperref} 

\definecolor{darkgreen}{rgb}{0.0, 0.5, 0.0}

\numberwithin{equation}{section} 

\newcommand{\Chi}{\mathcal{X}} 

\newcommand{\Supp}{\operatorname{spt}}

\newcommand{\Range}{\mathscr{R}}
\newcommand{\Kernel}{\mathscr{N}}
\newcommand{\Def}{\mathscr{D}}

\newcommand{\BFq}{\operatorname{BF}_q(\partial \Omega)}

\newcommand{\BF}{\operatorname{BF}}
\newcommand{\rot}{\mathrm{curl} \thinspace}
\newcommand{\LaplacePS}{\Delta_\mathrm{PS}}
\newcommand{\LaplacePSq}{\Delta_{\mathrm{PS},q}}

\newcommand{\LaplaceN}[1]{\Delta_{\mathrm{N},{#1}}}

\newcommand{\Stokesqpm}{A_{\mathrm{S},\alpha,q}^\pm}
\newcommand{\Stokesppm}{A_{\mathrm{S},\alpha,p}^\pm}

\newcommand{\Stokespm}{A_{\mathrm{S},\alpha}^\pm}
\newcommand{\Stokes}[1]{A_{\mathrm{S},\alpha,{#1}}^\pm}
\newcommand{\RN}[1]{\uppercase\expandafter{\romannumeral#1}}
\newcommand{\Div}{\operatorname{div}}
\newcommand{\dual}[2]{\langle #1 , #2 \rangle} 
\newcommand{\dualq}[2]{\langle #1 , #2 \rangle_{q,q'}}
\newcommand{\dualb}[2]{\langle #1 , #2 \rangle_{\partial \Omega}}

\newcommand{\dualGb}[2]{\big\langle #1 , #2 \big\rangle_{\partial \Omega}}
\newcommand{\Gradq}{G_q(\Omega)} 
\newcommand{\Grad}[1]{G_{#1}(\Omega)}
\newcommand{\modGradq}{\mathcal{G}_q(\Omega)}
\newcommand{\Tr}{\operatorname{Tr}}
\newcommand{\R}{\operatorname{R}_{\partial \Omega}}

\newcommand{\Helmholtz}{\mathbb{P}}

\newcommand{\dist}{\mathrm{dist}}
\newcommand{\D}{\thinspace d} 
\newcommand{\Dm}{\mathrm{D}_-}
\newcommand{\Dp}{\mathrm{D}_+}
\newcommand{\Dpm}{\mathrm{D}_\pm}
\newcommand{\ti}[1]{\widetilde{#1}}
\newcommand{\e}{\nu_+}
\newcommand{\codim}{\operatorname{codim}}

\newcommand{\Proj}{\widetilde{\mathbb{P}}}
\newcommand{\QProj}{\mathbb{Q}}

\newcommand{\thcite}[1]{{\upshape\cite{#1}}} 
\newcommand{\thref}[1]{{\upshape\ref{#1}}}

\newcommand{\transA}[1]{{#1}^T}
\newcommand{\transB}[1]{#1}

\excludecomment{AUSBLENDEN}

\newtheorem{proposition}{Proposition}[section]
\newtheorem{theorem}[proposition]{Theorem}
\newtheorem{lemma}[proposition]{Lemma}

\theoremstyle{definition}
\newtheorem{definition}[proposition]{Definition}

\newtheorem{remark}[proposition]{Remark}

\newtheorem{assumption}[proposition]{Assumption}

\begin{document}

\title[Stokes and Navier-Stokes subject to partial slip 
on uniform $C^{2,1}$-domains]{Stokes and Navier-Stokes equations subject to partial slip 
on uniform $C^{2,1}$-domains in $L_q$-spaces}

\author{Pascal Hobus and J\"urgen Saal}
		  



\begin{abstract}
This note concerns well-posedness of the Stokes and Navier-Stokes 
equations on uniform $C^{2,1}$-domains on $L_q$. 
In particular, classes of non-Helmholtz
domains, i.e., domains for which
the Helmholtz decomposition does not exist, are adressed.
On the one hand, 
it is proved that the Stokes equations subject to partial slip
in general are not well-posed in the standard setting that 
usually applies for Helmholtz domains.
On the other hand, it is proved that under certain reasonable 
assumptions the Stokes and Navier-Stokes 
equations subject to partial slip are well-posed in a generalized
setting. This setting relies on a generalized 
version of the Helmholtz decomposition which exists under suitable 
conditions on the intersection and the sum of 
gradient and solenoidal fields in $L_q$.
The proved well-posedness of the Stokes resolvent problem turns even 
out to be equivalent to the existence of the generalized Helmholtz
decomposition.
The presented approach, for instance, includes the sector-like
non-Helmholtz domains introduced by Bogovski\u{\i} and Maslennikova
as well as further wide classes of uniform $C^{2,1}$-domains.
\end{abstract}

\maketitle



\section{Introduction}

The question of well-posedness of Stokes and Navier-Stokes equations
on general classes of unbounded domains on $L_q$ has been an open 
problem for a couple of decades now. The aim of this note 
is to tackle this problem
for a wide class of uniform $C^{2,1}$-domains and of 
partial slip type boundary conditions.
To be precise, we consider the system
\begin{equation}
\left\{
\begin{array}{rll}
\partial_t u - \Delta u + \transB{\nabla p} + (u \cdot \nabla)u & = 0 & \text{in } (0,T) \times \Omega \\
\Div u & = 0 & \text{in } (0,T) \times \Omega \\
\Pi_\tau( \alpha u + \Dpm(u)\nu) & = 0 & \text{on } (0,T) \times \partial \Omega \\
\nu \cdot u & = 0 & \text{on } (0,T) \times \partial \Omega \\
u|_{t=0} & = u_0 & \text{in } \Omega
\end{array}
\right.
\end{equation}
as well as its linearized version (by skipping $(u\cdot \nabla)u$) 
known as the Stokes system.
Here $\Omega$ is a uniform $C^{2,1}$-domain and $u_0\in L_q(\Omega)^n$.
We set $\Dpm(u) := \transA{\nabla u} \pm \transB{\nabla u}$ 
and $\Pi_\tau$ denotes the projection onto the tangent space of
$\partial\Omega$. 
For the parameter $\alpha$ related to the 
slip length we assume $\alpha\in\mathbb R$, but we remark that 
a class of matrix-valued $\alpha$ is admitted too, see
Remark~\ref{matrixalpha}. So, our approach yields well-posedness
for Stokes and Navier-Stokes equations for a wide class
of partial slip type boundary conditions that includes, e.g.,
\begin{align}
	\Pi_\tau( \alpha u + \Dp(u)\nu)=0,\quad \nu\cdot u = 0
	&\quad\text{(Navier slip)},\nonumber\\
	\Pi_\tau \Dp(u)\nu=0,\quad \nu\cdot u = 0
	&\quad\text{(no stress)},\nonumber\\
	\Dm(u)\nu  = 0,\quad \nu \cdot u = 0
	&\quad\text{(perfect slip)},\label{eq:perfect slip}\\
	\alpha \Pi_\tau u + \partial_\nu \Pi_\tau u=0,
	\quad \nu\cdot u = 0
	&\quad\text{(Robin type)}.\nonumber
\end{align}
Note that in dimension $n=3$ we have 
$\Dm(u)\nu = - \nu \times \rot u$. Thus, in this case the perfect slip boundary 
conditions \eqref{eq:perfect slip} equal the vorticity condition
\[
	\nu \times \rot u = 0,\quad \nu \cdot u = 0. 
\]

By choosing $\alpha>0$ large, we
can come arbitrarily close to Dirichlet conditions, too. However,
pure Dirichlet conditions (formally the case $\alpha=\infty$) is
not covered, hence that case remains an open problem.   

There is a long history for Stokes and Navier-Stokes 
equations on standard domains
such as the whole space $\mathbb R^n$, half-space $\mathbb R^n_+$, 
perturbed half-spaces, and domains with compact boundary. 
We refrain from giving a long list of references and refer to the 
pertinent monographs \cite{Gal11,Soh01} and to the survey 
\cite{HS18}
instead. Standard domains $\Omega\subset\mathbb R^n$ of certain
regularity (e.g.\ $C^2$), as listed above, 
are known to be Helmholtz
domains, i.e., domains for which the Helmholtz decomposition
\begin{equation}\label{HHD}
	L_q(\Omega)^n = L_{q,\sigma}(\Omega)\oplus G_q(\Omega)
\end{equation}
into solenoidal fields $L_{q,\sigma}(\Omega)$ and gradient 
fields $G_q(\Omega)=\bigl\{\nabla p : p\in \widehat{W}^{1}_q(\Omega)\bigr\}$
exists (see, e.g., \cite{HS18}, Sec.~2.2).
The resulting Helmholtz projection $\Helmholtz$ onto $L_{q,\sigma}(\Omega)$ 
then serves as an important tool to define the Stokes operator
as
\[
 	A_S:= \Helmholtz \Delta\quad\text{in}\quad L_{q,\sigma}(\Omega). 	
\]

In the paper \cite{GHHS12} of Geissert, Heck, Hieber and Sawada
it is even proved that the validity of \eqref{HHD} is sufficient
for the well-posedness of the Stokes equations in $L_q$ subject to Dirichlet
boundary conditions, as long as $\Omega$ is a uniform $C^3$-Helmholtz domain. This triggers the question, whether the validity
of \eqref{HHD} is also necessary for well-posedness of the Stokes
equations. A negative answer to that question was given by Bolkart, 
Giga, Miura, Suzuki, and Tsutsui in \cite{BGMST17}. In that paper
well-posedness of the Stokes equations subject to Dirichlet 
conditions on $L_q(\Omega)^n$ for domains of the form
\[
	\Omega=\{x=(x',x_n)\in\mathbb R^n;\ x_n>h(x')\}
\]
with a $C^3$-function $h:\mathbb R^{n-1}\to \mathbb R$ is proved.
This includes sector-like domains in $\mathbb R^2$. In 
\cite{MB86}
Bogovski\u{\i} and Maslennikova proved those domains to 
be non-Helmholtz domains, i.e., \eqref{HHD} is false for
$q$ outside a certain interval about $2$. Consequently, 
\eqref{HHD} is not necessary for the well-posedness of the Stokes
equations. 

Another remarkable result in this context is given by Farwig, 
Kozono and Sohr in \cite{FKS07}. There it is proved that
the Stokes equations subject to Dirichlet conditions are well-posed on 
\[
	\widetilde L_q(\Omega)
	:=\left\{
	\begin{array}{lr}
	L_q(\Omega)\cap L_2(\Omega),& q\ge 2,\\
	L_q(\Omega)+ L_2(\Omega),& q< 2
	\end{array}
	\right.
\]
for general uniform $C^{1,1}$-domains $\Omega\subset\mathbb R^n$.  
The approach in \cite{FKS07} makes use of the fact that 
the Helmholtz decomposition exists
on $\widetilde L_q(\Omega)$ for all $q\in(1,\infty)$
and arbitrary uniform $C^{1,1}$-domains. The latter result is obtained
in the preceding paper \cite{FSSV16}. 
By Rosteck in \cite{farros2016} the results obtained in \cite{FKS07} 
are extended to Navier boundary conditions.
Note that the approach performed in \cite{FKS07,FSSV16,farros2016} 
utilizes in an essential way the fact that Helmholtz decomposition 
and well-posedness of the
Stokes equations are available on $L_2(\Omega)$ for 
arbitrary domains $\Omega$. 
This fact, however, does not help for an approach 
in $L_q(\Omega)$ with $q\neq 2$. As a consequence, the strategy 
performed in this note is in large part different from 
\cite{FKS07} (and also from \cite{BGMST17}).

According to \cite{GHHS12} on $C^3$-Helmholtz domains $\Omega$ 
the (Dirichlet) Stokes resolvent problem is well-posed on $L_q(\Omega)$
and the solution belongs to the class
\begin{equation}\label{rpwpc}
	(u,\nabla p)
	\in \left[W^2_q(\Omega)^n\cap L_{q,\sigma}(\Omega)\right]
	\times G_q(\Omega).
\end{equation}
The aim of this note is to clarify well-posedness on $L_q(\Omega)$
in the same regularity class for the solution, 
but for the general class of uniform $C^{2,1}$-domains  
(and for the class of partial slip type boundary conditions introduced
above).

Let us outline the outcome of our main results. 
The first parts state that the Stokes resolvent problem in general 
is not well-posed in the class given by \eqref{rpwpc}.
Here Theorem~\ref{thm:Stokes perfect slip}\eqref{thm:Stokes perfect slip 1},\eqref{thm:Stokes perfect slip 2} includes the case
of perfect slip
and Theorem~\ref{thm:Stokes partial slip}\eqref{thm:Stokes partial slip 2} in combination with
Remark~\ref{noninjectps} the case of partial slip type boundary conditions.
In fact, the results show that existence of a solution fails 
if $L_{q,\sigma}(\Omega)+G_q(\Omega) \ne L_q(\Omega)^n$ and uniqueness fails if 
$L_{q,\sigma}(\Omega)\cap G_q(\Omega)$ is nontrivial, in general.
As a consequence, for Bogovski\u{\i} and Maslennikova type 
sector-like domains in $\mathbb R^2$, e.g., existence 
fails for $1<q<2$ small enough 
and uniqueness for $2<q<\infty$ large enough.

In spite of this fact, the Stokes resolvent problem can be proved to be 
well-posed in $L_q$ 
in a certain generalized setting under the assumptions
that 
\begin{equation}\label{cruassumpdom}
	L_{q,\sigma}(\Omega)\cap G_q(\Omega)\text{ is complemented 
	in $L_q(\Omega)^n$ and } 
	L_{q,\sigma}(\Omega)+ G_q(\Omega)=L_q(\Omega)^n
\end{equation}
(see Assumptions~\ref{thm:Assumption A}\eqref{thm:Assumption A 1} and \ref{thm:Assumption B}). 
These assumptions imply that there exists a generalized
Helmholtz decomposition of the form
\begin{equation}\label{genHHD}
	L_q(\Omega)^n=L_{q,\sigma}(\Omega)\oplus \mathcal G_q(\Omega)
\end{equation}
(see Lemma~\ref{thm:complemented subspaces})
with $\mathcal G_q(\Omega):= (I-\mathbb Q_q)G_q(\Omega)$, where $\mathbb Q_q$
is the projection onto $L_{q,\sigma}(\Omega)\cap G_q(\Omega)$.
Theorem~\ref{thm:Stokes perfect slip}\eqref{thm:Stokes perfect slip 4} for perfect slip
and Theorem~\ref{thm:Stokes partial slip}\eqref{thm:Stokes partial slip 2} for partial slip
then yield well-posedness of the Stokes resolvent problem, 
provided the class in \eqref{rpwpc} is replaced by 
\begin{equation}\label{rpwpcgen}
	(u,\nabla p)
	\in \left[ W^2_q(\Omega)^n\cap L_{q,\sigma}(\Omega) \right]
	\times \mathcal G_q(\Omega).
\end{equation}
In fact, in combination with Theorem~\ref{thm:Stokes implies generalized HH}, these results
show that well-posedness of the Stokes resolvent problem as in
\eqref{rpwpcgen} is even equivalent to decomposition \eqref{genHHD}.
E.g., for Bogovski\u{\i} and Maslennikova type 
sector-like domains and $q$ large enough \eqref{genHHD} holds
(see Remark~\ref{remmainassum}\eqref{remmainassum f}).
Hence, the Stokes resolvent problem is well-posed for those domains
in the meaning of \eqref{rpwpcgen}. This in particular extends the result
in \cite{BGMST17} to partial slip type conditions.

\begin{remark}
Note that a fundamental assumption for the entire approach presented
here is the fact that
$C_c^\infty(\overline{\Omega}) \subset \widehat{W}^1_{q'}(\Omega)$ is dense
(Assumption~\ref{thm:Assumption C}). This, for instance, is not
fulfilled for aperture domains. As a consequence such type of domains
are not included in the presented approach. If, however, the difference
of the closure of $C_c^\infty(\overline{\Omega})$ in 
$\widehat{W}^1_{q'}(\Omega)$ and $\widehat{W}^1_{q'}(\Omega)$ itself is not 
too big (e.g.\ one dimensional as for aperture domains), there might 
be ways to generalize the approach in order to include such classes
(see Remark~\ref{remmainassum}\eqref{remmainassum c}).
\end{remark}

The well-posedness of the Stokes resolvent problem given by
Theorem~\ref{thm:Stokes partial slip}\eqref{thm:Stokes partial slip 2}
and 
Theorem~\ref{thm:Stokes perfect slip}\eqref{thm:Stokes perfect slip 3},\eqref{thm:Stokes perfect slip 4} also implies well-posedness
of the related instationary Stokes system. This is the content
of Theorem~\ref{thm:Stokes semigroup}. In particular, the corresponding
generalized Stokes operator is sectorial and the generated 
strongly continuous analytic semigroup satisfies the typical $L_p$-$L_q$-estimates.
Having these tools at hand, Theorem~\ref{thm:NSE} establishes
existence of a local-in-time mild solution of the corresponding
Navier-Stokes equations.

Next, we sketch the strategy for the proofs and 
the organization of this note.
The initiating point is to establish decomposition \eqref{genHHD}
under assumptions \eqref{cruassumpdom}.  
This is given by Lemma~\ref{thm:complemented subspaces},
which is not profound but crucial, since it opens the door for
the treatment of the Stokes equations in subsequent sections.
In Sections~\ref{sec:Main results} (resolvent problem) 
and \ref{sec:instationary} 
(instationary Stokes and Navier-Stokes systems)
we give precise statements of the 
main results of this note. 
After collecting some preliminary tools on trace operators,
solenoidal fields, and coverings of $\partial\Omega$
in Section~\ref{sec:preliminaryproofs}, 
we prove well-posedness of the (vector-valued) heat equation subject
to perfect slip in $L_q$ in Section~\ref{sec:heat}.
This technical part is performed by utilizing a suitable 
localization procedure. A substantial difficulty here is given by the fact
that due to the boundary conditions one has to deal with a system.
In fact, a cautious handling of tangent and normal
trace parts is required. This is different from previous 
literature in which
the applied localization procedure is predominantly applied to
scalar equations, see \cite{Kun03}.

Another crucial step is represented by Section~\ref{invariance}.
There we establish that  
$L_{q,\sigma}(\Omega)$ is an invariant space for the resolvent 
of the Laplace operator subject to perfect slip.
For Helmholtz domains this is (formally) equivalent to the 
fact that Helmholtz projection and Laplace operator commute. 
In the latter form this specific feature of perfect slip 
boundary conditions is already utilized in a number of 
former papers, such as \cite{mitmon2009,wesa2018,baam2019}. 
Note that the fact that
here the Helmholtz decomposition in general does not exist makes
the proof of the invariance a bit more delicate.
Based on the invariance, the main results on the Stokes resolvent
problem subject to perfect slip are then given in
Section~\ref{sec:proofsofstokes}. 
The generalization of this result to partial slip type conditions
relies on a perturbation argument.
For this purpose, the well-posedness of the Stokes resolvent problem
with tangential inhomogeneous perfect slip conditions is required.
Note, that in that case 
the space $L_{q,\sigma}(\Omega)$ is no longer invariant for the 
solution operator to the corresponding inhomogeneous heat equation.
The idea is to compensate this discrepancy by constructing
and adding a suitable pressure gradient depending only on the data, 
see the proof of Theorem~\ref{thm:Stokes perfect slip inhom} in 
Section~\ref{sec:proofsofstokes}. 
Based on  Theorem~\ref{thm:Stokes perfect slip inhom} and a 
perturbation argument, the proof of the well-posedness of the 
Stokes resolvent problem for a large class of  
partial slip type conditions is then given in Section~\ref{sec:proofpartial}. 

The proof of the main result on the instationary Stokes system 
is performed in Section~\ref{sec:instat} and of the local-in-time 
well-posedness of the Navier-Stokes equations in 
Section~\ref{sec:ns}.
Finally, the Appendix represents a collection of basic facts 
that are hard to find in the existing literature. This concerns 
certain trace operators and the Gau{\ss} theorem on uniform
$C^{2,1}$-domains, the density of 
$C_c^\infty(\overline{\Omega})$ in $\widehat{W}^1_{q'}(\Omega)$
for specific classes of domains, etc.

\section{Basic assumptions and notation}\label{sec:basics}

For parameters $a,b,c,\dots$ we write $C = C(a,b,c,\dots)$ to express 
that $C$ is a constant depending on (and only on) these parameters.
In general, $C,C',C'',\dots$ are positive constants that may change from line to line.
(We primarily denote constants by $C$ and make use of $C',C'',\dots$
when it is relevant to indicate that the constant has changed.)
For any normed space $X$ the related dual space is denoted by $X'$ and
the duality pairing is denoted by $\dual{\cdot}{\cdot}_{X,X'}$.
For a linear continuous operator $T: X \rightarrow Y$ and two normed spaces $X,Y$ we write $\Range(T)$ for its range and $\Kernel(T)$ for its kernel as well as $\| T \|_{X \rightarrow Y}$ for the operator norm.
The Lebesgue measure in $\mathbb{R}^n$ is $\lambda_n$ and $\sigma$ denotes the related surface measure.
The natural numbers $\mathbb{N}$ do not contain zero and we put $\mathbb{N}_0 := \mathbb{N} \cup \{ 0 \}$.
We denote the Euclidean norm on $\mathbb{R}^n$ or $\mathbb{R}^{n \times n}$ by $| \cdot |$.
The ball in $\mathbb{R}^n$ with respect to the Euclidean norm with radius $r > 0$ and center $a \in \mathbb{R}^n$ is denoted by $B_r(a)$.
The sector in the complex plane with opening angle $0 < \theta < \pi$ is
$\Sigma_\theta := \{ \lambda \in \mathbb{C} : \lambda \ne 0,~ |\arg(\lambda)| < \theta \}$.

For $x \in \mathbb{R}^n$ we denote the components by $x_j$, $j = 1,\dots,n$ and we write $x'$ for the vector of the first $n-1$ components.
We denote the components of a vector field $u$ in $\mathbb{R}^n$ by $u^{j}$, so $u = (u^{1},\dots,u^{n})^T$.
The identity matrix is $I := (\delta_{ij})_{i,j=1,\dots,n} \in \mathbb{R}^{n \times n}$.
We also denote the identity map between normed vector spaces by $I$.
The transposed of some vector or matrix $v$ is $v^T$.

By the gradient of a function $u: \Omega \rightarrow \mathbb{R}$ we mean the column vector $\nabla u = (\partial_1 u, \dots, \partial_n u)^T$ and by the gradient $\nabla u$ of a vector field $u: \Omega \rightarrow \mathbb{R}^m$ we mean the matrix with columns $\nabla u^{j}$ for $j=1,\dots,m$, i.e., $\nabla u^T$ is the Jacobian matrix of $u$.
The vector containing all partial derivatives of order $k \ge 2$ of a real-valued function $u$ is $\nabla^k u$ (with $n^k$ entries) and similarly we define $\nabla^k u$ (with $m n^k$ entries) if $u$ is a vector field with values in $\mathbb{R}^m$.

For the representation of boundary conditions we make use of the two operators
$\Dpm(u) := \transA{\nabla u} \pm \transB{\nabla u}$ when $u: \Omega \rightarrow \mathbb{R}^n$ is a vector field, as well as of the
normal and tangential projections of $u$ on $\partial \Omega$, given by
$\Pi_\nu u = (\nu \nu^T)u$ and
$\Pi_\tau u = (I - \nu \nu^T)u$
respectively.
Here $\nu: \partial \Omega \rightarrow \mathbb{R}^n$ denotes the outward unit normal vector at $\partial \Omega$ if the boundary is sufficiently regular.
Writing the normal projection in the scalar product form, i.e.,  $\Pi_\nu u = (\nu \cdot u) \nu$, we see that
\begin{equation*}
\Pi_\nu u = 0 \text{ on } \partial \Omega
\quad \Leftrightarrow \quad
\nu \cdot u = 0 \text{ on } \partial \Omega
\end{equation*}
and in dimension $n = 3$ we can use the vector product to write
$\Pi_\tau u = - \nu \times (\nu \times u)$
and consequently
\begin{equation*}
\Pi_\tau u = 0 \text{ on } \partial \Omega
\quad \Leftrightarrow \quad
\nu \times u = 0 \text{ on } \partial \Omega.
\end{equation*}
Also note that
\begin{equation} \label{eq:Dm invariant under tangential projection}
\Pi_\tau \Dm(u) \nu = \Dm(u) \nu \quad \text{on } \partial \Omega.
\end{equation}

For a function $\omega$ on $\mathbb{R}^{n-1}$ its gradient with respect
to the $n-1$ components is $\nabla' \omega$ and similarly we use the
notation $\nabla'^k \omega$ for higher derivatives and $k \in \mathbb{N}$. 
Analogously we write
$\Delta'$ for the Laplace operator with respect to the first $n-1$ components.

As usual, $C^k(\Omega)$ is the space of $k$-times continuously differentiable functions on $\Omega \subset \mathbb{R}^n$ for $k \in \mathbb{N}_0$ and $C^{k,1}(\Omega)$ is the subspace of functions with a Lipschitz continuous $k$-th derivative.

All along the paper, we assume $\Omega \subset \mathbb{R}^n$ to be a domain with uniform $C^{2,1}$-boundary, $n \ge 2$ and $1 < q < \infty$ if nothing else is declared. The dual exponent is $q'$, i.e., $1 < q' < \infty$ with $\frac{1}{q} + \frac{1}{q'} = 1$.
By a $C^{2,1}$-boundary we mean, that
we can cover $\overline{\Omega}$ with open balls $B_l$, $l \in \Gamma$ and a countable index set $\Gamma$ such that, writing $\Gamma_0 := \{ l \in \Gamma : B_l \subset \Omega \}$ and $\Gamma_1 := \{ l \in \Gamma : B_l \cap \partial \Omega \ne \emptyset \}$, for each $l \in \Gamma_1$ we can find a compactly supported function $\omega_l \in C^{2,1}(\mathbb{R}^{n-1})$ which describes the boundary locally in $B_l$ after rotating and shifting the coordinates.
The latter precisely means that for $l \in \Gamma_1$ we can find a rotation matrix $Q_l \in \mathbb{R}^{n \times n}$ and a translation vector $\tau_l \in \mathbb{R}^n$ so that
\begin{equation*}
\Omega \cap B_l = H_l \cap B_l \quad \text{and} \quad \partial \Omega \cap B_l = \partial H_l \cap B_l,
\end{equation*}
where $H_l := Q_l^T H_{\omega_l} + \tau_l$ is the rotation and translation of the bent half space
\begin{equation*}
H_{\omega_l} = \{ x \in \mathbb{R}^n : x_n > \omega_l(x_1,\dots,x_{n-1}) \}.
\end{equation*}

\begin{definition} \label{thm:uniform C^2,1 domain}
A domain $\Omega \subset \mathbb{R}^n$ has uniform $C^{2,1}$-boundary (or $\Omega$ is a uniform $C^{2,1}$-domain) if we can choose the cover $B_l$, $l \in \Gamma$ in such a way that the radii
are all bigger or equal to some fixed $\rho > 0$ and if there is a constant $M \ge 1$ such that
\begin{equation} \label{eq:uniformity}
\| \nabla' \omega_l \|_{\infty}, \| \nabla'^2 \omega_l \|_{\infty}, \| \nabla'^3 \omega_l \|_{\infty} \le M
\end{equation}
for all $l \in \Gamma_1$.
\end{definition}

In the situation of Definition~\ref{thm:uniform C^2,1 domain}, without loss of generality, we can assume that all of the balls $B_l$, $l \in \Gamma$ have the same radius $\rho > 0$
and that there is $\bar{N} \in \mathbb{N}$ so that at most $\bar{N}$ of the balls $B_l$ have nonempty intersection.
Moreover, for arbitrary $\kappa > 0$ we can assume that
\begin{equation} \label{eq:littleness gradient}
\| \nabla' \omega_l \|_{\infty} \le \kappa
\end{equation}
holds for all $l \in \Gamma_1$.
This can be achieved by choosing the radius $\rho$ small enough and the
rotations $Q_l$ in such a way that the hyper plane $\{ x_n = 0 \}$ is
rotated into the tangent hyper plane of some 
point on $\partial \Omega \cap B_l$.

For two indices $l,m \in \Gamma$ we write $m \sim l$ if $B_m \cap B_l \ne \emptyset$ and we write $m \approx l$ if $m \sim l$ and $l,m \in \Gamma_1$.
Note that for any $l \in \Gamma$ we have $\# \{ m \sim l \} \le \bar{N}$.

The Lebesgue space is denoted by $L_q(\Omega)$, 
the Sobolev space for some $k \in \mathbb{N}_0$ is $W^k_q(\Omega)$
and the Lebesgue space on the boundary is $L_q(\partial \Omega)$.
We write $\| \cdot \|_{q,\Omega}$ resp.\ $\| \cdot \|_{q,k,\Omega}$
for  the corresponding norms. Frequently we also write
$\| \cdot \|_q$ for the Lebesgue norm and $\| \cdot \|_{k,q}$ 
for the Sobolev norm, in case the underlying domain $\Omega$ 
is clear from the context.

We use the notation
$\dualq{f}{g} := \int_\Omega f g ~d\lambda_n$ for $f \in L_q(\Omega)$, $g \in L_{q'}(\Omega)$ resp.\
$\dualq{f}{g} := \int_\Omega f \cdot g ~d\lambda_n$ for $f \in L_q(\Omega)^n$, $g \in L_{q'}(\Omega)^n$.
Here the standard scalar products in $\mathbb{R}^n$ and $\mathbb{R}^{n \times n}$ are denoted by $x \cdot y := \sum_{i=1}^n x_i y_i$ and $A : B := \sum_{i,j=1}^n A_{ij} B_{ij}$, respectively.
For the application of a distribution $f \in \mathscr{D}'(\Omega)$ to a test function $\varphi \in C_c^\infty(\Omega)$ we write $\dual{f}{\varphi}$,
in particular
$\dual{f}{\varphi} = \int_\Omega f \varphi \D \lambda_n$ in case $f \in
L_{1,\mathrm{loc}}(\Omega)$ (similarly for $f \in \mathscr{D}'(\Omega)^n$ and $\varphi \in C_c^\infty(\Omega)^n$).

The space of smooth functions with compact support in $\Omega$ is
$C_c^\infty(\Omega)$ and $C_{c,\sigma}^\infty(\Omega)$ is the subspace
of vector fields $u\in C_c^\infty(\Omega)^n$ with vanishing divergence,
i.e., $\Div u=0$.

The Sobolev-Slobodecki\u{\i} space $W^s_q(\Omega)$ for $s = k + \lambda$, $k \in \mathbb{N}_0$, $0 < \lambda < 1$ can be defined as the space of functions $u \in W^k_q(\Omega)$ such that
\begin{equation*}
\| u \|_{W^s_q(\Omega)}
:= \| u \|_{W^k_q(\Omega)}
+ \sum_{|\alpha| = k} \left( \int_\Omega \int_\Omega \frac{|\partial_\alpha u(y) - \partial_\alpha u(x)|^q}{|y - x|^{n + \lambda q}} \D y \D x \right)^\frac{1}{q}
\end{equation*}
is finite (cf.~\cite{Mar87}).
We will further need Sobolev-Slobodecki\u{\i} spaces on the boundary
$W^s_q(\partial \Omega)$ for $s = 1 - \frac{1}{q}$, constituted by the image of the trace operator
\begin{equation*}
\Tr: W^1_q(\Omega) \rightarrow L_q(\partial \Omega), \quad
\Tr u = u|_{\partial \Omega} ~\forall u \in C_c^\infty(\overline{\Omega}).
\end{equation*}
For a treatment of the trace operator and a concrete definition of Sobolev-Slobodecki\u{\i} spaces on the boundary we refer to~\cite{Mar87} (note that the Besov scale $B^s_q(\partial \Omega)$ from~\cite{Mar87} coincides with the Sobolev-Slobodecki\u{\i} scale, since in our considerations $s$ never is an integer, except $s = 0$).
See also~\cite{Tri78}, Thm.\ 4.7.1 for the special case of bounded smooth domains.

The most important subspaces of $L_q(\Omega)^n$ for a treatment of the Stokes equations are
the Lebesgue space of solenoidal functions, defined as 
\[ 
	L_{q,\sigma}(\Omega)
	:=\overline{C_{c,\sigma}^\infty(\Omega)}^{L_q(\Omega)^n}
\]
and the space of gradient fields
\[
	\Gradq := \{ \transB{\nabla p} : p \in \widehat{W}^1_q(\Omega) \},
\]
where
$\widehat{W}^1_q(\Omega) = \{ p \in L_{q,\mathrm{loc}}(\Omega) : \transB{\nabla p} \in L_q(\Omega)^n \}$
is the homogeneous Sobolev space, endowed with the seminorm
$|p|_{\widehat{W}^1_q(\Omega)} = \| \nabla p \|_q$.
As usual, for some domain $\Omega$ and some $1 < q < \infty$, we say that the Helmholtz decomposition exists if the direct decomposition
\begin{equation}\label{standardhhd}
L_q(\Omega)^n = L_{q,\sigma}(\Omega) \oplus \Gradq
\end{equation}
holds.

As explained in the introduction, a specific feature of our approach is
that it covers classes of non-Helmholtz domains, i.e., domains $\Omega$
(and values of $q$) for which decomposition (\ref{standardhhd}) is false. 
Instead of (\ref{standardhhd}) we consider the following 
weaker assumptions on
the domain $\Omega$ and the parameter $q$.
Here 
\begin{equation}\label{defuintsp}
	U_q(\Omega) := L_{q,\sigma}(\Omega) \cap \Gradq
\end{equation}
denotes the intersection. Note that $U_q(\Omega)$ is closed in
$L_q(\Omega)^n$ by the closedness of $L_{q,\sigma}(\Omega)$ and $\Gradq$
and that $U_q(\Omega)=\{0\}$ for Helmholtz domains $\Omega$.

\begin{assumption} \label{thm:Assumption A}
~
\begin{enumerate}[\upshape (i)]
\item \label{thm:Assumption A 1} $U_q(\Omega)$ is a complemented subspace of $L_q(\Omega)^n$.
\item \label{thm:Assumption A 2} $L_{q,\sigma}(\Omega) + \Gradq$ is a closed subspace of $L_q(\Omega)^n$.
\end{enumerate}
\end{assumption}

\begin{assumption} \label{thm:Assumption B}
$L_q(\Omega)^n = L_{q,\sigma}(\Omega) + \Gradq$.
\end{assumption}

\begin{assumption} \label{thm:Assumption C}
$C_c^\infty(\overline{\Omega}) \subset \widehat{W}^1_{q'}(\Omega)$ is dense.
\end{assumption}

 We will see that
the strength of the results derived for the Stokes equations depends
on how many of the above assumptions are fulfilled.
In case Assumption~\ref{thm:Assumption A}\eqref{thm:Assumption A 1} is valid, we denote the continuous linear projection onto $U_q(\Omega)$ by
\begin{equation}\label{projontoinsec}
	\QProj_q: L_q(\Omega)^n \rightarrow L_q(\Omega)^n.
\end{equation}
The significant assertions on topological decompositions 
in the following lemma 
are more or less direct consequences of 
Assumption~\ref{thm:Assumption A}\eqref{thm:Assumption A 1}. However,
as we will see, in a certain sense it represents the key to the rigorous 
treatment of the Stokes equations on uniform $C^{2,1}$-domains
presented in this note.

\begin{lemma} \label{thm:complemented subspaces}
For a Banach space $E$ and subspaces $E_1,E_2 \subset E$ we set 
$U:= E_1 \cap E_2$.
If $Q:E\to E$ denotes a (algebraic and hence not necessarily bounded) 
projection onto $U$ we set
$\widetilde{E}_1 := (I-Q)E_1$ and $\widetilde{E}_2 := (I-Q)E_2$.
Then the following holds true:
\begin{enumerate}[\upshape (i)]
\item \label{thm:complemented subspaces 2} 
We have 
\[
	E_j = \widetilde{E}_j \oplus U,\quad j=1,2, 
\]
as algebraic decompositions. 
Furthermore, if $Q:E\to E$ is bounded and $E_j$ is closed in $E$, then 
$\widetilde{E}_j$ is closed in $E$ as well and the two
decompositions are topological ones.
\item \label{thm:complemented subspaces 3} 
We have 
\[
	E_1+E_2 = \widetilde{E}_1 \oplus E_2 
	= E_1 \oplus \widetilde{E}_2 =
	\widetilde{E}_1 \oplus \widetilde{E}_2 \oplus U
\]
as algebraic decompositions.
In addition, if $Q:E \to E$ is bounded and 
$E_1$, $E_2$ and $E_1+E_2$ are closed in $E$ then 
there exist bounded projections $\mathcal Q_j:E_1+E_2\to E_1+E_2$ 
such that $\mathcal Q_j(E_1+E_2)=\widetilde{E}_j$ for $j=1,2$. 
Hence, in that case all the decompositions above are topological ones.
\end{enumerate}
\end{lemma}

\begin{proof}

Note that an algebraic projection onto $U$ always exists by the basis
extension theorem.

In order to see \eqref{thm:complemented subspaces 2}, note that the
definition of $Q$ yields $U = Q E_1 = Q E_2$. Therefore, $E_j = Q E_j
\oplus (I - Q) E_j = U \oplus \widetilde{E}_j$ holds for $j = 1,2$. 
If $Q:E\to E$ additionally is bounded, from this we infer 
that $I-Q:E_j\to E_j$ is a bounded projection 
onto $\widetilde{E}_j$ (with respect to the norm in $E$). 
By this fact, assuming $E_j$, $j=1,2$, to be closed in $E$ yields
that $\widetilde{E}_j$ is closed in $E$, too.

For \eqref{thm:complemented subspaces 3} observe that 
$\widetilde{E}_1 \cap \widetilde{E}_2 \subset E_1 \cap E_2 = \Range(Q)$
and, on the other hand, that $\widetilde{E}_1 \cap \widetilde{E}_2 \subset
\widetilde{E}_1 = (I-Q)E_1 \subset (I-Q)E = \Kernel(Q)$. Thus
$\widetilde{E}_1 \cap \widetilde{E}_2 = \{ 0 \}$. Now, for
$x = x_1 + x_2$ with $x_1 \in E_1$ and $x_2 \in E_2$ we obtain
\[
	x = (I-Q)x_1 + (I-Q)x_2 + Q(x_1 + x_2) 
	\in \widetilde{E}_1 + \widetilde{E}_2 + U.
\]
Consequently, $E_1+E_2 = \widetilde{E}_1 \oplus \widetilde{E}_2 \oplus U$. The
remaining equalities in \eqref{thm:complemented subspaces 3} are
consequences of \eqref{thm:complemented subspaces 2}. If in addition 
$Q:E\to E$ is bounded and 
$E_1$, $E_2$, $E_1+E_2$ are closed in $E$, then due to 
\eqref{thm:complemented subspaces 2} the spaces $\widetilde{E}_2$,
$j=1,2$, are closed in $E$ as well. This yields the assertion.
%
%
\end{proof}

Utilizing the projection $\QProj_q$ from \eqref{projontoinsec}, we can define a smaller space of gradient fields as the subspace $\modGradq := (I - \QProj_q) \Gradq$.
If Assumptions~\ref{thm:Assumption A}\eqref{thm:Assumption A 1} and
\ref{thm:Assumption B} are both valid, thanks to 
Lemma~\ref{thm:complemented subspaces}\eqref{thm:complemented subspaces 3} we have the decomposition
\begin{equation} \label{eq:decomposition modGrad a}
L_q(\Omega)^n = L_{q,\sigma}(\Omega) ~ \oplus ~ \modGradq.
\end{equation}
We denote the related continuous linear projection onto $L_{q,\sigma}(\Omega)$ by $\Proj = \Proj_q$.
If only Assumption~\ref{thm:Assumption A}\eqref{thm:Assumption A 1} is valid, 
Lemma~\ref{thm:complemented subspaces}\eqref{thm:complemented subspaces 3} still yields
\begin{equation} \label{eq:decomposition modGrad b}
L_{q,\sigma}(\Omega) + \Gradq = L_{q,\sigma}(\Omega) ~\oplus~ \modGradq.
\end{equation}
Note, however, that the direct decomposition \eqref{eq:decomposition
modGrad b} may not be a topological one by the fact that 
$L_{q,\sigma}(\Omega) + \Gradq$ might not be closed in $L_q(\Omega)^n$, in general. 
This can only be guaranteed if additionally Assumption
\ref{thm:Assumption A}\eqref{thm:Assumption A 2} is fulfilled.

Decomposition \eqref{eq:decomposition modGrad a} may be regarded 
as a {\em generalized Helmholtz decomposition}. 
We gather some remarks on our main Assumptions:
\begin{remark} \label{remmainassum}
\begin{enumerate}[\upshape (a)]
\item \label{remmainassum a} Obviously any Helmholtz domain in the classical sense fulfills
Assumptions~\ref{thm:Assumption A} and~\ref{thm:Assumption B} with
$U_q(\Omega) = \{ 0 \}$.
\item \label{remmainassum b} For domains $\Omega \subset \mathbb{R}^n$ with uniform $C^{2,1}$-boundary, Assumption~\ref{thm:Assumption C} is known to be valid for
\begin{itemize}
\item $\Omega = \mathbb{R}^n$, $\Omega = \mathbb{R}^n_+$ and perturbed half spaces, i.e., there exists some $R > 0$ such that $\Omega \setminus B_R(0) = \mathbb{R}^n_+ \setminus B_R(0)$ (Lemma~\ref{thm:perturbed cones}; cf.~\cite{Gal11}, Thm.\ II.7.8 for the half space),
\item bent half spaces $\Omega = H_\omega$ (see~\cite{FS94}, Lem.\ 5.1; alternatively one could check that bent half spaces are $(\epsilon,\infty)$-domains, see the definition in Lemma~\ref{thm:(epsilon,infty)-domains}),
\item bounded domains (Lemma~\ref{thm:perturbed cones}: choose $R > 0$ such that $\Omega \subset B_R(0)$; cf.~\cite{Gal11}, Thm.\ II.7.2, Def.\ II.1.1),
\item exterior domains, i.e., $\Omega$ is the complement of some compact set in $\mathbb{R}^n$ (Lemma~\ref{thm:perturbed cones}: choose $R > 0$ such that $\Omega \setminus B_R(0) = \mathbb{R}^n \setminus B_R(0)$; cf.~\cite{FS94}, Lem.\ 5.1 and~\cite{Gal11}, Thm.\ II.7.2, Def.\ II.1.1),
\item asymptotically flat domains, i.e., $\Omega$ is a layer-like domain
$\Omega = \{ x \in \mathbb{R}^n : \gamma_-(x') < x_n < \gamma_+(x') \}$
which is delimited by two functions $\gamma_+, \gamma_- \in C^{2,1}(\mathbb{R}^{n-1})$ satisfying the asymptotic behavior
$\gamma_{\pm}(x') \to c_\pm$ for $|x'| \to \infty$, where $c_- < c_+$ and $\nabla \gamma_{\pm}(x') \to 0$ for $|x'| \to \infty$ (see~\cite{DissA03}, Lem.\ 2.6, Cor.\ 6.4),
\item $(\epsilon,\infty)$-domains, as considered in~\cite{Chu92} and~\cite{Jon81} (Lemma~\ref{thm:(epsilon,infty)-domains}), and
\item perturbed cones 
(see Definition~\ref{perturbed cones} and 
Lemma~\ref{thm:perturbed cones}).
\end{itemize}
\item \label{remmainassum c} Assumption~\ref{thm:Assumption C} is crucial
for all the results on the Stokes equations derived in this paper.
In fact, already for the key statement concerning our main results, Lemma~\ref{thm:resolvent},  a proof without this condition seems hopeless.
Note that the identity
\begin{equation} \label{eq:characterization L_q,sigma 0}
L_{q,\sigma}(\Omega)
= \{ f \in L_q(\Omega)^n : \Div f = 0, ~\nu \cdot f = 0 \text{ on } \partial \Omega \}
\end{equation}
is a consequence of Assumption~\ref{thm:Assumption C} (see Lemma~\ref{thm:characterization L_q,sigma}).
An aperture domain, as considered in~\cite{FS96} and~\cite{Far96}, is an example of a domain for which Assumption~\ref{thm:Assumption C} does not hold for all $1 < q < \infty$. The identity \eqref{eq:characterization L_q,sigma 0} is not satisfied in this case as well (see Remark~\ref{thm:aperture}).
An approach to circumvent this problem and to include also domains not
satisfying Assumption~\ref{thm:Assumption C} might be to define the
space $\Gradq$ by replacing $\widehat{W}^1_{q'}(\Omega)$ by the closure
of $C_c^\infty(\overline{\Omega})^n$ in $\widehat{W}^1_{q'}(\Omega)$ in
its definition.
Then Lemma~\ref{thm:resolvent} had to be proved for a larger space
$\widebar{L}_{q,\sigma}(\Omega)$ which seems to be possible for 
aperture domains. However, without Assumption~\ref{thm:Assumption C}
it seems unclear if Lemma~\ref{thm:resolvent} holds in general.
For this reason we stick to Assumption~\ref{thm:Assumption C} in this
paper.
\item \label{remmainassum d} Observe that for $1 < q \le 2$ we have $U_q(\Omega) =
L_{q,\sigma}(\Omega) \cap \Gradq = \{ 0 \}$ for any uniform
$C^{2,1}$-domain $\Omega$. Consequently, Assumption~\ref{thm:Assumption
A}\eqref{thm:Assumption A 1} is fulfilled and we have $\modGradq = \Gradq$.
This is due to~\cite{FKS07}, Thm.~1.2 (see also~\cite{FKS05}, Thm.~2.1
for the 3-dimensional case), from which we obtain the direct decomposition
\begin{equation*}
L_q(\Omega)^n + L_2(\Omega)^n
= [L_{q,\sigma}(\Omega) + L_{2,\sigma}(\Omega)] \oplus [\Gradq + \Grad{2}]
\end{equation*}
and therefore $L_{q,\sigma}(\Omega) \cap \Gradq
\subset [L_{q,\sigma}(\Omega) + L_{2,\sigma}(\Omega)] \cap [\Gradq + \Grad{2}] = \{ 0 \}$.
\item \label{remmainassum e} Obviously, in case $U_q(\Omega)$ has finite dimension, Assumption~\ref{thm:Assumption A}\eqref{thm:Assumption A 1} is valid and, in case $L_{q,\sigma}(\Omega) + G_q(\Omega)$ has finite codimension, Assumption~\ref{thm:Assumption A}\eqref{thm:Assumption A 2} is valid.
In this regard, we refer to~\cite{FSSV16} for an 
approach to generalized Helmholtz
decompositions of that type.
\item \label{remmainassum f} A sector-like domain with opening angle $\beta > \pi$ and a
smoothed vertex, as considered by Bogovski\u{\i} and Maslennikova
(see~\cite{MB86}), is an example of a non-Helmholtz domain (for $q$
either small or large enough). To that sort of domains our main theorems (in the subsequent Section~\ref{sec:Main results}) apply:
Lemma~\ref{thm:perturbed cones} gives that Assumption~\ref{thm:Assumption C} is valid for sector-like domains.
For these domains Assumptions~\ref{thm:Assumption A} and~\ref{thm:Assumption B} are valid if $q > \frac{2}{1 - \pi / \beta}$. We have $\dim U_q(\Omega) = 1$ in this case.
If $\frac{2}{1 + \pi / \beta} < q < \frac{2}{1 - \pi / \beta}$, Assumptions~\ref{thm:Assumption A} and~\ref{thm:Assumption B} hold and we have $\modGradq = \Gradq$.
If $q < \frac{2}{1 + \pi / \beta}$, Assumption~\ref{thm:Assumption A} holds, but~\ref{thm:Assumption B} does not. We have $\codim (L_{q,\sigma}(\Omega) + \Gradq) = 1$ in this case.
In the special cases $q = \frac{2}{1 \pm \pi / \beta}$, Assumption~\ref{thm:Assumption A}\eqref{thm:Assumption A 1} is still valid, but~\ref{thm:Assumption A}\eqref{thm:Assumption A 2} is not.
Hence, Theorem~\ref{thm:Stokes perfect slip} is applicable to
sector-like domains for any $q \in (1,\infty) \setminus \{ \frac{2}{1
\pm \pi / \beta} \}$ (merely the assertion \eqref{thm:Stokes perfect
slip 4} in Theorem~\ref{thm:Stokes perfect slip} does not apply to
the cases $q = \frac{2}{1 \pm \pi / \beta}$).
Theorems~\ref{thm:Stokes perfect slip inhom} and~\ref{thm:Stokes partial slip} are applicable for
$q \in (\frac{2}{1 + \pi / \beta},\infty) \setminus \{ \frac{2}{1 - \pi / \beta} \}$ and
Theorem~\ref{thm:NSE} is applicable for
$q \in (\frac{4}{1 + \pi / \beta},\infty) \setminus \{ \frac{2}{1 - \pi / \beta},\frac{4}{1 - \pi / \beta} \}$.
\item \label{remmainassum g}
The sector-like domains discussed in \eqref{remmainassum f} are examples of non-Helmholtz domains covered by
the approach presented in this note. A more general class of non-Helmholtz domains that is covered too,
are perturbed sector-like domains, or even more general perturbed cones, see Definition~\ref{perturbed cones} and Lemma~\ref{thm:perturbed cones}.
To the best of the authors knowledge these are up to now the only non-Helmholtz domains known. If $G$ denotes a perturbed sector-like non-Helmholtz domain, we think that
$G \times \mathbb{R}^k$ remains a ($(k+2)$-dimensional) non-Helmholtz domain. Assuming periodicity in the $\mathbb{R}^k$ directions could even lead to the case that the intersection of solenoidal and gradient fields has infinite dimension.
To include such cases the approach given here had to be extended to domains of the form $\Omega \times \mathbb{R}^k$ with periodicity in $\mathbb{R}^k$ direction. This, however, is not subject of the underlying note. 
\end{enumerate}
\end{remark}

\section{Main results for the resolvent problems} \label{sec:Main results}

Let $\Omega \subset \mathbb{R}^n$ be a domain with uniform $C^{2,1}$-boundary, $n \ge 2$ and $1 < q < \infty$.
Our first main result concerns perfect slip boundary conditions.

\begin{theorem} \label{thm:Stokes perfect slip}
Let $0 < \theta < \pi$, $U_q$ be given by \eqref{defuintsp}
and let Assumption~\thref{thm:Assumption C} be valid.
Then there exist $\lambda_0 = \lambda_0(n,q,\theta,\Omega) > 0$ and $C = C(n,q,\theta,\Omega) > 0$ such that for $\lambda \in \Sigma_\theta$, $|\lambda| \ge \lambda_0$ we have the following, concerning
\begin{equation} \label{eq:Stokes perfect slip}
\left\{
\begin{array}{rll}
\lambda u - \Delta u + \transB{\nabla p} & = f & \text{in } \Omega \\
\Div u & = 0 & \text{in } \Omega \\
\Dm(u)\nu & = 0 & \text{on } \partial \Omega \\
\nu \cdot u & = 0 & \text{on } \partial \Omega.
\end{array}
\right.
\end{equation}
\begin{enumerate}[\upshape (i)]
\item \label{thm:Stokes perfect slip 1}
Provided that $f \in L_q(\Omega)^n$,
problem \eqref{eq:Stokes perfect slip} has a solution
\begin{equation*}
(u,\transB{\nabla p}) \in [W^2_q(\Omega)^n \cap L_{q,\sigma}(\Omega)] \times \Gradq
\end{equation*}
if and only if $f \in L_{q,\sigma}(\Omega) + \Gradq$.
In particular, there exists a solution of \eqref{eq:Stokes perfect slip} for any $f \in L_q(\Omega)^n$ in case Assumption~\thref{thm:Assumption B} is valid.
\item \label{thm:Stokes perfect slip 2} The solution space $S_\mathrm{hom} \subset [W^2_q(\Omega)^n \cap L_{q,\sigma}(\Omega)] \times \Gradq$ of the homogeneous problem \eqref{eq:Stokes perfect slip} (i.e., $f = 0$) is
\begin{equation*}
S_\mathrm{hom} = \Big\{ \big( (\lambda - \LaplacePSq)^{-1} \transB{\nabla \pi}, -\transB{\nabla \pi} \big) ~|~ \transB{\nabla \pi} \in U_q(\Omega) \Big\},
\end{equation*}
where
\begin{equation*}
\LaplacePS = \LaplacePSq: \Def(\LaplacePSq) \subset L_q(\Omega)^n \rightarrow L_q(\Omega)^n, \quad u \mapsto \Delta u
\end{equation*}
on $\Def(\LaplacePSq) := \{ u \in W^2_q(\Omega)^n : \Dm(u)\nu = 0 \text{ and } \nu \cdot u = 0 \text{ on } \partial \Omega \}$
is the Laplace operator subject to perfect slip boundary conditions. In
particular, its resolvent $(\lambda - \LaplacePSq)^{-1}$ exists.
Furthermore, $\dim S_\mathrm{hom} = \dim U_q(\Omega)$ holds in the algebraic sense.
\item \label{thm:Stokes perfect slip 3} In case Assumption~\thref{thm:Assumption A}\eqref{thm:Assumption A 1} is valid, we obtain:
For $f \in L_q(\Omega)^n$ there exists a unique solution
$(u,\transB{\nabla p}) \in [W^2_q(\Omega)^n \cap L_{q,\sigma}(\Omega)] \times \modGradq$
of \eqref{eq:Stokes perfect slip}
if and only if $f \in L_{q,\sigma}(\Omega) + \Gradq$.
In particular, in case Assumption~\thref{thm:Assumption B} is valid as well, there exists a unique solution of \eqref{eq:Stokes perfect slip} in $[W^2_q(\Omega)^n \cap L_{q,\sigma}(\Omega)] \times \modGradq$ for any $f \in L_q(\Omega)^n$.
\item \label{thm:Stokes perfect slip 4} In case Assumption~\thref{thm:Assumption A} (i.e.,~\thref{thm:Assumption A}\eqref{thm:Assumption A 1} and~\thref{thm:Assumption A}\eqref{thm:Assumption A 2}) is valid, the solution in \eqref{thm:Stokes perfect slip 3} fulfills the resolvent estimate
\begin{equation} \label{eq:resolvent estimate perfect slip}
\| (\lambda u, \sqrt{\lambda} \nabla u, \nabla^2 u, \nabla p) \|_q
\le C \| f \|_q
\end{equation}
for any $f \in L_{q,\sigma}(\Omega) + \Gradq$.
\end{enumerate}
\end{theorem}
The above result can be interpreted as follows:
\begin{remark}
For Helmholtz domains $\Omega$ the Stokes resolvent problem subject
to perfect slip with right-hand side $f \in L_q(\Omega)^n$
is known to be well-posed on $[W^2_q(\Omega)^n \cap
L_{q,\sigma}(\Omega)] \times \Gradq$, see e.g.\ \cite{HS18}. 
Theorem~\ref{thm:Stokes perfect
slip}\eqref{thm:Stokes perfect slip 1},\eqref{thm:Stokes perfect slip 2} shows that in general this is no longer true for non-Helmholtz
domains, that is, if $\dim(L_{q,\sigma}(\Omega) \cap \Gradq) \ge 1$
or if $\codim(L_{q,\sigma}(\Omega) + \Gradq) \ge 1$. This, e.g., is 
the case for (perturbed) sector-like domains, 
see Remark~\ref{remmainassum}\eqref{remmainassum f},\eqref{remmainassum g}.
On the other hand, if the existence of the Helmholtz decomposition
is replaced by the generalized Assumptions~\ref{thm:Assumption A}
and~\ref{thm:Assumption B}, then the Stokes resolvent problem
subject to perfect slip with right-hand side $f \in L_q(\Omega)^n$
is well-posed in 
$[W^2_q(\Omega)^n \cap L_{q,\sigma}(\Omega)] \times \modGradq$,
i.e., if we replace $\Gradq$ by the smaller space $\modGradq$.
Again this is the case for (perturbed) sector-like domains
and $q>\frac{2}{1-\pi/\beta}$, 
see Remark~\ref{remmainassum}\eqref{remmainassum f},\eqref{remmainassum g}.
\end{remark}

The next result concerns inhomogeneous tangential boundary
conditions.

\begin{theorem} \label{thm:Stokes perfect slip inhom}
Let $0 < \theta < \pi$ and let Assumption~\thref{thm:Assumption C} be valid.
Then there exist $\lambda_0 = \lambda_0(n,q,\theta,\Omega) > 0$ and $C = C(n,q,\theta,\Omega) > 0$ such that for $\lambda \in \Sigma_\theta$, $|\lambda| \ge \lambda_0$ we have the following, concerning
\begin{equation} \label{eq:Stokes perfect slip inhom}
\left\{
\begin{array}{rll}
\lambda u - \Delta u + \transB{\nabla p} & = f & \text{in } \Omega \\
\Div u & = 0 & \text{in } \Omega \\
\Dm(u)\nu & = \Pi_\tau g & \text{on } \partial \Omega \\
\nu \cdot u & = 0 & \text{on } \partial \Omega.
\end{array}
\right.
\end{equation}
\begin{enumerate}[\upshape (i)]
\item \label{thm:Stokes perfect slip inhom 1}
If Assumption~\thref{thm:Assumption B} is valid, then for all $f \in L_q(\Omega)^n$ and $g \in W^1_q(\Omega)^n$ there exists a solution
\begin{equation*}
(u,\transB{\nabla p}) \in [W^2_q(\Omega)^n \cap L_{q,\sigma}(\Omega)] \times \Gradq
\end{equation*}
of \eqref{eq:Stokes perfect slip inhom} (which is not unique, in general; 
see Theorem~\thref{thm:Stokes perfect slip}\eqref{thm:Stokes perfect slip 2}).
\item \label{thm:Stokes perfect slip inhom 2} If Assumptions~\thref{thm:Assumption A} and~\thref{thm:Assumption B} are valid, then
for all $f \in L_q(\Omega)^n$ and $g \in W^1_q(\Omega)^n$
there exists a unique solution
$(u,\transB{\nabla p}) \in [W^2_q(\Omega)^n \cap L_{q,\sigma}(\Omega)] \times \modGradq$
of \eqref{eq:Stokes perfect slip inhom} such that
\begin{equation} \label{eq:resolvent estimate perfect slip inhom}
\| (\lambda u, \sqrt{\lambda} \nabla u, \nabla^2 u, \nabla p) \|_q
\le C \| (f,\sqrt{\lambda} g,\nabla g) \|_q.
\end{equation}
\end{enumerate}
\end{theorem}

Finally, we state the corresponding result concerning partial slip type conditions.

\begin{theorem} \label{thm:Stokes partial slip}
Let Assumptions~\thref{thm:Assumption B} and~\thref{thm:Assumption C} be valid, $0 < \theta < \pi$ and $\alpha \in \mathbb{R}$.
Then there exist $\lambda_0 = \lambda_0(n,q,\theta,\Omega,\alpha) > 0$ and $C = C(n,q,\theta,\Omega) > 0$ such that for $\lambda \in \Sigma_\theta$, $|\lambda| \ge \lambda_0$ we have the following with regard to
\begin{equation} \label{eq:Stokes partial slip Dpm}
\left\{
\begin{array}{rll}
\lambda u - \Delta u + \transB{\nabla p} & = f & \text{in } \Omega \\
\Div u & = 0 & \text{in } \Omega \\
\Pi_\tau (\alpha u + \Dpm(u)\nu) & = \Pi_\tau g & \text{on } \partial \Omega \\
\nu \cdot u & = 0 & \text{on } \partial \Omega.
\end{array}
\right.
\end{equation}
Here we denote $\eqref{eq:Stokes partial slip Dpm}_+$ and $\eqref{eq:Stokes partial slip Dpm}_-$ for the respective boundary terms $\Dpm$ again.
\begin{enumerate}[\upshape (i)]
\item \label{thm:Stokes partial slip 1} There exists $\epsilon = \epsilon(n,q,\Omega,\lambda) > 0$ so that in case $|\alpha| < \epsilon$ for any $f \in L_q(\Omega)^n$ and $g \in W^1_q(\Omega)^n$ there exists a solution
\begin{equation*}
(u,\transB{\nabla p}) \in [W^2_q(\Omega)^n \cap L_{q,\sigma}(\Omega)] \times \Gradq
\end{equation*}
of $\eqref{eq:Stokes partial slip Dpm}_-$.
\item \label{thm:Stokes partial slip 2} If Assumption~\thref{thm:Assumption A} is valid, then for any $f \in L_q(\Omega)^n$ and $g \in W^1_q(\Omega)^n$ there exists a unique solution
\begin{equation*}
(u,\transB{\nabla p}) \in [W^2_q(\Omega)^n \cap L_{q,\sigma}(\Omega)] \times \modGradq
\end{equation*}
of $\eqref{eq:Stokes partial slip Dpm}_+$ resp.\ of $\eqref{eq:Stokes partial slip Dpm}_-$
and the estimate
\begin{equation} \label{eq:resolvent estimate partial slip}
\| (\lambda u, \sqrt{\lambda} \nabla u, \nabla^2 u, \nabla p) \|_q
\le C \| (f,\sqrt{\lambda}g,\nabla g) \|_q
\end{equation}
holds in each case.
\end{enumerate}
\end{theorem}


\begin{remark}
Theorem~\ref{thm:Stokes partial slip}\eqref{thm:Stokes partial slip 2} yields well-posedness of the Stokes resolvent problem subject to partial slip type boundary conditions for a large class of domains
including non-Helmholtz domains such as sector-like domains of Bogovski\u{\i} and Maslennikova type. 
It hence extends the results on no slip conditions obtained in~\cite{BGMST17} 
to partial slip type boundary conditions. 
Of course, Theorem~\ref{thm:Stokes partial slip}\eqref{thm:Stokes partial slip 2} 
also includes the large class of Helmholtz domains in case
Assumption~\ref{thm:Assumption C} is satisfied. 
Consequently, for that case
it also extends the main result in~\cite{GHHS12}, 
concerning the Stokes resolvent problem subject to no slip conditions,
to partial slip type boundary conditions. 
\end{remark}

\begin{remark}\label{noninjectps}
Note that Theorem~\ref{thm:Stokes partial slip}\eqref{thm:Stokes perfect slip 2} yields that solutions $(u,\nabla p)$ of \eqref{eq:Stokes partial slip Dpm} in the class $[W^2_q(\Omega)^n \cap L_{q,\sigma}(\Omega)] \times \Gradq$ are not unique in case
$[W^2_q(\Omega)^n \cap L_{q,\sigma}(\Omega)] \times \modGradq$
is a proper subspace.
In fact, if $\nabla \pi \in U_q(\Omega) = L_{q,\sigma}(\Omega) \cap \Gradq$ is a nonzero function, then
Theorem~\ref{thm:Stokes partial slip}\eqref{thm:Stokes perfect slip 2} yields a solution
$(u,\transB{\nabla p}) \in [W^2_q(\Omega)^n \cap L_{q,\sigma}(\Omega)] \times \modGradq$ of \eqref{eq:Stokes partial slip Dpm} with $f = \nabla \pi$ and $g = 0$,
so $(u,\nabla p - \nabla \pi) \in [W^2_q(\Omega)^n \cap L_{q,\sigma}(\Omega)] \times \Gradq$ is a solution of the homogeneous problem \eqref{thm:Stokes partial slip}.
This solution is non-trivial, since $\nabla p - \nabla \pi = 0$ would yield $\nabla \pi = 0$, due to the definition of $\modGradq$.
\end{remark}

\begin{remark} \label{matrixalpha}
As we will see in the proof of Theorem~\ref{thm:Stokes partial slip}, we could further add another zero order boundary term of the form $Au$ with some matrix
$A \in W^1_\infty(\Omega)^{n \times n}$ to the partial slip type boundary conditions, i.e.,
\begin{equation} \label{eq:partial slip plus perturbation}
\left\{
\begin{array}{rll}
\Pi_\tau (\alpha u + \Dpm(u)\nu + Au) & = \Pi_\tau g & \text{on } \partial \Omega \\
\nu \cdot u & = 0 & \text{on } \partial \Omega
\end{array}
\right.
\end{equation}
and the assertion is still valid
(where the quantities now may additionally depend on the matrix $A$, of course).
We could further replace \eqref{eq:partial slip plus perturbation} by
Robin type boundary conditions, as mentioned in the introduction.
Indeed, an inspection of
the proof of Lemma~\ref{thm:matrix different partial slip b.c.} shows
that Robin type boundary conditions can be regarded as 
a perturbation of perfect slip boundary conditions as well.
\end{remark}

Theorem~\ref{thm:Stokes partial slip} is based on the existence 
of decomposition \eqref{eq:decomposition modGrad a}, which results
from Lemma~\ref{thm:complemented subspaces}(\ref{thm:complemented subspaces 3}) and relies on
Assumptions~\ref{thm:Assumption A} and~\ref{thm:Assumption B}.
Note that 
without Assumptions~\ref{thm:Assumption A} and~\ref{thm:Assumption B},
Lemma~\ref{thm:complemented subspaces}(\ref{thm:complemented subspaces 3})
still implies decomposition \eqref{eq:decomposition modGrad b} to hold
in the algebraic sense with a certain space 
\begin{equation} \label{algdefmodg}
	\modGradq = (I - \QProj_q) \Gradq
\end{equation}
and where $\QProj_q$ is a (algebraic and possibly unbounded) 
projection onto $U_q(\Omega)$.
Assuming the validity of the assertion in
Theorem~\ref{thm:Stokes partial slip}(ii) for such a given $\modGradq$
even implies the necessity of 
\eqref{eq:decomposition modGrad a} as a topological decomposition
with that $\modGradq$. 
In other words, for a fixed $\modGradq$ given through
\eqref{algdefmodg},
well-posedness of the Stokes resolvent problem in the sense 
of Theorem~\ref{thm:Stokes partial slip}(ii) is equivalent to
\eqref{eq:decomposition modGrad a} (as a topological
decomposition). While sufficiency of that equivalence is proved by 
Theorem~\ref{thm:Stokes partial slip},
the necessity follows from

\begin{theorem} \label{thm:Stokes implies generalized HH}
For arbitrary $\alpha \in \mathbb{R}$, $0 < \theta < \pi$ and either
$\eqref{eq:Stokes partial slip Dpm}_+$ or $\eqref{eq:Stokes partial slip Dpm}_-$ assume that there exist
$\lambda_0 = \lambda_0(n,q,\theta,\Omega,\alpha) > 0$ and $C = C(n,q,\theta,\Omega) > 0$ such that for $\lambda \in \Sigma_\theta$, $|\lambda| \ge \lambda_0$ and every $f \in L_q(\Omega)^n$ there exists a unique solution
$
(u,\transB{\nabla p}) \in [W^2_q(\Omega)^n \cap L_{q,\sigma}(\Omega)] \times \modGradq
$
of $\eqref{eq:Stokes partial slip Dpm}$ (with $g = 0$)
that satisfies
\begin{equation} \label{eq:resolvent estimate partial slip b}
\| (\lambda u, \sqrt{\lambda} \nabla u, \nabla^2 u, \nabla p) \|_q
\le C \| f \|_q.
\end{equation}
Then $\modGradq$ is a closed subspace of $L_q(\Omega)^n$
and we have
\begin{equation} \label{eq:gen Helmholtz}
L_q(\Omega)^n = L_{q,\sigma}(\Omega) ~ \oplus ~ \modGradq
\end{equation}
as a topological decomposition.
In particular, Assumptions~\ref{thm:Assumption A} and~\ref{thm:Assumption B} are valid.
\end{theorem}

\begin{proof}
We will show that in the given situation Assumption~\ref{thm:Assumption B} 
holds. For this purpose we can adapt the argument given 
in~\cite{Shi13} for the classical Helmholtz decomposition.
For $f \in L_q(\Omega)^n$ and some sufficiently large $\lambda > 0$ denote the unique solution of \eqref{eq:Stokes partial slip Dpm} by $(u_\lambda,\nabla p_\lambda)$. For arbitrary $\varphi \in \widehat{W}^1_{q'}(\Omega)$ we have
\begin{equation}
\langle \lambda u_\lambda, \nabla \varphi \rangle_{q,q'}
- \langle \Delta u_\lambda, \nabla \varphi \rangle_{q,q'}
+ \langle \nabla p_\lambda, \nabla \varphi \rangle_{q,q'}
= \langle f, \nabla \varphi \rangle_{q,q'}.
\end{equation}
Now, $u_\lambda \in L_{q,\sigma}(\Omega)$ gives $\langle \lambda
u_\lambda, \nabla \varphi \rangle_{q,q'} = 0$. Note that we have
utilized \eqref{eq:L_q,sigma} here.
Using \eqref{eq:resolvent estimate partial slip b}, we obtain that $(u_\lambda)_{\lambda \ge \lambda_0}$ is bounded in $W^2_q(\Omega)^n$ and therefore has a weak limit $\widetilde{u} \in W^2_q(\Omega)^n$ for $\lambda \to \infty$ (by considering some sequence $\lambda_n \xrightarrow{n \to \infty} \infty$ and passing to a subsequence if necessary). We further receive from \eqref{eq:resolvent estimate partial slip b} that $\widetilde{u}_\lambda \to 0$ in $L_q(\Omega)^n$ for $\lambda \to \infty$ and hence $\widetilde{u} = 0$.
Moreover, $(\nabla p_\lambda)_{\lambda \ge \lambda_0}$ is bounded (due to \eqref{eq:resolvent estimate partial slip b}) and therefore has a weak limit in $L_q(\Omega)^n$ for $\lambda \to \infty$ (again for some suitable subsequence). This weak limit must be some gradient $\nabla \widetilde{p} \in \Gradq$ since $\Gradq$ is weakly closed in $L_q(\Omega)^n$.
In total, letting $\lambda \to \infty$ in \eqref{eq:resolvent estimate partial slip b} yields
\begin{equation*}
\langle \nabla \widetilde{p},\nabla \varphi \rangle_{q,q'}
= \langle f,\nabla \varphi \rangle_{q,q'}
\end{equation*}
for all $\varphi \in \widehat{W}^1_{q'}(\Omega)$
so we have
$f = (f - \nabla \widetilde{p}) + \nabla \widetilde{p} \in L_{q,\sigma}(\Omega) + \Gradq$.
Thus Assumption~\ref{thm:Assumption B} is valid.

By this fact and thanks to Lemma~\ref{thm:complemented subspaces}
we obtain decomposition \eqref{eq:gen Helmholtz}
in the algebraic sense. 
Unique solvability of \eqref{eq:Stokes partial slip Dpm} (for some arbitrary $\lambda$) yields that the related solution operator is an isomorphism.
This implies $\modGradq$ to be closed in $L_q(\Omega)^n$.

Finally, if $\QProj_q:L_q(\Omega)\to L_q(\Omega)$ is the 
(a priori algebraic) projection onto $U_q(\Omega)$ such that
$\modGradq=(I - \QProj_q) \Gradq$ we deduce
\[
	\Gradq = \QProj_q \Gradq \oplus (I - \QProj_q) \Gradq 
	= U_q(\Omega)\oplus \modGradq.
\]
Since both, $U_q(\Omega)$ and $\modGradq$ are closed in $L_q(\Omega)^n$,
$\QProj_q$ is bounded. Consequently, 
Assumption~\ref{thm:Assumption A} is fulfilled, too.
\end{proof}

\section{Main results for the time dependent problems}
\label{sec:instationary}

Still let $\Omega \subset \mathbb{R}^n$ be a domain with uniform 
$C^{2,1}$-boundary, $n \ge 2$ and $1 < q < \infty$.
We aim to define a suitable Stokes operator, related to the Stokes equations 
\begin{equation} \label{eq:SE Dpm}
\left\{
\begin{array}{rll}
\partial_t u - \Delta u + \transB{\nabla p} & = f & \text{in } (0,T) \times \Omega \\
\Div u & = 0 & \text{in } (0,T) \times \Omega \\
\Pi_\tau (\alpha u + \Dpm(u)\nu) & = 0 & \text{on } (0,T) \times \partial \Omega \\
\nu \cdot u & = 0 & \text{on } (0,T) \times \partial \Omega \\
u|_{t=0} & = u_0 & \text{in } \Omega.
\end{array}
\right.
\end{equation}
Again by $\eqref{eq:SE Dpm}_+$ and $\eqref{eq:SE Dpm}_-$ we refer to
the Stokes equations subject to the boundary conditions related to the 
boundary operator $\Dm$ and $\Dp$, respectively.
Under Assumptions~\ref{thm:Assumption A} and~\ref{thm:Assumption B}
we can use decomposition \eqref{eq:decomposition modGrad a} to reformulate $\eqref{eq:SE Dpm}_+$ resp.\ $\eqref{eq:SE Dpm}_-$ with $f: (0,T) \rightarrow L_q(\Omega)^n$ and $u_0 \in L_{q,\sigma}(\Omega)$ as the equivalent
problems
\begin{equation} \label{eq:PSE Dpm}
\left\{
\begin{array}{rll}
\partial_t u - \Proj \Delta u & = f_0 & \text{in } (0,T) \times \Omega \\
\Div u & = 0 & \text{in } (0,T) \times \Omega \\
\Pi_\tau (\alpha u + \Dpm(u)\nu) & = 0 & \text{on } (0,T) \times \partial \Omega \\
\nu \cdot u & = 0 & \text{on } (0,T) \times \partial \Omega \\
u|_{t=0} & = u_0 & \text{in } \Omega
\end{array}
\right.
\end{equation}
with $f_0: (0,T) \rightarrow L_{q,\sigma}(\Omega)$ and $u_0 \in L_{q,\sigma}(\Omega)$,
where
$\Proj = \Proj_q: L_q(\Omega)^n \rightarrow L_q(\Omega)^n$
is the continuous linear projection onto $L_{q,\sigma}(\Omega)$ related to decomposition \eqref{eq:decomposition modGrad a}.

For $\alpha \in \mathbb{R}$ we define the Stokes operator subject to partial slip type boundary conditions as
\begin{equation} \label{eq:Stokes operator def}
\Stokespm = \Stokesqpm: \Def(\Stokesqpm) \subset L_{q,\sigma}(\Omega) \rightarrow L_{q,\sigma}(\Omega), \quad
u \longmapsto \Proj_q \Delta u
\end{equation}
on
$\Def(\Stokesqpm) := \{ u \in W^2_q(\Omega)^n : \Pi_\tau(\alpha u + \Dpm(u)\nu) = 0 \text{ and } \nu \cdot u = 0 \text{ on } \partial \Omega \} \cap L_{q,\sigma}(\Omega)$.

We obtain the reformulation~\eqref{eq:PSE Dpm} by the following equivalence of the corresponding resolvent problems which is an immediate consequence of the continuity of the projection $\Proj$.

\begin{lemma} \label{thm:Stokes operator}
Let Assumptions~\thref{thm:Assumption A} and~\thref{thm:Assumption B} be valid and let $0 < \theta < \pi$ and $\alpha \in \mathbb{R}$.
Then for any $\lambda \in \Sigma_\theta$ and $f \in L_q(\Omega)^n$ a couple
\begin{equation*}
(u,\transB{\nabla p}) \in \Def(\Stokesqpm) \times \modGradq
\end{equation*}
solves
\begin{equation} \label{eq:Stokes partial slip Dpm hom}
\left\{
\begin{array}{rll}
\lambda u - \Delta u + \transB{\nabla p} & = f & \text{in } \Omega \\
\Div u & = 0 & \text{in } \Omega \\
\Pi_\tau (\alpha u + \Dpm(u)\nu) & = 0 & \text{on } \partial \Omega \\
\nu \cdot u & = 0 & \text{on } \partial \Omega
\end{array}
\right.
\end{equation}
if and only if $u \in \Def(\Stokesqpm)$ solves
\begin{equation} \label{eq:Stokes partial slip Dpm hom Proj}
\left\{
\begin{array}{rll}
\lambda u - \Proj \Delta u & = \Proj f & \text{in } \Omega \\
\Div u & = 0 & \text{in } \Omega \\
\Pi_\tau (\alpha u + \Dpm(u)\nu) & = 0 & \text{on } \partial \Omega \\
\nu \cdot u & = 0 & \text{on } \partial \Omega
\end{array}
\right.
\end{equation}
and if $\transB{\nabla p} = (I - \Proj) (f - \lambda u + \Delta u)$.
In either case we have
\begin{equation} \label{eq:Stokes operator resolvent estimate a}
\| (\lambda u, \sqrt{\lambda} \nabla u, \nabla^2 u, \nabla p) \|_q \le C \| f \|_q
\end{equation}
for all $|\lambda| > \lambda_0$
with some $C = C(n,q,\theta,\Omega) > 0$.
\end{lemma}

\begin{theorem} \label{thm:Stokes semigroup}
Let $\alpha \in \mathbb{R}$, $1 < p \le q < \infty$ and
let Assumptions~\thref{thm:Assumption A},~\thref{thm:Assumption B} and~\thref{thm:Assumption C} be valid for $q$.
Then the Stokes operator $\Stokesqpm$ is the generator of a strongly continuous analytic semigroup
$(e^{t \Stokesqpm})_{t \ge 0}$
on $L_{q,\sigma}(\Omega)$.
For arbitrary $\omega \in (0,\frac{\pi}{2})$ we can find $d \ge 0$ such that the semigroup, generated by the shiftet Stokes operator $\Stokesqpm - d$, is bounded with angle $\omega$.
If Assumptions~\thref{thm:Assumption A},~\thref{thm:Assumption B}
and~\thref{thm:Assumption C} do hold for $p$ as well,
then for $T > 0$
there exists a constant $C = C(n,q,p,\Omega,\alpha,T) > 0$ such that for
all $t \in (0,T)$ and any $f \in L_{p,\sigma}(\Omega)$ the following
inequalities hold:
\begin{enumerate}[\upshape (i)]
\item \label{thm:estimate Stokes semigroup 1} $\| e^{t \Stokesppm} f
\|_q \le C t^{-\frac{n}{2}(\frac{1}{p} - \frac{1}{q})} \| f \|_p$ \quad
if $\frac{1}{p} - \frac{1}{q} < \frac{2}{n}$,
\item \label{thm:estimate Stokes semigroup 2} $\| \nabla e^{t \Stokesppm} f \|_q \le C t^{-\frac{1}{2} -\frac{n}{2}(\frac{1}{p} - \frac{1}{q})} \| f \|_p$ \quad if $\frac{1}{p} - \frac{1}{q} < \frac{1}{n}$.
\end{enumerate}
\end{theorem}

Theorem~\ref{thm:Stokes semigroup} 
leads to the following result on mild solutions for the 
corresponding Navier-Stokes equations
\begin{equation} \label{eq:NSE Dpm}
\left\{
\begin{array}{rll}
\partial_t u - \Delta u + \transB{\nabla p} + (u \cdot \nabla)u & = 0 & \text{in } (0,T) \times \Omega \\
\Div u & = 0 & \text{in } (0,T) \times \Omega \\
\Pi_\tau( \alpha u + \Dpm(u)\nu) & = 0 & \text{on } (0,T) \times \partial \Omega \\
\nu \cdot u & = 0 & \text{on } (0,T) \times \partial \Omega \\
u|_{t=0} & = u_0 & \text{in } \Omega,
\end{array}
\right.
\end{equation}
where $(u \cdot \nabla)u = \sum_{j=1}^n u^{j} (\partial_j u)$.
Once again we use the notation $\eqref{eq:NSE Dpm}_+$ and 
$\eqref{eq:NSE Dpm}_-$ for the
system related to $\Dm$ and $\Dp$, respectively.
Also note that by $\Div u = 0$ we can write 
$(u \cdot \nabla)u = \sum_{j=1}^n \partial_j (u^{j} u)$.

\begin{theorem} \label{thm:NSE}
Let $n < q < \infty$ such that Assumptions~\thref{thm:Assumption A},~\thref{thm:Assumption B} and~\thref{thm:Assumption C} are valid for $q$ and also for $\frac{q}{2}$.
As before, we denote the projections related to decomposition
\eqref{eq:decomposition modGrad a} by $\Proj_q$ resp.\ $\Proj_{q/2}$.
Let $0 < \theta < \pi$, $\alpha \in \mathbb{R}$ and $u_0 \in L_{q,\sigma}(\Omega)$.
Then the Navier-Stokes equations $\eqref{eq:NSE Dpm}_+$ resp.\ $\eqref{eq:NSE Dpm}_-$ admit a unique local mild solution depending continuously on $u_0$, i.e., there exists $T > 0$
such that the integral equation
\begin{equation} \label{eq:Integral equation}
u(t) = e^{t \Stokesqpm} u_0 - \int_0^t e^{(t - s) \Stokes{q/2}} \Proj_{
q/2} \sum_{j=1}^n \partial_j (u^{j}(s) u(s)) ds, \quad t \in [0,T]
\end{equation}
related to the projected Navier-Stokes equations
\begin{equation} \label{eq:PNSE Dpm}
\left\{
\begin{array}{rll}
\partial_t u - \Proj_{q} \Delta u + \Proj_{q/2} (u \cdot \nabla)u & = 0 & \text{in } (0,T) \times \Omega \\
\Div u & = 0 & \text{in } (0,T) \times \Omega \\
\Pi_\tau( \alpha u + \Dpm(u)\nu) & = 0 & \text{on } (0,T) \times \partial \Omega \\
\nu \cdot u & = 0 & \text{on } (0,T) \times \partial \Omega \\
u|_{t=0} & = u_0 & \text{in } \Omega
\end{array}
\right.
\end{equation}
admits a unique solution $u$ satisfying
\begin{equation} \label{eq:mild solution condition}
\begin{split}
u &\in \mathrm{BC} \big( [0,T],L_{q,\sigma}(\Omega)
\big),\\
[t \mapsto \sqrt{t} \nabla u(t)] &\in \mathrm{BC} \big ([0,T],L_q(\Omega)^{n \times n}
\big).
\end{split}
\end{equation}
\end{theorem}

\begin{remark}
Observe that rewriting \eqref{eq:PNSE Dpm} as the original Navier-Stokes equations \eqref{eq:NSE Dpm} might be not possible in case the projections $\Proj_q$ and $\Proj_{q/2}$ fail to coincide on $L_q(\Omega)^n \cap L_{q/2}(\Omega)^n$.
In fact, the projection $\Proj_q$ might not be consistent with respect to $1 < q < \infty$ in general.
For the sake of well-definedness of \eqref{eq:Integral equation} and for the 
construction of mild solutions, however, 
we require $\Proj_{q/2}$ in front of the nonlinearity.
On the other hand, if $(\Proj_q)_{q_0<q<\infty}$ is a consistent scale
for some $q_0>1$, then Theorem~\ref{thm:NSE}
provides a unique mild solution of the classical 
Navier-Stokes system for that range of $q$, as usual.
Note that consistency of $(\Proj_q)_{1<q<\infty}$ is known
for a large class of Helmholtz domains
(see~\cite{HS18}, Sec. 2.2) and for $(\Proj_q)_{q_0<q<\infty}$
with $q_0=2/(1-\pi/\beta)$ on sector-like domains, 
see Remark~\ref{remmainassum}\eqref{remmainassum f} and \cite{MB86}. 
\end{remark}

\section{Preliminary tools for the proofs}\label{sec:preliminaryproofs}

We intend to make use of various versions of Gauß's theorem, Green's
formula and continuity of the trace map for the normal component.
Since for domains with uniform $C^{2,1}$-boundary
these tools are partly hard to find in the common literature, we
recall precise statements including proofs here. Some
of the basic proofs, however, are outsourced 
to Appendix~\ref{sec:Appendix Gauss}.
For a domain $\Omega \subset \mathbb{R}^n$ with uniform $C^{2,1}$-boundary, $n \ge 2$ and $1 < q < \infty$ we define, as in~\cite{Soh01},
\begin{equation*}
E_q(\Omega) := \{ f \in L_q(\Omega)^n : \Div f \in L_q(\Omega) \}
\end{equation*}
with norm $\| f \|_{E_q(\Omega)} := \| f \|_q + \| \Div f \|_q$ and
\begin{equation*}
W^{-\frac{1}{q}}_q(\partial \Omega)
:= [W^{1 - \frac{1}{q'}}_{q'}(\partial \Omega)]'.
\end{equation*}

We denote the standard trace operator by
\begin{equation*}
\Tr = \Tr_{\partial \Omega}: W^1_q(\Omega) \rightarrow W^{1-\frac{1}{q}}_q(\partial \Omega), \quad
\Tr u = u|_{\partial \Omega} ~\forall u \in C_c^\infty(\overline{\Omega})
\end{equation*}
which is continuous for $1 \le q < \infty$ (see Lemma~\ref{thm:Trace}).
We will write $u|_{\partial \Omega} = \Tr u$ also for $u \in W^1_q(\Omega)$. Furthermore, for the surface integral we will write
$\int_{\partial \Omega} u \D \sigma = \int_{\partial \Omega}
u|_{\partial \Omega} \D \sigma$ for $u \in W^1_1(\Omega)$ if no
confusion seems likely.

We will further make use of the generalized normal trace operator 
$\Tr_\nu$. For this purpose, we require density of the embedding $C_c^\infty(\overline{\Omega})^n \subset E_q(\Omega)$ for $1 < q < \infty$ (see Lemma~\ref{thm:density E_q}). Existence of this trace means that
there is a bounded linear operator
\begin{equation*}
\Tr_\nu: E_{q'}(\Omega) \longrightarrow W^{-\frac{1}{q'}}_{q'}(\partial \Omega)
\end{equation*}
such that for $v \in W^1_{q'}(\Omega)^n$ we have
$\Tr_\nu v = \nu \cdot v|_{\partial \Omega}$ in $W^{-\frac{1}{q'}}_{q'}(\partial \Omega)$, that is,
\begin{equation*}
\Tr_\nu v = \Big[ W^{1-\frac{1}{q}}_q(\partial \Omega) \ni g \mapsto \int_{\partial \Omega} g(\nu \cdot v) \D \sigma \Big]
\end{equation*}
(see Lemma~\ref{thm:Trace of the normal component}).
For $v \in E_{q'}(\Omega)$, we denote by $\langle u , \nu \cdot v \rangle_{\partial \Omega}
:= \langle \Tr u , \Tr_\nu v \rangle_{\partial \Omega}$ the application of $\Tr_\nu v$ to some $g = \Tr u \in W^{1-1/q}_q(\partial \Omega)$, $u \in W^1_q(\Omega)$.

The most general form of Gauß's theorem
\begin{equation*}
\int_\Omega \Div (u v) \D \lambda_n
= \dualb{u}{\nu \cdot v}
\end{equation*}
resp.\ of corresponding Green's formula
\begin{equation*}
\int_\Omega u (\Div v) \D \lambda_n
= \dualb{u}{\nu \cdot v} - \int_\Omega \nabla u \cdot v \D \lambda_n
\end{equation*}
that we intend to make use of, is for 
$u \in W^1_q(\Omega)$ and $v \in E_{q'}(\Omega)$, where $1 < q < \infty$ (see Lemmas~\ref{thm:Greens formula in E_q} and~\ref{thm:Extended Gauss theorem}).
Note that in this case
$\Div (u v) = \nabla u \cdot v + u (\Div v)$ is in fact a function in $L_1(\Omega)$, which can be seen via approximation (Lemma~\ref{thm:density E_q}).

We proceed with useful characterizations of the space $L_{q,\sigma}(\Omega)$ for $n \ge 2$ and $1 < q < \infty$. One well-known characterization is
\begin{equation} \label{eq:L_q,sigma}
L_{q,\sigma}(\Omega) = \{ f \in L_q(\Omega)^n : \dualq{f}{\nabla \varphi} = 0 ~\forall \varphi \in \widehat{W}^1_{q'}(\Omega) \}
\end{equation}
which is even true for arbitrary domains $\Omega \subset \mathbb{R}^n$ (see~\cite{Gal11}, Lem.\ III.1.1.).
Now let $\Omega \subset \mathbb{R}^n$ be a domain with uniform $C^{2,1}$-boundary.

\begin{lemma} \label{thm:characterization L_q,sigma}
We have
\begin{equation} \label{eq:embedding L_q,sigma}
L_{q,\sigma}(\Omega) \subset \{ f \in L_q(\Omega)^n : \Div f = 0, ~\nu \cdot f|_{\partial \Omega} = 0 \},
\end{equation}
where $\nu \cdot f|_{\partial \Omega} = \Tr_\nu f \in
W^{-1/q}_q(\partial \Omega)$.
If additionally Assumption~\thref{thm:Assumption C} is valid, then we have
equality in \eqref{eq:embedding L_q,sigma}.
\end{lemma}

\begin{proof}
Let $f \in L_{q,\sigma}(\Omega)$. For any $\varphi \in C_c^\infty(\Omega)$ we have
$\langle \Div f , \varphi \rangle = - \langle f , \nabla \varphi \rangle = 0$, due to \eqref{eq:L_q,sigma}, and therefore $\Div f = 0$ in the sense of distributions.
We now aim to show that $\dualb{g}{\Tr_\nu f} = 0$ holds for $g \in
W^{1-1/q'}_{q'}(\partial \Omega)$. By the surjectivity of the 
trace operator, we can write $g = \Tr u$ with some $u \in
W^1_{q'}(\Omega)$.
We use Lemma~\ref{thm:Extended Gauss theorem} (note that $f \in E_q(\Omega)$) and \eqref{eq:L_q,sigma} to obtain
\begin{equation*}
\dualb{g}{\Tr_\nu f}
= \dualb{u}{\nu \cdot f}
= \int_\Omega \Div (u f) \D \lambda_n
= \int_\Omega \nabla u \cdot f \D \lambda_n
= 0.
\end{equation*}

Conversely, let $f \in L_q(\Omega)^n$ with $\Div f = 0$ and $\nu \cdot f|_{\partial \Omega} = 0$ and additionally assume that Assumption~\ref{thm:Assumption C} is valid. For $\varphi \in C_c^\infty(\overline{\Omega})$ we have, using Lemma~\ref{thm:Extended Gauss theorem},
\begin{equation*}
\dualq{f}{\nabla \varphi}
= \int_\Omega \Div (\varphi f) \D \lambda_n
= \dualb{\varphi}{\nu \cdot f}
= 0.
\end{equation*}
Since $C_c^\infty(\overline{\Omega}) \subset \widehat{W}^1_{q'}(\Omega)$ is dense, this holds for $\varphi \in \widehat{W}^1_{q'}(\Omega)$ as well. Hence, \eqref{eq:L_q,sigma} gives that $f \in L_{q,\sigma}(\Omega)$.
\end{proof}

\begin{remark} \label{thm:aperture}
Note that without Assumption~\ref{thm:Assumption C} the right-hand side of \eqref{eq:embedding L_q,sigma} can in fact be larger than $L_{q,\sigma}(\Omega)$.
An aperture domain as considered in~\cite{FS96} and~\cite{Far96} is an example of a Helmholtz domain with uniform $C^{2,1}$-boundary 
for which equality in \eqref{eq:embedding L_q,sigma} does not hold if 
$q > \frac{n}{n-1}$. In that case we have
\begin{equation*}
L_{q,\sigma}(\Omega)
= \{ f \in L_q(\Omega)^n ~|~ \Div f = 0, ~\nu \cdot f|_{\partial \Omega} = 0, ~\Phi(f) = 0 \},
\end{equation*}
where $\Phi(f) = \int_{M} \nu \cdot f ~d\sigma$ denotes the flux of a function $f$ through the aperture of the domain and $M$ is an $(n-1)$-dimensional manifold shutting the aperture.
\end{remark}

When dealing with the boundary conditions under consideration, certain
elementary calculations will appear several times. Therefore, once
and for all we state them here and make use of them in the sequel often
without any further notice.

\begin{lemma} \label{thm:Dminus}
Consider a function $\varphi: \Omega \rightarrow \mathbb{R}$ and vector fields $u,v,w: \Omega \rightarrow \mathbb{R}^n$. We have the following calculation rules (in case the product rule for derivatives is applicable):
\begin{enumerate}[\upshape (i)]
\item \label{thm:Dminus 1} $\Div (\Dm(u) v)
= (\nabla \Div u - \transB{\Delta u}) \cdot v + (\transA{\nabla u} - \transB{\nabla u}) : \transB{\nabla v}$.
\item \label{thm:Dminus 2} $\Div (\Dm(u) \transB{\nabla \varphi})
= (\nabla \Div u - \transB{\Delta u}) \cdot \nabla \varphi$.
\item \label{thm:Dminus 3} $v \cdot \Dm(u) w = - w \cdot \Dm(u) v$.
\end{enumerate}
\end{lemma}

\begin{proof}
Straight forward computations yield \eqref{thm:Dminus 1} and \eqref{thm:Dminus 3}. Now \eqref{thm:Dminus 2} is a consequence of \eqref{thm:Dminus 1}, since $v := \transB{\nabla \varphi}$ implies
$(\transA{\nabla u} - \transB{\nabla u}) : \transB{\nabla v}
= \sum_{i,j=1}^n (\partial_j u^{i} - \partial_i u^{j}) \partial_i \partial_j \varphi
= 0$.
\end{proof}

In addition to \eqref{eq:uniformity},
the application of perturbation theory in our treatment of the appearing boundary conditions necessitates to take into account 
some further estimates associated with the uniform $C^{2,1}$-boundary of $\Omega$.
For this purpose, we introduce a more concrete view on the parametrization of $\partial \Omega$.
Recall the notation of the rotation $Q_l$, the translation $\tau_l$, the
balls $B_l$ as well as the functions $\omega_l$, the domains $H_l$ and
the index set $\Gamma = \Gamma_0 \cup \Gamma_1$ that we introduced in
order to describe a uniform $C^{2,1}$-boundary in Definition~\ref{thm:uniform C^2,1 domain}.
Fix some $l \in \Gamma_1$. A $C^{2,1}$-diffeomorphism between $H_{\omega_l}$ and $\mathbb{R}^n_+$ is given by
\begin{equation*}
\Phi_l : H_{\omega_l} \overset{\cong}{\longrightarrow} \mathbb{R}^n_+, \quad
x \mapsto (x',x_n - \omega_l(x'))^T.
\end{equation*}
We obtain
$\transA{\nabla \Phi_l} = I - ((\partial_j \omega_l) \delta_{in})_{i,j = 1,\dots,n}$ and
$(\transA{\nabla \Phi_l})^{-1} = 2I - \nabla \transA{\Phi_l}$.
Now $\Psi_l(x) := \Phi_l(Q_l(x - \tau_l))$ defines a $C^{2,1}$-diffeomorphism $\Psi_l: H_l \overset{\cong}{\longrightarrow} \mathbb{R}^n_+$.
Using the canonical extension of $\Phi_l$ to $\mathbb{R}^n$ and therefore of $\Psi_l$ as well, we receive functions $\Phi_l: \mathbb{R}^n \overset{\cong}{\longrightarrow} \mathbb{R}^n$ resp.\ $\Psi_l: \mathbb{R}^n \overset{\cong}{\longrightarrow} \mathbb{R}^n$.
Restriction to $B_l$ gives
\begin{equation*}
\Psi_l: B_l \overset{\cong}{\longrightarrow} V_l, \quad
x \mapsto \Phi_l(Q_l(x - \tau_l))
\end{equation*}
onto some open subset $V_l \subset \mathbb{R}^n$ and its inverse 
\begin{equation*}
\Psi_l^{-1}: V_l \overset{\cong}{\longrightarrow} B_l, \quad
x \mapsto Q_l^T \Phi_l^{-1}(x) + \tau_l.
\end{equation*}
The set of diffeomorphisms $\Psi_l$, $l \in \Gamma_1$ characterizes the $C^{2,1}$-manifold $\partial \Omega$.
The related parametrization is given by $\phi_l(\xi) := \Psi_l^{-1} \left( \begin{smallmatrix}
\xi \\ 0
\end{smallmatrix} \right)$, that is,
\begin{equation} \label{eq:Parametrization}
\phi_l: U_l \longrightarrow \partial \Omega \cap B_l, \quad
\xi \mapsto Q_l^T \begin{pmatrix} \xi \\ \omega_l(\xi) \end{pmatrix} + \tau_l,
\end{equation}
where $U_l := \{ \xi \in \mathbb{R}^{n-1} : \left( \begin{smallmatrix} \xi \\ 0 \end{smallmatrix} \right) \in V_l \}$
(see~\cite{For77}).
Using the theorem of Binet-Cauchy, we obtain
$\det \big( \transB{(\nabla \phi_l)} \transA{\nabla \phi_l} \big) = 1 + |\nabla' \omega_l|^2$,
in particular
\begin{equation} \label{eq:determinant estimate}
\| \det \big( \transB{(\nabla \phi_l)} \transA{\nabla \phi_l} \big) \|_\infty \ge 1.
\end{equation}
We further have
\begin{equation} \label{eq:Jacobi Parametrization estimate}
\| \nabla \phi_l \|_\infty \le C \quad \forall l \in \Gamma_1
\end{equation}
with a constant $C = C(n,M) > 0$ and $M > 0$ from \eqref{eq:uniformity}.
Using \eqref{eq:Jacobi Parametrization estimate} and Cramer's rule, we obtain
\begin{equation} \label{eq:estimate Gram matrix}
\| \big( \transB{(\nabla \phi_l)} \transA{\nabla \phi_l} \big)^{-1} \|_{1,\infty} \le C \quad \forall l \in \Gamma_1
\end{equation}
where $C = C(n,M) > 0$.

We choose a suitable partition of unity subordinate to the 
cover $B_l$, $l\in\Gamma$, of the uniform $C^{2,1}$-domain
$\Omega$. More precisely,
let $(\varphi_l)_{l \in \Gamma} \subset C^\infty(\mathbb{R}^n)$ so that $0 \le \varphi_l \le 1$, $\Supp(\varphi_l) \subset B_l$ and
\begin{equation} \label{eq:Partition of unity}
\sum_{l \in \Gamma} \varphi_l^2 = 1.
\end{equation}
Since the $B_l$ have a fixed radius $\rho$, we can choose $(\varphi_l)_{l \in \Gamma}$ in such a way that
\begin{equation} \label{eq:uniformity partition of unity}
\sup_{l \in \Gamma} \| \nabla \varphi_l \|_\infty < \infty
\quad \text{and} \quad
\sup_{l \in \Gamma} \| \nabla^2 \varphi_l \|_\infty < \infty.
\end{equation}

The outward unit normal vector at $\partial \Omega$ is $\nu: \partial \Omega \rightarrow \mathbb{R}^n$.
Let $\widehat{\nu}_l: \partial H_{\omega_l} \rightarrow \mathbb{R}^n$ be the outward unit normal vector at $\partial H_{\omega_l}$ for $l \in \Gamma_1$, which is given by
\begin{equation} \label{eq:unit normal of H_omega_l}
\widehat{\nu}_l = \frac{1}{\sqrt{|\nabla' \omega_l|^2 + 1}} (\partial_1 \omega_l,\dots,\partial_{n-1} \omega_l,-1)^T,
\end{equation}
and
let $\nu_l: \partial H_l \rightarrow \mathbb{R}^n$ be the outward unit
normal vector at $\partial H_l$, i.e., $\nu_l$ results from rotating and translating $\widehat{\nu}_l$. Then we have $\nu = \nu_l$ on $\partial \Omega \cap B_l = \partial H_l \cap B_l$.
The representation \eqref{eq:unit normal of H_omega_l} gives that we can extend $\widehat{\nu}_l$ constantly to a function in $W^2_\infty(H_{\omega_l})^n$ and therefore we can also extend $\nu_l$ to a function $\widebar{\nu}_l \in W^2_\infty(H_l)^n$.
This trivial extension yields a constant $C = C(n,M) > 0$ so that
\begin{equation} \label{eq:normal uniformity 1}
\| \widebar{\nu}_l \|_{2,\infty,H_l} \le C
\end{equation}
for all $l \in \Gamma_1$, where $M$ is the constant from \eqref{eq:uniformity}.
Now
\begin{equation} \label{eq:extension nu}
\widebar{\nu}
:= \sum_{l \in \Gamma_1} \varphi_l^2 \widebar{\nu}_l \in W^2_\infty(\Omega)^n
\end{equation}
is an extension of $\nu$, since we have
\begin{equation*}
\| \widebar{\nu} \|_\infty
= \sup_{m \in \Gamma_1} \| \Chi_{B_m} \sum_{l \approx m} \varphi_l^2 \widebar{\nu}_l \|_\infty
\le \sup_{m \in \Gamma_1} \sum_{l \approx m} \| \widebar{\nu}_l \|_\infty
\le \bar{N} C
\end{equation*}
and the analogous estimates for $\| \nabla \widebar{\nu} \|_\infty$ and $\| \nabla^2 \widebar{\nu} \|_\infty$.
Consequently,
\begin{equation} \label{eq:normal uniformity 2}
\| \widebar{\nu} \|_{2,\infty,\Omega} \le C
\end{equation}
for $C = C(n,M) > 0$.

\section{The resolvent problem for the heat equation}
\label{sec:heat}

An essential tool for the proof of 
the results on the Stokes equations in the previous sections 
are resolvent estimates for the 
heat equation subject to perfect slip on uniform $C^{2,1}$-domains 
$\Omega \subset \mathbb{R}^n$, where we assume $n \ge 2$ and 
$1 < q < \infty$ again.

\begin{theorem} \label{thm:HE with PS}
Let $0 < \theta < \pi$. Then there exist $\lambda_0 = \lambda_0(n,q,\theta,\Omega) > 0$ and $C = C(n,q,\theta,\Omega) > 0$ such that for $\lambda \in \Sigma_\theta$, $|\lambda| \ge \lambda_0$ the problem
\begin{equation} \label{eq:HE}
\left\{
\begin{array}{rll}
\lambda u - \Delta u & = f & \text{in } \Omega \\
\Dm(u)\nu & = \Pi_\tau g & \text{on } \partial \Omega \\
\Pi_\nu u & = \Pi_\nu h & \text{on } \partial \Omega
\end{array}
\right.
\end{equation}
has a unique solution $u \in W^2_q(\Omega)^n$ for any $f \in L_q(\Omega)^n$, $g \in W^1_q(\Omega)^n$ and $h \in W^2_q(\Omega)^n$ and this solution fulfills the resolvent estimate
\begin{equation} \label{eq:HE resolvent estimate}
\| (\lambda u, \sqrt{\lambda} \nabla u, \nabla^2 u) \|_q
\le C \| (f, \sqrt{\lambda} g, \nabla g, \lambda h, \sqrt{\lambda} \nabla h, \nabla^2 h) \|_q.
\end{equation}
In particular
\begin{equation*}
\LaplacePS = \LaplacePSq: \Def(\LaplacePSq) \subset L_q(\Omega)^n \rightarrow L_q(\Omega)^n, \quad u \mapsto \Delta u
\end{equation*}
on $\Def(\LaplacePSq) := \{ u \in W^2_q(\Omega)^n : \Dm(u)\nu = 0 \text{ and } \nu \cdot u = 0 \text{ on } \partial \Omega \}$
is the generator of a strongly continuous analytic semigroup.
\end{theorem}

Observe that due to the
boundary conditions \eqref{eq:HE} is a system that does not 
decouple into scalar equations (except for flat boundaries).
For the proof of Theorem~\ref{thm:HE with PS}
we apply a localization procedure as it is performed, e.g., in
\cite{Kun03}. 
To this end, the proof is divided into several steps: We start with the half space $\Omega = \mathbb{R}^n_+ = \{ x = (x',x_n)^T \in \mathbb{R}^n: x_n > 0 \}$ and proceed with
bending, rotating and shifting the half space.
The bent half space $H_\omega = \{ x = (x',x_n)^T \in \mathbb{R}^n: x_n
> \omega(x') \}$ is determined by some height function $\omega \in W^3_\infty(\mathbb{R}^{n-1})$ with $\| \nabla' \omega \|_\infty$ sufficiently small so that a perturbation argument carries over the result for $\mathbb{R}^n_+$ to $H_\omega$.
Afterwards, we localize the domain $\Omega$ such that on a local level
it is reduced to either the whole space or some bent, rotated and shifted half space.

\begin{lemma} \label{thm:HE half space}
Let $n \ge 2$, $1 < q < \infty$ and $0 < \theta < \pi$. Then for $f \in L_q(\mathbb{R}^n_+)^n$, $g \in W^1_q(\mathbb{R}^n_+)^n$, $h \in W^2_q(\mathbb{R}^n_+)^n$ and any $\lambda \in \Sigma_\theta$ there exists a unique solution $u \in W^2_q(\mathbb{R}^n_+)^n$ of
\begin{equation} \label{eq:HE half space}
\left\{
\begin{array}{rll}
\lambda u - \Delta u & = f & \text{in } \mathbb{R}^n_+ \\
\Dm(u)\nu & = \Pi_\tau g & \text{on } \partial \mathbb{R}^n_+ \\
\Pi_\nu u & = \Pi_\nu h & \text{on } \partial \mathbb{R}^n_+
\end{array}
\right.
\end{equation}
such that
\begin{equation} \label{eq:HE half space resolvent estimate}
\| (\lambda u, \sqrt{\lambda} \nabla u, \nabla^2 u) \|_q
\le C \| (f, \sqrt{\lambda} g, \nabla g, \lambda h, \sqrt{\lambda}
\nabla h, \nabla^2 h) \|_q,
\end{equation}
where $C = C(n,q,\theta) > 0$.
\end{lemma}

\begin{proof}
In the half space the outward unit normal vector is
$\nu = (0,\dots,0,-1)^T$. The tangential and normal projections are given by
$\Pi_\tau g = (g^{1},\dots,g^{n-1},0)^T$ resp.\ $\Pi_\nu h = (0,\dots,0,h^{n})^T$.
Then \eqref{eq:HE half space} reads
\begin{equation*}
\left\{
\begin{array}{rll}
\lambda u - \Delta u & = f & \text{in } \mathbb{R}^n_+ \\
\partial_1 u^{n} - \partial_n u^{1} & = g^{1} & \text{on } \partial \mathbb{R}^n_+ \\
\partial_2 u^{n} - \partial_n u^{2} & = g^{2} & \text{on } \partial \mathbb{R}^n_+ \\
 & ~\vdots & \\
\partial_{n-1} u^{n} - \partial_n u^{n-1} & = g^{n-1} & \text{on } \partial \mathbb{R}^n_+ \\
u^{n} & = h^{n} & \text{on } \partial \mathbb{R}^n_+.
\end{array}
\right.
\end{equation*}
Hence, we can first solve the inhomogeneous Dirichlet boundary problem
\begin{equation*}
\left\{
\begin{array}{rll}
\lambda u^{n} - \Delta u^{n} & = f^{n} & \text{in } \mathbb{R}^n_+ \\
u^{n} & = h^{n} & \text{on } \partial \mathbb{R}^n_+
\end{array}
\right.
\end{equation*}
and then, after inserting the solution $u^{n} \in W^2_q(\mathbb{R}^n_+)$, solve the decoupled Neumann boundary problems
\begin{equation*}
\left\{
\begin{array}{rll}
\lambda u^{j} - \Delta u^{j} & = f^{j} & \text{in } \mathbb{R}^n_+ \\
- \partial_n u^{j} & = g^{j} - \partial_j u^{n}  & \text{on } \partial \mathbb{R}^n_+ \\
\end{array}
\right.
\end{equation*}
for $j = 1,\dots,n-1$.
See~\cite{KW04}, Thm.~7.7 and Sec.~7.18 for a detailed treatment of the problems with Dirichlet resp.\ Neumann boundary conditions.
Thus we obtain unique solvability of \eqref{eq:HE half space} as well as estimate \eqref{eq:HE half space resolvent estimate}.
\end{proof}

\begin{theorem} \label{thm:HE bent rotated shifted half space}
Let $Q^T H_\omega + \tau$ be a bent, rotated and shifted half space, i.e., $Q \in \mathbb{R}^{n \times n}$ is a rotation matrix
and $\tau \in \mathbb{R}^n$ is some shifting vector.
Let $\omega \in W^3_\infty(\mathbb{R}^{n-1})$, $n \ge 2$, $1 < q < \infty$ and $0 < \theta < \pi$.
Fix $M \ge 1$ such that
\begin{equation} \label{eq:uniformity 3}
\| \nabla' \omega \|_{\infty}, \| \nabla'^2 \omega \|_{\infty}, \|
\nabla'^3 \omega \|_{\infty} \le M.
\end{equation}
Then there exist $\kappa = \kappa(n,q,\theta) > 0$ and $\lambda_0 = \lambda_0(n,q,\kappa,M) > 0$
such that in case
$\| \nabla' \omega \|_\infty \le \kappa$,
$\lambda \in \Sigma_\theta$, $|\lambda| \ge \lambda_0$ for
$f \in L_q(Q^T H_\omega + \tau)^n$, $g \in W^1_q(Q^T H_\omega + \tau)^n$ and $h \in W^2_q(Q^T H_\omega + \tau)^n$ there exists a unique solution $u \in W^2_q(Q^T H_\omega + \tau)^n$ of
\begin{equation} \label{eq:HE bent rotated shifted half space}
\left\{
\begin{array}{rll}
\lambda u - \Delta u & = f & \text{in } Q^T H_\omega + \tau \\
\Dm(u)\nu & = \Pi_\tau g & \text{on } \partial (Q^T H_\omega + \tau) \\
\Pi_\nu u & = \Pi_\nu h & \text{on } \partial (Q^T H_\omega + \tau)
\end{array}
\right.
\end{equation}
satisfying
\begin{equation} \label{eq:HE bent rotated shifted half space resolvent estimate}
\| (\lambda u, \sqrt{\lambda} \nabla u, \nabla^2 u) \|_q
\le C \| (f, \sqrt{\lambda} g, \nabla g, \lambda h, \sqrt{\lambda}
\nabla h, \nabla^2 h) \|_q,
\end{equation}
where $C = C(n,q,\theta,M) > 0$.
\end{theorem}

In order to prove Theorem~\ref{thm:HE bent rotated shifted half space},
we begin by observing that without loss of generality we may assume
$\tau = 0$. In fact, it is obvious that the shift $x^\tau := x - \tau$ 
leads to an equivalent system on $Q^T H_\omega$.

Next, writing $u^Q(x) := u(Q^T x)$ we obtain that the transformation
$u \mapsto Q u^Q$ is an isomorphism $W^k_q(Q^T H_\omega)^n \overset{\cong}{\longrightarrow} W^k_q(H_\omega)^n$ for $k =0,1,2$.
Furthermore, the behavior of the Laplacian and the boundary terms under this transformation yields that \eqref{eq:HE bent rotated shifted half space} is equivalent to the problem
\begin{equation*}
\left\{
\begin{array}{rll}
\lambda (Q u^Q) - \Delta (Q u^Q) & = Q f^Q & \text{in } H_\omega \\
\Pi_\tau^Q \big( \transA{\nabla (Q u^Q)} - \transB{\nabla (Q u^Q)} \big) Q \nu^Q & = \Pi_\tau^Q Q g^Q & \text{on } \partial H_\omega \\
\Pi_\nu^Q Q u^Q & = \Pi_\nu^Q Q h^Q & \text{on } \partial H_\omega.
\end{array}
\right.
\end{equation*}
Hence, it remains to treat the bent half space problem
\begin{equation} \label{eq:HE bent half space}
\left\{
\begin{array}{rll}
\lambda u - \Delta u & = f & \text{in } H_\omega \\
\Dm(u)\nu & = \Pi_\tau g & \text{on } \partial H_\omega \\
\Pi_\nu u & = \Pi_\nu h & \text{on } \partial H_\omega.
\end{array}
\right.
\end{equation}

To this end, we apply the change of coordinates
$\Phi : H_\omega \overset{\cong}{\longrightarrow} \mathbb{R}^n_+$, $x \mapsto \ti{x}$, given by $\ti{x} = (x',x_n - \omega(x'))^T$
and we write $u \circ \Phi^{-1} =: J_\omega^{-1} u =: \ti{u}$ for a function $u$ on $H_\omega$.
For the partial derivatives we have the following behavior under the
change of coordinates, which yields that
\begin{equation} \label{eq:Isom change of coord}
J_\omega: W^k_q(\mathbb{R}^n_+) \overset{\cong}{\longrightarrow} W^k_q(H_\omega), \quad \ti{u} \mapsto u
\end{equation}
is an isomorphism for $k = 0,1,2$ such that the continuity constants of
$J_\omega$ and $J_\omega^{-1}$ depend on $M$ from \eqref{eq:uniformity 3} and 
on $n$, only:
\begin{itemize}
\item $\ti{\partial_i u} = \partial_i \ti{u} - (\partial_i \omega) \partial_n \ti{u}$ ~for $i = 1,\dots,n-1$.
\item $\ti{\partial_n u} = \partial_n \ti{u}$.
\item $\ti{\partial_j \partial_i u}
= \partial_j \partial_i \ti{u}
- (\partial_j \omega) \partial_i \partial_n \ti{u}
- (\partial_j \partial_i \omega) \partial_n \ti{u}
- (\partial_i \omega) \partial_j \partial_n \ti{u}
+ (\partial_i \omega) (\partial_j \omega) \partial_n^2 \ti{u}$
~for $i,j = 1,\dots,n-1$.
\item $\ti{\partial_n \partial_i u}
= \partial_i \partial_n \ti{u}
+ (\partial_i \omega) \partial_n^2 \ti{u}$
~for $i = 1,\dots,n-1$.
\item $\ti{\partial_n^2 u} = \partial_n^2 \ti{u}$.
\item $\ti{\Delta u} = \Delta \ti{u} - 2(\nabla' \transA{\omega},0) \cdot \nabla \partial_n \ti{u} - (\Delta' \omega) \partial_n \ti{u} + |\nabla' \omega|^2 \partial_n^2 \ti{u}$ when $u$ is a scalar function.
\item $\ti{\transA{\nabla u}} = \nabla \transA{\ti{u}} -
(\nabla' \transA{\omega},0) \partial_n \ti{u}$
when $u$ is a scalar function.
\item $\ti{\transA{\nabla u}} = \transA{\nabla \ti{u}} - E(\ti{u})$
with $E(\ti{u}) = \left( (\partial_j \omega) 
\partial_n \ti{u}^{i} \right)_{i,j = 1,\dots,n}$ when $u$ is a vector field,
where we set $\partial_n \omega := 0$.
\end{itemize}
Hence, we can write
\begin{equation} \label{eq:HE change of coordinates 1}
\ti{(\lambda - \Delta)u} = (\lambda - \Delta)\ti{u} + B \ti{u},
\end{equation}
where $B \ti{u} := 2(\nabla' \transA{\omega},0) \cdot \nabla \partial_n \ti{u} + (\Delta' \omega) \partial_n \ti{u} - |\nabla' \omega|^2 \partial_n^2 \ti{u}$ for a scalar function $\ti{u}$ and we define $B \ti{u}$ componentwise if $\ti{u}$ is a vector field.
For the boundary condition operator we further have
\begin{equation} \label{eq:HE change of coordinates 2}
\ti{\Dm(u)\nu}
= \Dm(\ti{u}) \ti{\nu} + (E(\ti{u})^T - E(\ti{u})) \ti{\nu}.
\end{equation}
From this representation it is easily read off that \eqref{eq:Dm
invariant under tangential projection} still holds for $\ti{\Dm(u)\nu}$, 
i.e., we have $(I - \ti{\nu} \ti{\nu}^T) \ti{\Dm(u)\nu} = \ti{\Dm(u)\nu}$.
Also note that mapping the normal vector $\nu: \partial H_\omega
\rightarrow \mathbb{R}^n$ of the bent half space via the introduced 
change of coordinates to
$\ti{\nu}: \partial \mathbb{R}^n_+ \rightarrow \mathbb{R}^n$ does not
yield the normal vector of the half space. In fact, since $\nu(x)$ does
not depend on the last component $x_n$, which can be seen from 
the concrete representation
\begin{equation} \label{eq:normal representation}
\ti{\nu} = \frac{1}{\sqrt{|\nabla' \omega|^2 + 1}} (\partial_1 \omega,\dots,\partial_{n-1} \omega,-1)^T,
\end{equation}
we can identify $\nu = \ti{\nu}$ and even consider it as a function on the whole space, i.e.,
$\nu = \ti{\nu}: \mathbb{R}^n \rightarrow \mathbb{R}^n$. In this case \eqref{eq:normal representation} gives that
\begin{equation} \label{eq:normal estimate}
\| \nu \|_{2,\infty} \le C_n \| (\nabla' \omega,\nabla'^2 \omega,\nabla'^3 \omega) \|_\infty
\end{equation}
holds with a constant $C_n > 0$ depending only on the space dimension $n$.
We denote the outward unit normal of the half space by $\e := (0,\dots,0,-1)^T$.

The boundary condition
\begin{equation} \label{eq:HE bent half space BC}
\left\{
\begin{array}{rll}
\Dm(u)\nu & = \Pi_\tau g & \text{on } \partial H_\omega \\
\Pi_\nu u & = \Pi_\nu h & \text{on } \partial H_\omega
\end{array}
\right.
\end{equation}
can be written equivalently as
\begin{equation*}
\Dm(u)\nu + \Pi_\nu u = \Pi_\tau g + \Pi_\nu h \quad \text{on } \partial H_\omega
\end{equation*}
due to separation of the tangential and the normal part in \eqref{eq:HE
bent half space BC} and by using \eqref{eq:Dm invariant under tangential projection}.
Now, \eqref{eq:HE change of coordinates 1} and \eqref{eq:HE change of coordinates 2} give that a change of coordinates in \eqref{eq:HE bent half space} yields the equivalent
problem
\begin{equation} \label{eq:HE half space perturbed}
\left\{
\begin{array}{rll}
\lambda \ti{u} - \Delta \ti{u} + B \ti{u} & = \ti{f} & \text{in } \mathbb{R}^n_+ \\
\Dm(\ti{u}) \ti{\nu} + (E(\ti{u})^T - E(\ti{u})) \ti{\nu} + \Pi_\nu \ti{u} & = \Pi_\tau \ti{g} + \Pi_\nu \ti{h} & \text{on } \partial \mathbb{R}^n_+
\end{array}
\right.
\end{equation}
with $\ti{f} \in L_q(\mathbb{R}^n_+)^n$, $\ti{g} \in W^1_q(\mathbb{R}^n_+)^n$ and $\ti{h} \in W^2_q(\mathbb{R}^n_+)^n$.

We apply the matrix
$\transA{\nabla \Phi} = I - ((\partial_j \omega) \delta_{in})_{i,j = 1,\dots,n}$
to the boundary condition of \eqref{eq:HE half space perturbed}. The matrix $\transA{\nabla \Phi}$ satisfies $\det \transA{\nabla \Phi} = 1$, $(\transA{\nabla \Phi})^{-1} = 2 I - \transA{\nabla \Phi}$ and it maps the tangent space of any point $x \in \partial H_\omega$ into the tangent space of $\partial \mathbb{R}^n_+$.
Therefore we have:
\begin{itemize}
\item $(\transA{\nabla \Phi}) \Dm(\ti{u}) \nu
= (I - \e \e^T) (\transA{\nabla \Phi}) \Dm(\ti{u}) \nu \vspace{2pt} \\
= (I - \e \e^T) \Dm(\ti{u}) \e + (I - \e \e^T) (\transA{\nabla \Phi} - I) \Dm(\ti{u}) \e + \vspace{2pt} \\
(I - \e \e^T) (\transA{\nabla \Phi}) \Dm(\ti{u}) (\nu - \e) \vspace{2pt} \\
= \Dm(\ti{u}) \e + (I - \e \e^T) (\transA{\nabla \Phi} - I) \Dm(\ti{u}) \e + \vspace{2pt} \\
(I - \e \e^T) (\transA{\nabla \Phi}) \Dm(\ti{u}) (\nu - \e)$.
\item $(\transA{\nabla \Phi}) \Pi_\nu \ti{u} \vspace{2pt} \\
= \e \e^T \ti{u} + (I - \e \e^T)((\transA{\nabla \Phi}) \nu \nu^T - \e \e^T) \ti{u} + \e \e^T ((\transA{\nabla \Phi}) \nu \nu^T - \e \e^T) \ti{u}$.
\item $(\transA{\nabla \Phi}) (E(\ti{u})^T - E(\ti{u})) \nu
= (I - \e \e^T) (\transA{\nabla \Phi}) (E(\ti{u})^T - E(\ti{u})) \nu$.
\item $(\transA{\nabla \Phi}) \Pi_\tau \ti{g} = (I - \e \e^T) (\transA{\nabla \Phi}) (I - \nu \nu^T) \ti{g}$.
\item $(\transA{\nabla \Phi}) \Pi_\nu \ti{h} = \e \e^T (\transA{\nabla \Phi}) \nu \nu^T \ti{h} + (I - \e \e^T) (\transA{\nabla \Phi}) \nu \nu^T \ti{h}$.
\end{itemize}
Hence, \eqref{eq:HE half space perturbed} becomes
\begin{equation} \label{eq:HE half space perturbed 2}
\left\{
\begin{array}{rll}
\lambda \ti{u} - \Delta \ti{u} + B \ti{u} & = \ti{f} & \text{in } \mathbb{R}^n_+ \\
\Dm(\ti{u}) \e + \e \e^T \ti{u} + B_\gamma \ti{u} & = (I - \e \e^T) \ti{G} + \e \e^T \ti{H} & \text{on } \partial \mathbb{R}^n_+
\end{array}
\right.
\end{equation}
where
\begin{equation*}
\begin{split}
B_\gamma \ti{u} :=
& (I - \e \e^T) [ (\transA{\nabla \Phi} - I) \Dm(\ti{u}) \e + (\transA{\nabla \Phi}) \Dm(\ti{u}) (\nu - \e) \\
& + ((\transA{\nabla \Phi}) \nu \nu^T - \e \e^T) \ti{u} + (\transA{\nabla \Phi}) (E(\ti{u})^T - E(\ti{u})) \nu] \\
& + \e \e^T ((\transA{\nabla \Phi}) \nu \nu^T - \e \e^T) \ti{u}
\end{split}
\end{equation*}
and
\begin{equation} \label{eq:G and H}
\begin{split}
& \ti{G} := (\transA{\nabla \Phi}) (I - \nu \nu^T) \ti{g} + (\transA{\nabla \Phi}) \nu \nu^T \ti{h}
\in W^1_q(\mathbb{R}^n_+)^n , \\
& \ti{H} := (\transA{\nabla \Phi}) \nu \nu^T \ti{h}
\in W^2_q(\mathbb{R}^n_+)^n.
\end{split}
\end{equation}
We see that $\ti{G}$ and $\ti{H}$ are the new right-hand side 
functions in the boundary condition.
\begin{remark}
Note that, thanks to the presence of $\transA{\nabla \Phi}$, 
the functions $\ti{G}$ and $\ti{H}$ belong to the desired
regularity classes. 
Also note that, concerning the Stokes equations, the 
matrix $\transA{\nabla \Phi}$ is already utilized in~\cite{Sol77}
(see also \cite{NS03}).
\end{remark}

Summarizing, the proof of Theorem~\ref{thm:HE bent rotated shifted half space} is reduced to the following perturbed version of Lemma~\ref{thm:HE half space}.

\begin{lemma} \label{thm:HE half space perturbed}
Let $n \ge 2$, $1 < q < \infty$, $0 < \theta < \pi$, $\omega \in W^3_\infty(\mathbb{R}^{n-1})$ and let $M > 0$ such that \eqref{eq:uniformity 3} holds.
Then there exist $\kappa = \kappa(n,q,\theta) > 0$ and $\lambda_0 = \lambda_0(n,q,\kappa,M) > 0$ such that in case $\| \nabla' \omega \|_\infty \le \kappa$, $\lambda \in \Sigma_\theta$, $|\lambda| \ge \lambda_0$
for $\ti{f} \in L_q(\mathbb{R}^n_+)^n$, $\ti{G} \in W^1_q(\mathbb{R}^n_+)^n$ and $\ti{H} \in W^2_q(\mathbb{R}^n_+)^n$
there exists a unique solution $\ti{u} \in W^2_q(\mathbb{R}^n_+)^n$ of \eqref{eq:HE half space perturbed 2}
satisfying
\begin{equation} \label{eq:HE half space perturbed resolvent estimate}
\| (\lambda \ti{u}, \sqrt{\lambda} \nabla \ti{u}, \nabla^2 \ti{u}) \|_q
\le C \| (\ti{f}, \sqrt{\lambda} \ti{G}, \nabla \ti{G}, \lambda \ti{H}, \sqrt{\lambda} \nabla \ti{H}, \nabla^2 \ti{H}) \|_q,
\end{equation}
where $C = C(n,q,\theta,M) > 0$.
\end{lemma}

\begin{proof}
We prove the statement using a perturbation argument via the Neumann series, where the version we make use of is~\cite{KW04}, Lem.\ 7.10.
Therefore we define the spaces
\begin{equation*}
\begin{split}
X & := W^2_q(\mathbb{R}^n_+)^n, \\
Y & := L_q(\mathbb{R}^n_+)^n \times
\big\{ (I - \e \e^T) \ti{G} + \e \e^T \ti{H} : \ti{G} \in W^1_q(\mathbb{R}^n_+)^n, \ti{H} \in W^2_q(\mathbb{R}^n_+)^n \big\} \\
& ~= L_q(\mathbb{R}^n_+)^n \times
\big\{ (\ti{G}^1,\dots,\ti{G}^{n-1},\ti{H}^n)^T : \ti{G} \in W^1_q(\mathbb{R}^n_+)^n, \ti{H} \in W^2_q(\mathbb{R}^n_+)^n \big\}, \\
Z & := L_q(\mathbb{R}^n_+)^n \times L_q(\partial \mathbb{R}^n_+)^n
\end{split}
\end{equation*}
with norms (depending on $\lambda \in \Sigma_\theta$)
\begin{equation*}
\begin{split}
\| \ti{u} \|_X
&:= \| (\lambda \ti{u},\sqrt{\lambda} \nabla \ti{u},\nabla^2 \ti{u}) \|_q, \\
\| (\ti{f},(I - \e \e^T) \ti{G} + \e \e^T \ti{H}) \|_Y
& := \| (\ti{f}, \sqrt{\lambda} \ti{G}^1, \dots, \sqrt{\lambda} \ti{G}^{n-1}, \nabla \ti{G}^1, \dots \\
& \qquad \dots, \nabla \ti{G}^{n-1}, \lambda \ti{H}^n, \sqrt{\lambda} \nabla \ti{H}^n, \nabla^2 \ti{H}^n) \|_q, \\
\| \cdot \|_Z
&:= \| \cdot \|_{L_q(\mathbb{R}^n_+)^n \times L_q(\partial \mathbb{R}^n_+)^n}
\end{split}
\end{equation*}
as well as the continuous linear operators
\begin{equation*}
\begin{split}
S: X & \longrightarrow Y, \quad
\ti{u} \mapsto ((\lambda - \Delta) \ti{u},\Dm(\ti{u}) \e + \e \e^T \ti{u}), \\
P: X & \longrightarrow Y, \quad
\ti{u} \mapsto (B \ti{u}, B_\gamma \ti{u}), \\
Q: Y & \longrightarrow Z, \quad
(\ti{f},\ti{k}) \mapsto (\ti{f},\Tr_{\partial \mathbb{R}^n_+} \ti{k}).
\end{split}
\end{equation*}
By standard arguments we obtain that the space $Y$ is complete so $X$, $Y$ and $Z$ are Banach spaces.
Due to Lemma~\ref{thm:HE half space}, for any $(\ti{f},\ti{k}) \in Y$ there exists a unique $\ti{u} \in X$ satisfying
$QS \ti{u} = Q (\ti{f},\ti{k})$
and there exists $C = C(n,q,\theta) > 0$ such that
\begin{equation}
\| \ti{u} \|_X \le C \| (\ti{f},\ti{k}) \|_Y.
\end{equation}
We now aim to show that we can choose $\lambda_0 = \lambda_0(n,q,\kappa,M) > 0$ and a constant $C' = C'(n,M) > 0$ such that for $\lambda \in \Sigma_\theta$, $|\lambda| \ge \lambda_0$ and $\| \nabla' \omega \|_\infty \le \kappa < 1$ we have
\begin{equation} \label{eq:HE perturbation estimate}
\| P \ti{u} \|_Y \le C' \kappa \| \ti{u} \|_X
\end{equation}
for all $\ti{u} \in X$.
Then, prescribing $\kappa < \frac{1}{2 C C'}$, 
we deduce
$\| P \|_{X \rightarrow Y} \le \frac{1}{2C}$
and as a consequence (see~\cite{KW04}, Lem.\ 7.10) we receive:
For any $(\ti{f},\ti{k}) \in Y$ there exists a unique $\ti{u} \in X$ satisfying
\begin{equation*}
Q(S + P) \ti{u} = Q (\ti{f},\ti{k})
\end{equation*}
and we have
\begin{equation*}
\| \ti{u} \|_X \le 2C \| (\ti{f},\ti{k}) \|_Y.
\end{equation*}
This is exactly the claim of the lemma.

It remains to prove \eqref{eq:HE perturbation estimate}.
For this purpose, we assume $M \ge 1$, $\kappa < 1$ and $\lambda_0 \ge \frac{M^2}{\kappa^2}$.
Let $\lambda \in \Sigma_\theta$, $|\lambda| \ge \lambda_0$, $\| \nabla' \omega \|_\infty \le \kappa$ and $\ti{u} \in X$.
Then, for the operator $B$, we have
\begin{equation*}
\begin{split}
\| B \ti{u} \|_q
& = \Big\| \left( 2(\nabla' \transA{\omega},0) \cdot \nabla \partial_n \ti{u}^{j} + (\Delta' \omega) \partial_n \ti{u}^{j} - |\nabla' \omega|^2 \partial_n^2 \ti{u}^{j} \right)_{j=1,\dots,n} \Big\|_q \\
& \le C \Big( \kappa \| \nabla^2 \ti{u} \|_q + \frac{M}{\sqrt{|\lambda|}} \| \sqrt{\lambda} \nabla \ti{u} \|_q + \kappa^2 \| \nabla^2 \ti{u} \|_q \Big) \\
& \le C' \kappa \| \ti{u} \|_X
\end{split}
\end{equation*}
with some constants $C = C(n) > 0$ and $C' = C'(n,M) > 0$.
For the operator $B_\gamma$, we have
\begin{equation} \label{eq:HE perturbation boundary}
\begin{split}
\| B_\gamma \ti{u} \|_{Y_2}
\le & \Big\| \sqrt{\lambda} \big[ (\transA{\nabla \Phi} - I) \Dm(\ti{u}) \e + (\transA{\nabla \Phi}) \Dm(\ti{u}) (\nu - \e) \\
& + ((\transA{\nabla \Phi}) \nu \nu^T - \e \e^T) \ti{u} + (\transA{\nabla \Phi}) (E(\ti{u})^T - E(\ti{u})) \nu \big] \Big\|_q \\
+ & \Big\| \nabla \big[ (\transA{\nabla \Phi} - I) \Dm(\ti{u}) \e + (\transA{\nabla \Phi}) \Dm(\ti{u}) (\nu - \e) \\
& + ((\transA{\nabla \Phi}) \nu \nu^T - \e \e^T) \ti{u} + (\transA{\nabla \Phi}) (E(\ti{u})^T - E(\ti{u})) \nu \big] \Big\|_q \\
+ & \Big\| \lambda \big[ ((\transA{\nabla \Phi}) \nu \nu^T - \e \e^T) \ti{u} \big] \Big\|_q \\
+ & \Big\| \sqrt{\lambda} \nabla \big[ ((\transA{\nabla \Phi}) \nu \nu^T - \e \e^T) \ti{u} \big] \Big\|_q \\
+ & \Big\| \nabla^2 \big[ ((\transA{\nabla \Phi}) \nu \nu^T - \e \e^T) \ti{u} \big] \Big\|_q.
\end{split}
\end{equation}
Now each of the summands in \eqref{eq:HE perturbation boundary} can be estimated by $C' \kappa \| \ti{u} \|_X$ with a constant $C' = C'(n,M) > 0$,
where all of the estimates can be done in a similar way.
One only has to keep in mind that $\| \nu - \e \|_\infty$ and $\|
\transA{\nabla \Phi} - I \|_\infty$ can be estimated by $\kappa$ up to a constant depending only on $n$, as well as \eqref{eq:uniformity 3} and the condition $\lambda_0 \ge \frac{M^2}{\kappa^2}$.
We exemplarily treat two of the terms in \eqref{eq:HE perturbation boundary}:
{\allowdisplaybreaks
\begin{align*}
& \Big\| \nabla \big[ (\transA{\nabla \Phi}) (E(\ti{u})^T - E(\ti{u})) \nu \big] \Big\|_q \\
& \le C'(n) \Big( \| (\nabla^2 \Phi) (E(\ti{u})^T - E(\ti{u})) \nu \|_q + \sum_{k=1}^n \| (\transA{\nabla \Phi}) [\partial_k (E(\ti{u})^T - E(\ti{u}))] \nu \|_q \\
& \qquad \qquad + \| (\transA{\nabla \Phi}) (E(\ti{u})^T - E(\ti{u})) \nabla \nu \|_q \Big) \\
& \le C'(n,M) \Big( \kappa \| \nabla \ti{u} \|_q
+ \kappa \| \nabla^2 \ti{u} \|_q + \| \nabla \ti{u} \|_q + \kappa \| \nabla \ti{u} \|_q \Big) \\
& \le C'(n,M) \bigg( \frac{1}{\sqrt{|\lambda|}} \| \sqrt{\lambda} \nabla \ti{u} \|_q + \kappa \| \nabla^2 \ti{u} \|_q \bigg) \\
& \le C'(n,M) \kappa \Big( \| \sqrt{\lambda} \nabla \ti{u} \|_q + \| \nabla^2 \ti{u} \|_q \Big) \\
& \le C'(n,M) \kappa \| \ti{u} \|_X,
\end{align*}
and
\begin{align*}
& \Big\| \nabla^2 \big[ ((\transA{\nabla \Phi}) \nu \nu^T - \e \e^T) \ti{u} \big] \Big\|_q \\
& \le C'(n) \Big( \frac{1}{|\lambda|} \| \nabla^2 ((\nabla \Phi^T) \nu \nu^T - \e \e^T) \|_\infty \| \lambda \ti{u} \|_q \qquad \qquad \qquad \qquad \\
& \qquad \qquad + \frac{1}{\sqrt{|\lambda|}} \| \nabla ((\nabla \Phi^T) \nu \nu^T - \e \e^T) \|_\infty \| \sqrt{\lambda} \nabla \ti{u} \|_q \\
& \qquad \qquad + \| (\nabla \Phi^T) (\nu \nu^T - \e \e^T) \|_\infty \| \nabla^2 \ti{u} \|_q \\
& \qquad \qquad + \| (\nabla \Phi^T - I) \e \e^T \|_\infty \| \nabla^2 \ti{u} \|_q \Big) \\
& \le C'(n,M) \kappa \| \ti{u} \|_X.
\end{align*}
}%
Hence, \eqref{eq:HE perturbation estimate} is verified.
\end{proof}

\begin{proof}[Proof of Theorem~\thref{thm:HE bent rotated shifted half space}]
For $f \in L_q(H_\omega)^n$, $g \in W^1_q(H_\omega)^n$ and $h \in W^2_q(H_\omega)^n$ we have
$\ti{f} \in L_q(\mathbb{R}^n_+)^n$, $\ti{g} \in W^1_q(\mathbb{R}^n_+)^n$ and $\ti{h} \in W^2_q(\mathbb{R}^n_+)^n$ and we define
$\ti{G} \in W^1_q(\mathbb{R}^n_+)^n$ and
$\ti{H} \in W^2_q(\mathbb{R}^n_+)^n$
as in \eqref{eq:G and H}.
We choose $\kappa$ and $\lambda_0$ 
as in Lemma~\ref{thm:HE half space perturbed}.
Then for $\lambda \in \Sigma_\theta$, $|\lambda| \ge \lambda_0$ and $\| \nabla' \omega \|_\infty \le \kappa$ there exists a unique solution
$\ti{u} \in W^2_q(\mathbb{R}^n_+)^n$ of \eqref{eq:HE half space perturbed 2},
satisfying \eqref{eq:HE half space perturbed resolvent estimate}.
The calculations above give that $u = J_\omega \ti{u}$ is the unique solution of \eqref{eq:HE bent half space}.

Now, assuming $|\lambda| \ge 1$, the isomorphism \eqref{eq:Isom change of coord} gives that $u = J_\omega \ti{u}$ fulfills
\begin{equation} \label{eq:res bent half space 1}
\| (\lambda u, \sqrt{\lambda} \nabla u, \nabla^2 u) \|_{q,H_\omega}
\le C \| (\lambda \ti{u}, \sqrt{\lambda} \nabla \ti{u}, \nabla^2 \ti{u}) \|_{q,\mathbb{R}^n_+},
\end{equation}
where $C = C(n,M) > 0$ and on the other hand
\begin{equation} \label{eq:res bent half space 2}
\| (\ti{f}, \sqrt{\lambda} \ti{g}, \nabla \ti{g}, \lambda \ti{h}, \sqrt{\lambda} \nabla \ti{h}, \nabla^2 \ti{h}) \|_{q,\mathbb{R}^n_+}
\le C \| (f, \sqrt{\lambda} g, \nabla g, \lambda h, \sqrt{\lambda} \nabla h, \nabla^2 h) \|_{q,H_\omega}.
\end{equation}
Based on representation \eqref{eq:G and H}, by taking into account
\eqref{eq:normal estimate}, \eqref{eq:Isom change of coord} and by 
assuming $|\lambda| \ge 1$ again, we further obtain
\begin{equation} \label{eq:res bent half space 3}
\begin{split}
\| & (\ti{f}, \sqrt{\lambda} \ti{G}, \nabla \ti{G}, \lambda \ti{H}, \sqrt{\lambda} \nabla \ti{H}, \nabla^2 \ti{H}) \|_{q,\mathbb{R}^n_+} \\
& \le C \| (\ti{f}, \sqrt{\lambda} \ti{g}, \nabla \ti{g}, \lambda \ti{h}, \sqrt{\lambda} \nabla \ti{h}, \nabla^2 \ti{h}) \|_{q,\mathbb{R}^n_+},
\end{split}
\end{equation}
where $C = C(n,M) > 0$.
Now \eqref{eq:HE half space perturbed resolvent estimate}, \eqref{eq:res bent half space 1}, \eqref{eq:res bent half space 2} and \eqref{eq:res bent half space 3} yield
\eqref{eq:HE bent rotated shifted half space resolvent estimate}.
\end{proof}

We turn to the proof of Theorem~\ref{thm:HE with PS}.
To this end, we follow a localization procedure that applies to non-compact
boundaries and which, for instance, is also utilized in \cite{Kun03}.
For the (countably many) parameters $l \in \Gamma$ we multiply
\eqref{eq:HE} by the smooth cut-off functions $\varphi_l$ as introduced
in \eqref{eq:Partition of unity} and \eqref{eq:uniformity partition of
unity}. 
This leads to a system of local equations (one equation for each $l \in
\Gamma$) with a sequence $(u_l)_{l \in \Gamma}$ 
of the form $u_l = \varphi_l u$ as the potential solution.
In order to receive such a system of local equations, we make use of the matrix identity
\begin{equation} \label{eq:Gradient matrix identity}
\transA{\nabla (\varphi u)} = u \transA{\nabla \varphi} + \varphi \transA{\nabla u}
\end{equation}
and the vector identity
\begin{equation} \label{eq:Laplace vector identity}
\Delta(\varphi u)
= (\Delta \varphi)u + 2 (\transA{\nabla u}) \transB{\nabla \varphi} + \varphi \Delta u
\end{equation}
for scalar functions $\varphi$ and vector fields $u$.
Introducing a suitable Banach space $X$ for the sequence $(u_l)_{l \in
\Gamma}$ as well as a Banach space $Y$ related to the right-hand sides 
of the local equations, the purpose is to obtain unique solvability on a
local level. Finally, the well-posedness shall be carried over 
to the original problem \eqref{eq:HE}.
Compared to \cite{Kun03}, where Dirichlet boundary conditions are
considered, the localization of the boundary conditions here is a bit
more intricate. 

In the sequel the space of $q$-summable sequences in a Banach space 
$X$ we denote by $l_q(X)$.
In case each element of the sequence shall be allowed to belong to a
different Banach space $X_i$, we write $l_q(\bigoplus_{i \in I} X_i)$,
where $I$ is a countable index set. 
Furthermore, in case $X_i$ is a function space $F(\Omega_i)$ or $F(\partial \Omega_i)$ of functions on some domain $\Omega_i$ or on its boundary $\partial \Omega_i$ (e.g., $F = W^k_q$ for $k \in \mathbb{N}_0$ and $1 \le q \le \infty$), we often write $\| \cdot \|_{l_q(F)}$ for the norm in $l_q(\bigoplus_{i \in I} X_i)$.

\begin{proof}[Proof of Theorem~\thref{thm:HE with PS}]
Due to \eqref{eq:Dm invariant under tangential projection}
we can rewrite \eqref{eq:HE} as
\begin{equation} \label{eq:HE 2}
\left\{
\begin{array}{rll}
\lambda u - \Delta u & = f & \text{in } \Omega \\
\Dm(u)\nu + \Pi_\nu u & = \Pi_\tau g + \Pi_\nu h & \text{on } \partial \Omega.
\end{array}
\right.
\end{equation}
The Banach space for the boundary functions in \eqref{eq:HE 2} is
defined as
\begin{equation*}
\begin{split}
& \BFq = \BF_{q,\lambda}(\partial \Omega) \\
& := \big\{ a \in L_q(\partial \Omega)^n : a = \Pi_\tau \Tr g + \Pi_\nu \Tr h, ~g \in W^1_q(\Omega)^n, ~h \in W^2_q(\Omega)^n \big\},
\end{split}
\end{equation*}
with norm
\begin{equation*}
\begin{split}
\| a \|_{\BF_{q,\lambda}(\partial \Omega)}
:= \inf_{g,h} \| (\sqrt{\lambda} g, \nabla g, \lambda h, \sqrt{\lambda} \nabla h, \nabla^2 h) \|_q,
\end{split}
\end{equation*}
where the infimum runs over all
$g \in W^1_q(\Omega)^n$, $h \in W^2_q(\Omega)^n$ such that $a = \Pi_\tau \Tr g + \Pi_\nu \Tr h$.
For $\lambda = 1$ the space $\BFq$ is therefore equipped with the natural norm for the range of the continuous linear operator $T: W^1_q(\Omega)^n \times W^2_q(\Omega)^n \rightarrow L_q(\partial \Omega)^n$, $(g,h) \mapsto \Pi_\tau \Tr g + \Pi_\nu \Tr h$.
We allow arbitrary $\lambda \in \Sigma_\theta$ in the definition of $\| \cdot \|_{\BF_{q,\lambda}(\partial \Omega)}$, since we will need this for a perturbation argument later on.

\emph{Step 1: Local coordinates.}
For the sake of consistent notation we put
\begin{equation*}
\Omega_l :=
\begin{cases}
H_l, & l \in \Gamma_1 \\
\mathbb{R}^n, & l \in \Gamma_0.
\end{cases}
\end{equation*}
Thus, by the space $\BF_q(\partial \Omega_l)$ we mean $\BF_q(\partial \Omega_l) = \BF_q(\partial H_l)$ for $l \in \Gamma_1$ and $\BF_q(\partial \Omega_l) := \{ 0 \}$ for $l \in \Gamma_0$.
We introduce the Banach spaces
\begin{equation*}
\begin{split}
X & := l_q\biggl(\bigoplus_{l \in \Gamma} W^2_q(\Omega_l)^n\biggr), \\
Y & := l_q\biggl(\bigoplus_{l \in \Gamma} L_q(\Omega_l)^n\biggr) 
\times l_q\biggl(\bigoplus_{l \in \Gamma} \BF_q(\partial \Omega_l)\biggr)
\end{split}
\end{equation*}
with ($\lambda$-dependend) norms 
\begin{equation*}
\begin{split}
\| (u_l)_{l \in \Gamma} \|_X
& := \| (\lambda u_l, \sqrt{\lambda} \nabla u_l, \nabla^2 u_l)_{l \in \Gamma} \|_{l_q(L_q)}, \\
\| (f_l,a_l)_{l \in \Gamma} \|_Y
& := \| (f_l)_{l \in \Gamma} \|_{l_q(L_q)} + \| (a_l)_{l \in \Gamma}
\|_{l_q(\BF_{q,\lambda})}.
\end{split}
\end{equation*}
Furthermore, we define the linear and continuous operator
\begin{equation*}
S: X \longmapsto Y, \quad
(u_l)_{l \in \Gamma}
\longmapsto \big( (\lambda - \Delta) u_l, \Tr_{\partial \Omega_l} \Dm(u_l) \nu_l + \nu_l \nu_l^T \Tr_{\partial \Omega_l} u_l \big)_{l \in \Gamma},
\end{equation*}
where we set $\Tr_{\partial \Omega_l} \Dm(u_l) \nu_l + \nu_l \nu_l^T \Tr_{\partial \Omega_l} u_l := 0$ in case $l \in \Gamma_0$.

For the bent, rotated and shifted half space $H_l = Q_l^T H_{\omega_l} + \tau_l$, $l \in \Gamma_1$ and the related constant $M \ge 1$ from \eqref{eq:uniformity},
let initially $\kappa = \kappa(n,q,\theta) > 0$ and $\lambda_0 = \lambda_0(n,q,\kappa,M) > 0$ such that the conditions of Theorem~\ref{thm:HE bent rotated shifted half space} are satisfied.
We further assume $\kappa < 1$ and $\lambda_0 \ge \frac{M^2}{\kappa^2}$.
Let $\lambda \in \Sigma_\theta$, $|\lambda| \ge \lambda_0$ and note that \eqref{eq:littleness gradient} gives $\| \nabla' \omega_l \|_\infty \le \kappa$ for all $l \in \Gamma_1$.
Theorem~\ref{thm:HE bent rotated shifted half space} then implies that 
\begin{equation} \label{eq:S isomorphism}
S: X \overset{\cong}{\longrightarrow} Y
\end{equation}
is an isomorphism and that the continuity constants 
of $S$ and $S^{-1}$ depend on $q,n,\theta$ and $M$ only.

To see this, pick $(f_l,a_l)_{l \in \Gamma} \in Y$. Then, for all $l \in \Gamma_1$ Theorem~\ref{thm:HE bent rotated shifted half space} yields a unique $u_l \in W^2_q(H_l)^n$ such that $(\lambda - \Delta) u_l = f_l$ and $\Tr_{\partial H_l} \Dm(u_l) \nu_l + \nu_l \nu_l^T \Tr_{\partial H_l} u_l = a_l$.
For $l \in \Gamma_0$, existence and uniqueness of the solution $u_l \in
W^2_q(\mathbb{R}^n)^n$ to $(\lambda - \Delta) u_l = f_l$ is clear.
In addition, there is a constant $C = C(n,q,\theta,M) > 0$ such that
\begin{equation} \label{eq:resolvent estimate local}
\| (\lambda u_l, \sqrt{\lambda} \nabla u_l, \nabla^2 u_l) \|_{q,H_l}
\le C \| (f_l, \sqrt{\lambda} g_l, \nabla g_l, \lambda h_l, \sqrt{\lambda} \nabla h_l, \nabla^2 h_l) \|_{q,H_l}
\end{equation}
for all $l \in \Gamma$ and $g_l \in W^1_q(H_l)^n$, $h_l \in W^2_q(H_l)^n$ such
that $a_l = \Pi_\tau \Tr_{\partial H_l} g_l + \Pi_\nu \Tr_{\partial H_l}
h_l$, where we put $g_l = h_l = 0$ for all $l \in \Gamma_0$.
Consequently, for all $l \in \Gamma$ we have
\begin{equation} \label{eq:resolvent estimate local 2}
\| (\lambda u_l, \sqrt{\lambda} \nabla u_l, \nabla^2 u_l) \|_{q,\Omega_l}
\le C \big( \| f_l \|_{q,\Omega_l} + \| a_l \|_{\BF_{q,\lambda}(\partial \Omega_l)} \big).
\end{equation}
This gives
\begin{equation} \label{eq:Localization resolvent estimate}
\begin{split}
\| (u_l)_{l \in \Gamma} \|_X^q
& = \sum_{l \in \Gamma} \| (\lambda u_l, \sqrt{\lambda} \nabla u_l, \nabla^2 u_l) \|_{q,\Omega_l}^q \\
& \le C^q \sum_{l \in \Gamma} \big( \| f_l \|_{q,\Omega_l} + \| a_l \|_{\BF_{q,\lambda}(\partial \Omega_l)} \big)^q \\
& \le C_S^q \| (f_l,a_l)_{l \in \Gamma} \|_Y^q,
\end{split}
\end{equation}
where $C_S = C_S(n,q,\theta,M) > 0$.
Conversely, it is not hard to see that we have
\begin{equation*}
\| S (u_l)_{l \in \Gamma} \|_Y \le C' \| (u_l)_{l \in \Gamma} \|_X
\end{equation*}
for every $(u_l)_{l \in \Gamma} \in X$ with $C' = C'(n,q) > 0$. 
Hence, \eqref{eq:S isomorphism} is verified.

\emph{Step 2: Localizing \eqref{eq:HE}.}
We now multiply \eqref{eq:HE 2} by the functions $\varphi_l$, $l \in \Gamma$ in order to receive corresponding local equations.
Writing $u_m = \varphi_m u$ and using \eqref{eq:Partition of unity}, we have
\begin{equation*}
\begin{split}
& \varphi_l (\lambda - \Delta) u \\
& = (\lambda - \Delta)(\varphi_l u) + 2 (\transA{\nabla u}) \transB{\nabla \varphi_l} + (\Delta \varphi_l) u \\
& = (\lambda - \Delta)(\varphi_l u) + 2 \transA{ \Big( \nabla \sum_{m \in \Gamma} \varphi_m^2 u \Big) } \transB{\nabla \varphi_l} + (\Delta \varphi_l) \sum_{m \in \Gamma} \varphi_m^2 u \\
& = (\lambda - \Delta) u_l + \sum_{m \sim l} \big[ 2 u_m (\transA{\nabla \varphi_m}) \transB{\nabla \varphi_l} + 2 \varphi_m (\transA{\nabla u_m}) \transB{\nabla \varphi_l} + (\Delta \varphi_l) \varphi_m u_m \big].
\end{split}
\end{equation*}
For the tangential boundary condition in \eqref{eq:HE 2} we obtain (note that $\nu = \nu_l$ on $\Supp(\varphi_l)$ for $l \in \Gamma_1$), using \eqref{eq:Partition of unity}, \eqref{eq:Gradient matrix identity} and writing $u_m = \varphi_m u$ again,
\begin{equation*}
\begin{split}
\varphi_l \Dm(u) \nu
& = (\varphi_l \transA{\nabla u} - \varphi_l \transB{\nabla u}) \nu_l \\
& = (\transA{\nabla u_l} - \transB{\nabla u_l}) \nu_l - u (\transA{\nabla \varphi_l}) \nu_l + (\transB{\nabla \varphi_l}) u^T \nu_l \\
& = \Dm(u_l) \nu_l - \sum_{m \in \Gamma} \varphi_m^2 u (\transA{\nabla \varphi_l}) \nu_l + \sum_{m \in \Gamma} \varphi_m^2 (\transB{\nabla \varphi_l}) u^T \nu_l \\
& = \Dm(u_l) \nu_l - \sum_{m \sim l} \varphi_m \big[ u_m \transA{\nabla
\varphi_l} - (u_m \transA{\nabla \varphi_l})^T \big] \nu_l. \\
\end{split}
\end{equation*}
For the normal boundary condition we have
\begin{equation*}
\varphi_l \Pi_\nu u
= \varphi_l \nu_l \nu_l^T u
= \nu_l \nu_l^T u_l
\end{equation*}
for $l \in \Gamma_1$.
Summarizing, multiplying \eqref{eq:HE 2} by $\varphi_l$ for $l \in \Gamma$ 
yields the local equations
\begin{equation} \label{eq:HE local}
\begin{cases}
\lambda u_l - \Delta u_l + \sum_{m \sim l} \big[ 2 u_m (\transA{\nabla \varphi_m}) \transB{\nabla \varphi_l} + \! 2 \varphi_m (\transA{\nabla u_m}) \transB{\nabla \varphi_l} + \! (\Delta \varphi_l) \varphi_m u_m \big] \\
= f_l \quad
\text{in } \Omega_l \text{ for all } l \in \Gamma, \\
\Dm(u_l) \nu_l + \nu_l \nu_l^T u_l - \sum_{m \approx l} \varphi_m \big[ u_m \transA{\nabla \varphi_l} - (u_m \transA{\nabla \varphi_l})^T \big] \nu_l \\
= (I - \nu_l \nu_l^T) g_l + \nu_l \nu_l^T h_l \quad
\text{on } \partial \Omega_l \text{ for all } l \in \Gamma_1.
\end{cases}
\end{equation}
Therefore, we define the perturbation operator
$P: X \longrightarrow Y$ by
\begin{equation*}
\begin{split}
(u_l)_{l \in \Gamma}
\longmapsto \Big( &
\sum_{m \sim l} \big[ 2 u_m (\transA{\nabla \varphi_m}) \transB{\nabla \varphi_l} + 2 \varphi_m (\transA{\nabla u_m}) \nabla \varphi_l + (\Delta \varphi_l) \varphi_m u_m \big], \\
& - \Tr_{\partial \Omega_l} \sum_{m \approx l} \varphi_m \big[ u_m \transA{\nabla \varphi_l} - (u_m \transA{\nabla \varphi_l})^T \big] \nu_l
\Big)_{l \in \Gamma},
\end{split}
\end{equation*}
where in case $l \in \Gamma_0$ we set $\Tr_{\partial \Omega_l} \sum_{m \approx l} \varphi_m \big[ u_m \transA{\nabla \varphi_l} - (u_m \transA{\nabla \varphi_l})^T \big] \nu_l := 0$.

\emph{Step 3: Well-posedness of local equations.}
We now aim to verify that there exists $C_P = C_P(n,q,\Omega) > 0$ such that
\begin{equation} \label{eq:Localization perturbation estimate}
\| P (u_l)_{l \in \Gamma} \|_Y
\le \frac{C_P}{\sqrt{|\lambda|}} \| (u_l)_{l \in \Gamma} \|_X
\end{equation}
for all $(u_l)_{l \in \Gamma} \in X$ and for $\lambda \in \Sigma_\theta$, $|\lambda| \ge \lambda_0$.
For this purpose, let $(u_l)_{l \in \Gamma} \in X$. Then for all $l \in \Gamma$ we have, using \eqref{eq:uniformity partition of unity},
\begin{equation} \label{eq:Perturbation estimate exemplarily}
\begin{split}
\Big\| \sum_{m \sim l} 2 u_m (\transA{\nabla \varphi_m}) \transB{\nabla \varphi_l} \Big\|_{q,\Omega_l}^q
& \le C \sum_{m \sim l} \int_{\Omega_l \cap B_l \cap B_m} \Big| u_m (\transA{\nabla \varphi_m}) \transB{\nabla \varphi_l} \Big|^q \D \lambda_n \\
& = C \sum_{m \sim l} \int_{\Omega_m \cap B_l \cap B_m} \Big| u_m (\transA{\nabla \varphi_m}) \transB{\nabla \varphi_l} \Big|^q \D \lambda_n \\
& \le C' \sum_{m \sim l} \| u_m \|_{q,\Omega_m \cap B_m}^q
\end{split}
\end{equation}
with constants $C = C(n,q) > 0$ and $C' = C'(n,q,\Omega) > 0$,
where we also used that the support of the function $u_m (\transA{\nabla \varphi_m}) \transB{\nabla \varphi_l}$ is contained in $B_m \cap B_l$.
Since at most $\bar{N}$ of the balls $B_l$ have nonempty intersection, we deduce
\begin{equation*}
\begin{split}
\Big\| \Big( \sum_{m \sim l} 2 u_m (\transA{\nabla \varphi_m}) \transB{\nabla \varphi_l} \Big)_{l \in \Gamma} \Big\|_{l_q(L_q)}^q
& \le C' \sum_{l \in \Gamma} \sum_{m \sim l} \| u_m \|_{q,\Omega_m \cap B_m}^q \\
& \le C'' \sum_{l \in \Gamma} \| u_l \|_{q,\Omega_l \cap B_l}^q \\
& \le C'' \sum_{l \in \Gamma} \| u_l \|_{q,\Omega_l}^q,
\end{split}
\end{equation*}
where $C'' = C''(n,q,\Omega) > 0$.
In the same way we obtain
\begin{equation*}
\Big\| \Big( \sum_{m \sim l} (\Delta \varphi_l) \varphi_m u_m \Big)_{l \in \Gamma} \Big\|_{l_q(L_q)}^q
\le C'' \sum_{l \in \Gamma} \| u_l \|_{q,\Omega_l}^q
\end{equation*}
and 
\begin{equation*}
\Big\| \Big( \sum_{m \sim l} 2 \varphi_m (\transA{\nabla u_m}) \transB{\nabla \varphi_l} \Big)_{l \in \Gamma} \Big\|_{l_q(L_q)}^q
\le C'' \sum_{l \in \Gamma} \| \nabla u_l \|_{q,\Omega_l}^q.
\end{equation*}
Altogether, by the definition of the norm in $X$, there is a constant $C_P = C_P(n,q,\Omega) > 0$ such that
\begin{equation*}
\begin{split}
& \Big\| \Big( \sum_{m \sim l} \big[ 2 u_m (\transA{\nabla \varphi_m}) \transB{\nabla \varphi_l} + 2 \varphi_m (\transA{\nabla u_m}) \nabla \varphi_l + (\Delta \varphi_l) \varphi_m u_m \big] \Big)_{l \in \Gamma} \Big\|_{l_q(L_q)} \\
& \le C_P \| (u_l, \nabla u_l)_{l \in \Gamma} \|_{l_q(L_q)} \\
& \le \frac{C_P}{\sqrt{|\lambda|}} \| (u_l)_{l \in \Gamma} \|_X.
\end{split}
\end{equation*}
In order to treat the boundary part of $P$, we make use of the extension $\widebar{\nu}_l \in W^2_\infty(H_l)^n$ of the outward unit normal vector $\nu_l$ for $H_l$, which satisfies \eqref{eq:normal uniformity 1}:
For $l \in \Gamma_1$, a function $g_l \in W^1_q(H_l)^n$ satisfying
\begin{equation*}
\Tr_{\partial H_l} g_l = \Tr_{\partial H_l} \sum_{m \approx l} \varphi_m \big[ u_m \transA{\nabla \varphi_l} - (u_m \transA{\nabla \varphi_l})^T \big] \nu_l
\end{equation*}
is given by
\begin{equation*}
g_l := \sum_{m \approx l} \varphi_m \big[ u_m \transA{\nabla \varphi_l} - (u_m \transA{\nabla \varphi_l})^T \big] \widebar{\nu}_l.
\end{equation*}
Note that $\Tr_{\partial H_l} g_l$ is contained in the tangent space of $\partial H_l$, since we have $\nu_l \nu_l^T \Tr_{\partial H_l} g_l = 0$.
Similar to \eqref{eq:Perturbation estimate
exemplarily} by additionally using \eqref{eq:normal uniformity 1} 
we obtain that
\begin{equation*}
\| (\sqrt{\lambda} g_l, \nabla g_l) \|_{q,H_l}^q
\le C \sum_{m \approx l} \int_{H_m \cap B_m} (|\sqrt{\lambda} u_m|^q + |\nabla u_m|^q) \D\lambda_n,
\end{equation*}
where again $C = C(n,q,\Omega) > 0$.
Consequently, due to the definition of $BF_{q,\lambda}(\partial\Omega_l)$ 
we can estimate
\begin{equation} \label{eq:Localization boundary term estimate}
\begin{split}
& \Big\| \Big( \Tr_{\partial \Omega_l} \sum_{m \approx l} \varphi_m \big[ u_m \transA{\nabla \varphi_l} - (u_m \transA{\nabla \varphi_l})^T \big] \nu_l \Big)_{l \in \Gamma} \Big\|_{l_q(BF_{q,\lambda})}^q \\
& \le \sum_{l \in \Gamma_1} \| (\sqrt{\lambda} g_l, \nabla g_l) \|_{q,H_l}^q \\
& \le C \sum_{l \in \Gamma_1} \sum_{m \approx l} \int_{H_m \cap B_m} (|\sqrt{\lambda} u_m|^q + |\nabla u_m|^q) \D\lambda_n \\
& \le C' \sum_{l \in \Gamma_1} \int_{H_l \cap B_l} (|\sqrt{\lambda} u_l|^q + |\nabla u_l|^q) \D\lambda_n
\end{split}
\end{equation}
with some constant $C' = C'(n,q,\Omega) > 0$. This results in
\begin{equation*}
\Big\| \Big( \Tr_{\partial \Omega_l} \sum_{m \approx l} \varphi_m \big[ u_m \transA{\nabla \varphi_l} - (u_m \transA{\nabla \varphi_l})^T \big] \nu_l \Big)_{l \in \Gamma} \Big\|_{l_q(BF_{q,\lambda})}
\le \frac{C_P}{\sqrt{|\lambda|}} \| (u_l)_{l \in \Gamma} \|_X
\end{equation*}
for a $C_P = C_P(n,q,\Omega) > 0$ and
\eqref{eq:Localization perturbation estimate} is proved.

We now increase $\lambda_0 = \lambda_0(n,q,\theta,\Omega)$ such that $\lambda_0 \ge (2 C_S C_P)^2$, where $C_P$ is the constant from \eqref{eq:Localization perturbation estimate} and $C_S$ is the constant from \eqref{eq:Localization resolvent estimate}.
This implies
\begin{equation} \label{eq:P estimate}
\| P \|_{X \rightarrow Y} \le \frac{1}{2 C_S}.
\end{equation}
Then a Neumann series argument gives that
\begin{equation} \label{eq:S+P isom}
S + P: X \overset{\cong}{\longrightarrow} Y
\end{equation}
is isomorphic with
\begin{equation} \label{eq:S+P estimate}
\| (S + P)^{-1} \|_{Y \rightarrow X}
\le C_S \frac{1}{1 - C_S \| P \|_{X \rightarrow Y}}
\le 2 C_S.
\end{equation}

Now, \eqref{eq:S+P isom} implies that \eqref{eq:HE local} is uniquely
solvable for any right-hand sides $f_l \in L_q(\Omega_l)^n$, $g_l \in
W^1_q(\Omega)^n$ and $h_l \in W^2_q(\Omega)^n$ satisfying
$(f_l)_{l \in \Gamma} \in l_q(\bigoplus_{l \in \Gamma} L_q(\Omega_l)^n)$ and
$(a_l)_{l \in \Gamma} \in l_q(\bigoplus_{l \in \Gamma} \BF_q(\partial \Omega_l))$ where we put
$a_l := (I -\nu_l \nu_l^T) \Tr_{\partial H_l} g_l + \nu_l \nu_l^T \Tr_{\partial H_l} h_l$ ($l \in \Gamma_1$) resp.\ $a_l := 0$ ($l \in \Gamma_0$).
Furthermore, \eqref{eq:S+P estimate} is a corresponding resolvent estimate for the local equations.

\emph{Step 4: Uniqueness and resolvent estimate.}
We convince ourselves that we have proved uniqueness for \eqref{eq:HE} as well as the related resolvent estimate \eqref{eq:HE resolvent estimate}.
For any solution $u \in W^2_q(\Omega)$ of \eqref{eq:HE} we have seen that $(u_l)_{l \in \Gamma} := (\varphi_l u)_{l \in \Gamma}$ solves the local equations \eqref{eq:HE local} with right-hand sides $(f_l)_{l \in \Gamma} := (\varphi_l f)_{l \in \Gamma}$, $(g_l)_{l \in \Gamma_1} := (\varphi_l g)_{l \in \Gamma_1}$ and $(h_l)_{l \in \Gamma_1} := (\varphi_l h)_{l \in \Gamma_1}$.
Since \eqref{eq:HE local} is uniquely solvable, so is \eqref{eq:HE}.
For $a_l := (I - \nu_l \nu_l^T) \Tr_{\partial H_l} g_l + \nu_l \nu_l^T \Tr_{\partial H_l} h_l$ if $l \in \Gamma_1$ and $a_l := 0$ if $l \in \Gamma_0$, we have
$(S + P)(u_l)_{l \in \Gamma} = (f_l,a_l)_{l \in \Gamma}$.
Estimate \eqref{eq:S+P estimate} therefore implies
\begin{equation} \label{eq:local estimate 0}
\begin{split}
\| (u_l)_{l \in \Gamma} \|_X
& \le 2 C_S \| (f_l,a_l)_{l \in \Gamma} \|_Y \\
& \le 2 C_S \| (f_l, \sqrt{\lambda} g_l, \nabla g_l, \lambda h_l, \sqrt{\lambda} \nabla h_l, \nabla^2 h_l)_{l \in \Gamma} \|_{l_q(L_q)}.
\end{split}
\end{equation}
It remains to prove existence of some constant $C = C(n,q,\Omega) > 0$ so that
\begin{equation} \label{eq:local estimate 1}
\| (\lambda u,\sqrt{\lambda} \nabla u, \nabla^2 u) \|_{q,\Omega}
\le C \| (u_l)_{l \in \Gamma} \|_X
\end{equation}
and
\begin{equation} \label{eq:local estimate 2}
\begin{split}
\| & (f_l, \sqrt{\lambda} g_l, \nabla g_l, \lambda h_l, \sqrt{\lambda} \nabla h_l, \nabla^2 h_l)_{l \in \Gamma} \|_{l_q(L_q)} \\
& \le C \| (f, \sqrt{\lambda} g, \nabla g, \lambda h, \sqrt{\lambda} \nabla h, \nabla^2 h) \|_{q,\Omega}.
\end{split}
\end{equation}

For $u \in W^2_q(\Omega)^n$ and $u_m := \varphi_m u$ we have
\begin{equation} \label{eq:local estimate exemplarily}
\begin{split}
\| \lambda u \|_{q,\Omega}^q
& = |\lambda|^q \int_\Omega \sum_{l \in \Gamma} \varphi_l^2 \Big| \sum_{m \sim l} \varphi_m u_m \Big|^q \D \lambda_n \\
& \le C |\lambda|^q \int_\Omega \sum_{l \in \Gamma} \sum_{m \sim l} |\varphi_m u_m|^q \D \lambda_n \\
& \le C' |\lambda|^q \int_\Omega \sum_{l \in \Gamma} |\varphi_l u_l|^q \D \lambda_n \\
& \le C' \| (\lambda u_l)_{l \in \Gamma} \|_{l_q(L_q)}^q,
\end{split}
\end{equation}
where $C = C(n,q,\Omega) > 0$ and $C' = C'(n,q,\Omega) > 0$.
Similarly, using \eqref{eq:uniformity partition of unity}, we obtain
\begin{equation*}
\| \sqrt{\lambda} \nabla u \|_{q,\Omega}^q
\le C \| (\sqrt{\lambda} u_l, \sqrt{\lambda} \nabla u_l)_{l \in \Gamma} \|_{l_q(L_q)}^q
\end{equation*}
and
\begin{equation*}
\| \nabla^2 u \|_{q,\Omega}^q
\le C \| (u_l, \nabla u_l, \nabla^2 u_l)_{l \in \Gamma} \|_{l_q(L_q)}^q
\end{equation*}
with some constant $C = C(n,q,\Omega) > 0$.
Taking into account $|\lambda| \ge 1$, \eqref{eq:local estimate 1} follows.

For $f \in L_q(\Omega)^n$ and $f_l := \varphi_l f$ we have
\begin{equation*}
\begin{split}
\| (f_l)_{l \in \Gamma} \|_{l_q(L_q)}^q
& \le \sum_{l \in \Gamma} \int_{\Omega_l \cap B_l} |f|^q \D \lambda_n \\
& = \sum_{l \in \Gamma} \int_{\Omega \cap B_l} |f|^q \D \lambda_n \\
& \le C\| f \|_{q,\Omega}^q,
\end{split}
\end{equation*}
where $C = C(n,q,\Omega) > 0$.
Using $|\lambda| \ge 1$ and \eqref{eq:uniformity partition of unity} again, we obtain similarly
\begin{equation*}
\| (\sqrt{\lambda} g_l, \nabla g_l, \lambda h_l, \sqrt{\lambda} \nabla h_l, \nabla^2 h_l)_{l \in \Gamma} \|_{l_q(L_q)}^q
\le C \| (\sqrt{\lambda} g, \nabla g, \lambda h, \sqrt{\lambda} \nabla h, \nabla^2 h) \|_{q,\Omega}^q
\end{equation*}
with some constant $C = C(n,q,\Omega) > 0$.
Hence, \eqref{eq:local estimate 2} is proved.
Gathering \eqref{eq:local estimate 0}, \eqref{eq:local estimate 1} and
\eqref{eq:local estimate 2} implies \eqref{eq:HE resolvent estimate}.

\emph{Step 5: Existence.}
In the last step we prove existence of a solution to \eqref{eq:HE}.
For this purpose we introduce the notation $\widebar{D} v := (\varphi_l v)_{l \in \Gamma}$ for functions $v$ on $\Omega$ and $\widebar{C} (v_l)_{l \in \Gamma} := \sum_{l \in \Gamma} \varphi_l v_l$ for sequences $(v_l)_{l \in \Gamma}$ of functions $v_l$ on $\Omega_l$.
If $v$ is a function on $\partial \Omega$, then we still write $\varphi_l v$ for the restriction $(\varphi_l|_{\partial \Omega}) v$ so that $\widebar{D} v$ is a sequence of functions on $\partial \Omega$ and similarly, if $v_l$, $l \in \Gamma$ are functions on $\partial \Omega_l$ (in particular $v_l = 0$ for $l \in \Gamma_0$), then $\widebar{C} (v_l)_{l \in \Gamma}$ is a function on $\partial \Omega$.
We further put $R_\Omega u := \Tr_{\partial \Omega} \Dm(u)\nu + \Pi_\nu \Tr_{\partial \Omega} u$.

In order to prove existence, we construct a perturbation 
$\tilde P: X \rightarrow Y$ such that
\begin{equation} \label{eq:HE solution identity 2}
u = \widebar{C} (S + \tilde P)^{-1} \widebar{D} (f,\Pi_\tau \Tr g + \Pi_\nu
\Tr h).
\end{equation}
To this end, for the moment assume that the (unknown) 
operator $S + \tilde P: X \rightarrow Y$ was an isomorphism.
Then in view of \eqref{eq:HE solution identity 2} we had
\begin{equation*}
\begin{split}
& (\lambda - \Delta, R_\Omega) \widebar{C} (S + \tilde P)^{-1} \widebar{D} (f,\Pi_\tau \Tr g + \Pi_\nu \Tr h) \\
& = (f,\Pi_\tau \Tr g + \Pi_\nu \Tr h) \\
& = \widebar{C} (S + \tilde P)(S + \tilde P)^{-1} \widebar{D} (f,\Pi_\tau \Tr g +
\Pi_\nu \Tr h),
\end{split}
\end{equation*}
which is valid, if
\begin{equation} \label{eq:Commutator}
(\lambda - \Delta, R_\Omega) \widebar{C}
= \widebar{C} (S + \tilde P)
\end{equation}
is satisfied.
Thus, to identify $\tilde P$, we 
compute $\widebar{C} S - ((\lambda - \Delta),R_\Omega)
\widebar{C}$: For $(u_l)_{l \in \Gamma} \in X$ by \eqref{eq:Laplace
vector identity} we have
\begin{equation*}
\begin{split}
& \sum_{l \in \Gamma} \varphi_l (\lambda - \Delta)u_l - (\lambda - \Delta) \sum_{l \in \Gamma} \varphi_l u_l \\
& = \sum_{l \in \Gamma} \big[ (\Delta \varphi_l) u_l + 2 (\transA{\nabla u_l}) \transB{\nabla \varphi_l} \big] \\
& = \sum_{m \in \Gamma} \varphi_m^2 \sum_{l \sim m} \big[ (\Delta \varphi_l) u_l + 2 (\transA{\nabla u_l}) \transB{\nabla \varphi_l} \big] \\
& = \widebar{C} \Big( \varphi_l \sum_{m \sim l} \big[ (\Delta \varphi_m)
u_m + 2 (\transA{\nabla u_m}) \transB{\nabla \varphi_m} \big] \Big)_{l
\in \Gamma}.
\end{split}
\end{equation*}
For the boundary part the identity $\nu = \nu_l$ on $\partial \Omega \cap B_l$ as well as \eqref{eq:Gradient matrix identity} yield
\begin{equation*}
\begin{split}
& \sum_{l \in \Gamma} \big[ \Tr_{\partial \Omega_l} \varphi_l \Dm(u_l) \nu_l + \nu_l \nu_l^T \Tr_{\partial \Omega_l} \varphi_l u_l \big] \\
& \quad - \big[ \Tr_{\partial \Omega} \Dm \Big( \sum_{l \in \Gamma} \varphi_l u_l \Big) \nu + \nu \nu^T \Tr_{\partial \Omega} \sum_{l \in \Gamma} \varphi_l u_l \big] \\
& = \sum_{l \in \Gamma_1} \big[ \Tr_{\partial \Omega_l} \varphi_l (\transA{\nabla u_l} - \transB{\nabla u_l}) \nu_l
- \Tr_{\partial \Omega} (\transA{\nabla (\varphi_l u_l)} - \transB{\nabla (\varphi_l u_l)}) \nu \big] \\
& = - \sum_{l \in \Gamma_1} \Tr_{\partial \Omega_l} (u_l \transA{\nabla \varphi_l} - (u_l \transA{\nabla \varphi_l})^T) \nu_l \\
& = - \sum_{m \in \Gamma} \sum_{l \approx m} \Tr_{\partial \Omega_l} \varphi_m^2 (u_l \transA{\nabla \varphi_l} - (u_l \transA{\nabla \varphi_l})^T) \nu_l \\
& = \widebar{C} \Big( - \Tr_{\partial \Omega_l} \sum_{m \approx l} \varphi_l (u_m \transA{\nabla \varphi_m} - (u_m \transA{\nabla \varphi_m})^T) \nu_m \Big)_{l \in \Gamma}.
\end{split}
\end{equation*}
Consequently, we define $\tilde P: X \longrightarrow Y$ by
\begin{equation*}
\begin{split}
(u_l)_{l \in \Gamma} \longmapsto
\Big(
& - \varphi_l \sum_{m \sim l} \big[ (\Delta \varphi_m) u_m + 2 (\transA{\nabla u_m}) \transB{\nabla \varphi_m} \big], \\
& \Tr_{\partial \Omega_l} \sum_{m \approx l} \varphi_l (u_m \transA{\nabla \varphi_m} - (u_m \transA{\nabla \varphi_m})^T) \nu_m
\Big)_{l \in \Gamma}.
\end{split}
\end{equation*}
Then \eqref{eq:Commutator} is true and as a conclusion \eqref{eq:HE
solution identity 2} is the solution of \eqref{eq:HE}, provided 
$(S + \tilde P)^{-1}$ exists. Hence, it remains to verify that $\tilde
P$ is a perturbation of $S$ so that $S + \tilde P: X \longrightarrow Y$ is an isomorphism.

In fact, for some $C_{\tilde P} = C_{\tilde P}(n,q,\Omega) > 0$ 
it can be shown that
\begin{equation} \label{eq:Localization perturbation estimate P'}
\| \tilde P (u_l)_{l \in \Gamma} \|_Y
\le \frac{C_{\tilde P}}{\sqrt{|\lambda|}} \| (u_l)_{l \in \Gamma} \|_X
\end{equation}
for $(u_l)_{l \in \Gamma} \in X$ and 
$\lambda \in \Sigma_\theta$, $|\lambda| \ge \lambda_0$.
We will not repeat the single steps, since this is very similar
to \eqref{eq:Localization perturbation estimate}.
(Note that again
$\Tr_{\partial H_l} \sum_{m \approx l} \varphi_l (u_m \transA{\nabla
\varphi_m} - (u_m \transA{\nabla \varphi_m})^T) \nu_m$ is an element 
of the tangent space of $\partial H_l$ for every $l \in \Gamma_1$.)
Thus, we choose $\lambda_0 = \lambda_0(n,q,\theta,\Omega)$ such that
$\lambda_0 \ge (2 C_S C_{\tilde P})^2$ with $C_S$ and $C_{\tilde P}$ the constants 
from \eqref{eq:Localization resolvent estimate} and 
\eqref{eq:Localization perturbation estimate P'}, respectively.
As for \eqref{eq:P estimate} this results in 
\begin{equation*}
\| \tilde P \|_{X \rightarrow Y} \le \frac{1}{2 C_S}.
\end{equation*}
A standard Neumann series argument then implies 
\begin{equation*}
S + \tilde P: X \overset{\cong}{\longrightarrow} Y.
\end{equation*}
to be an isomorphism and the proof of Theorem~\thref{thm:HE with PS} 
is completed.
\end{proof}

\section{Invariance of $L_{q,\sigma}$ for the Laplace resolvent}
\label{invariance}

In order to reduce the well-posedness of the Stokes resolvent problem
subject to perfect slip, as explained in the introduction, 
we will utilize the fact that the Laplace resolvent leaves 
$L_{q,\sigma}$ invariant. For this purpose, we start with a lemma on
vanishing divergence.
Let $\Omega \subset \mathbb{R}^n$ be a uniform $C^{2,1}$-domain, $n \ge 2$ and $1 < q < \infty$.

\begin{lemma} \label{thm:Neumann Laplace weak}
Let $0 < \theta < \pi$.
Then there exists $\lambda_0 = \lambda_0(n,q,\theta,\Omega) > 0$ such that for $\lambda \in \Sigma_\theta$, $|\lambda| \ge \lambda_0$ we have: Any $w \in W^2_q(\Omega)^n$ solving
\begin{equation} \label{eq:Neumann div}
\left\{
\begin{array}{rll}
(\lambda - \Delta) \Div w & = 0 & \text{in } \Omega \\
\partial_\nu \Div w & = 0 & \text{on } \partial \Omega
\end{array}
\right.
\end{equation}
(i.e., $(\lambda - \Delta) \Div w = 0$ in the sense of distributions and
$\Tr_\nu \nabla \Div w = 0$ in $W^{-1/q}_q(\partial \Omega)$)
satisfies $\Div w = 0$.
\end{lemma}

\begin{proof}
Let $\LaplaceN{q}: \Def(\LaplaceN{q}) \subset L_q(\Omega) \rightarrow
L_q(\Omega), ~u \mapsto \Delta u$ be the Neumann-Laplace operator, that
is, $\Def(\LaplaceN{q}) = \{ u \in W^2_q(\Omega) : \partial_\nu u = 0$
on $\partial \Omega \}$. Further, let
$\LaplaceN{q}^*: L_{q'}(\Omega) \rightarrow \Def(\LaplaceN{q})'$ be the
continuous dual operator (here we equip $\Def(\LaplaceN{q})$ with the
graph norm). Observe that we can regard $L_{q'}(\Omega)$ as a subspace of $\Def(\LaplaceN{q})'$, since $\Def(\LaplaceN{q}) \subset L_q(\Omega)$ is dense.

We aim to prove that $(\lambda - \LaplaceN{q'}^*) \Div w = 0$ for 
$\lambda$ as asserted. For this purpose, fix some $\varphi \in \Def(\LaplaceN{q'})$. Then the Neumann boundary conditions $\nu \cdot \nabla \varphi = 0$ and $\nu \cdot \nabla \Div w = 0$ on $\partial \Omega$ yield
\begin{equation*}
\begin{split}
& \dual{(\lambda - \LaplaceN{q'}^*) \Div w}{\varphi}_{\Def(\LaplaceN{q'})',\Def(\LaplaceN{q'})} \\
& = \dualq{\Div w}{(\lambda - \Delta) \varphi} \\
& = \dualq{\Div w}{\lambda \varphi} - \dualq{\Div w}{\Div \nabla \varphi} \\
& = \dualq{\Div w}{\lambda \varphi} + \int_\Omega \nabla \Div w \cdot \nabla \varphi \D \lambda_n - \dualb{\Div w}{\nu \cdot \nabla \varphi} \\
& = \dualq{\Div w}{\lambda \varphi} - \int_\Omega (\Delta \Div w) \varphi \D \lambda_n + \dualb{\varphi}{\nu \cdot \nabla \Div w} \\
& = \dualq{(\lambda - \Delta) \Div w}{\varphi}\\ 
& = 0.
\end{split}
\end{equation*}
Note that here we applied Lemma~\ref{thm:Greens formula in E_q}, 
once for $\nabla \varphi \in E_{q'}(\Omega)$ and 
$\Div w \in W^1_q(\Omega)$ and second for 
$\nabla \Div w \in E_q(\Omega)$ and $\varphi \in W^1_{q'}(\Omega)$.
Also note that $(\lambda - \Delta) \Div w = 0$ in the sense
of distributions implies this to be valid also as an equality in $L_q$,
due to $\Delta \Div w = -\lambda \Div w \in L_q(\Omega)$. 
Thus, we conclude $(\lambda - \LaplaceN{q'}^*) \Div w = 0$.

Now $\lambda - \LaplaceN{q'}: \Def(\LaplaceN{q'}) \overset{\cong}{\longrightarrow} L_{q'}(\Omega)$ is an isomorphism when $\lambda \in \Sigma_\theta$ and $|\lambda| \ge \lambda_0$ for some $\lambda_0 = \lambda_0(n,q,\theta,\Omega) > 0$
(see Lemma~\ref{thm:Neumann Laplace}; simply choose $\lambda_0$ such that the conditions of Lemma~\ref{thm:Neumann Laplace} are satisfied for $q$ and $q'$), so its continuous dual operator, $\lambda - \LaplaceN{q'}^*$, is injective.
Hence, $\Div w = 0$.
\end{proof}

Now we can prove the desired invariance.

\begin{lemma} \label{thm:resolvent}
Let Assumption~\thref{thm:Assumption C} be valid.
Let $0 < \theta < \pi$,
choose $\lambda_0 = \lambda_0(n,q,\theta,\Omega) > 0$ so that the
conditions of Theorem~\thref{thm:HE with PS} and Lemma~\thref{thm:Neumann
Laplace weak} are satisfied and let $\lambda \in \Sigma_\theta$,
$|\lambda| \ge \lambda_0$. Then we have the following implications:
\begin{enumerate}[\upshape (i)]
\item \label{thm:resolvent 1} $u \in \Def(\LaplacePS) \cap L_{q,\sigma}(\Omega)
\thickspace \Rightarrow \thickspace
\Delta u \in L_{q,\sigma}(\Omega)$.
\item \label{thm:resolvent 2} $f \in L_{q,\sigma}(\Omega)
\thickspace \Rightarrow \thickspace
(\lambda - \LaplacePSq)^{-1} f \in L_{q,\sigma}(\Omega)$.
\end{enumerate}
\end{lemma}

\begin{proof}
We will make use of both, the $L_{q,\sigma}(\Omega)$-representation in \eqref{eq:L_q,sigma} and Lemma~\ref{thm:characterization L_q,sigma}.
Let $u \in \Def(\LaplacePS) \cap L_{q,\sigma}(\Omega)$ and $\varphi \in C_c^\infty(\overline{\Omega})$. 
Based on Lemma~\ref{thm:Gauss in W^1_1} and Lemma
\ref{thm:Dminus}\eqref{thm:Dminus 2},\eqref{thm:Dminus 3} we obtain
\begin{equation*}
\begin{split}
\dualq{\Delta u}{\nabla \varphi}
& = - \int_\Omega (\nabla \Div u - \transB{\Delta u}) \cdot \nabla \varphi \D \lambda_n \\
& = - \int_\Omega \Div (\Dm(u) \transB{\nabla \varphi}) \D \lambda_n \\
& = - \int_{\partial \Omega} \nu \cdot \Dm(u) \transB{\nabla \varphi} \D \sigma \\
& = \int_{\partial \Omega} \nabla \varphi \cdot \Dm(u) \nu \D \sigma \\
& = 0.
\end{split}
\end{equation*}
This holds for all $\varphi \in \widehat{W}^1_{q'}(\Omega)$ as well,
since $C_c^\infty(\overline{\Omega}) \subset \widehat{W}^1_{q'}(\Omega)$
is dense. Hence, \eqref{thm:resolvent 1} is proved.

In order to see \eqref{thm:resolvent 2}, pick
$f \in L_{q,\sigma}(\Omega)$.
The function $u := (\lambda - \LaplacePSq)^{-1} f \in W^2_q(\Omega)^n$ is the solution of
\begin{equation} \label{eq:HE homogeneous}
\left\{
\begin{array}{rll}
\lambda u - \Delta u & = f & \text{in } \Omega \\
\Dm(u)\nu & = 0 & \text{on } \partial \Omega \\
\nu \cdot u & = 0 & \text{on } \partial \Omega.
\end{array}
\right.
\end{equation}
So, applying $\Tr_\nu$ to the first line of \eqref{eq:HE homogeneous}, we obtain
\begin{equation} \label{eq:Trace Laplace}
\Tr_\nu \Delta u = 0 \quad \text{in } W^{-\frac{1}{q}}_q(\partial \Omega).
\end{equation}
Applying $\Div$ to the first line of \eqref{eq:HE homogeneous}, 
we also see that $(\lambda - \Delta) \Div u = 0$ in the sense of distributions.
Next, we show that $\partial_\nu \Div u =
0$ on $\partial \Omega$:

Let $k \in W^{1-1/q'}_{q'}(\partial \Omega)$ and choose $w \in W^1_{q'}(\Omega)$ so that $\Tr w = k$.
First note that $\nabla \Div u \in E_q(\Omega)$, so $\Tr_\nu \nabla \Div u$ is well defined. We have
\begin{equation} \label{eq:applying Gauss A}
\begin{split}
\dualGb{k}{\Tr_\nu \nabla \Div u}
&= \dualGb{w}{\nu \cdot (\nabla \Div u - \transB{\Delta u})} \\
&= \int_\Omega \Div (w(\nabla \Div u - \transB{\Delta u})) \D \lambda_n \\
&= \int_\Omega \nabla w \cdot (\nabla \Div u - \transB{\Delta u}) \D \lambda_n,
\end{split}
\end{equation}
using \eqref{eq:Trace Laplace}, Lemma~\ref{thm:Extended Gauss theorem} and $\Div (\nabla \Div u - \transB{\Delta u}) = 0$.
In case $w \in C_c^\infty(\overline{\Omega})$, we obtain for the last
term in \eqref{eq:applying Gauss A} that
\begin{equation} \label{eq:applying Gauss B}
\begin{split}
\int_\Omega \nabla w \cdot (\nabla \Div u - \transB{\Delta u}) \D \lambda_n
&= \int_\Omega \Div(\Dm(u) \transB{\nabla w}) \D \lambda_n \\
&= \int_{\partial \Omega} \nu \cdot \Dm(u) \transB{\nabla w} \D \sigma \\
&= - \int_{\partial \Omega} \nabla w \cdot \Dm(u) \nu \D \sigma \\
&= 0,
\end{split}
\end{equation}
using Lemma~\ref{thm:Dminus}\eqref{thm:Dminus 2},\eqref{thm:Dminus 3} and Lemma~\ref{thm:Gauss in W^1_1}.
The density of $C_c^\infty(\overline{\Omega}) \subset W^1_{q'}(\Omega)$ gives that \eqref{eq:applying Gauss B} holds for $w \in W^1_{q'}(\Omega)$ as well.
Therefore, \eqref{eq:applying Gauss A} and \eqref{eq:applying Gauss B}
yield $\partial_\nu \Div u = 0$ on $\partial \Omega$. In other words,
we have
\begin{equation*}
\left\{
\begin{array}{rll}
(\lambda - \Delta) \Div u & = 0 & \text{in } \Omega \\
\partial_\nu \Div u & = 0 & \text{on } \partial \Omega.
\end{array}
\right.
\end{equation*}
Consequently, $\Div u = 0$ due to Lemma~\ref{thm:Neumann Laplace weak}. Lemma~\ref{thm:characterization L_q,sigma} yields $u \in L_{q,\sigma}(\Omega)$.
\end{proof}

\section{Perfect slip boundary conditions: Proof of Theorems~\ref{thm:Stokes perfect slip} and~\ref{thm:Stokes perfect slip inhom}}
\label{sec:proofsofstokes}

\begin{proof}[Proof of Theorem~\thref{thm:Stokes perfect slip}]
Choose $\lambda_0 = \lambda_0(n,q,\theta,\Omega)$ and $C = C(n,q,\theta,\Omega)$ such that the conditions of Theorem~\ref{thm:HE with PS} and Lemma~\ref{thm:Neumann Laplace weak} are satisfied and let $\lambda \in \Sigma_\theta$, $|\lambda| \ge \lambda_0$.

In order to prove \eqref{thm:Stokes perfect slip 1}, we decompose a given function $f \in L_{q,\sigma}(\Omega) + \Gradq$ into $f_0 \in L_{q,\sigma}(\Omega)$ and $\transB{\nabla \pi} \in \Gradq$. Setting
\begin{equation*}
(u,\transB{\nabla p}) := ((\lambda - \LaplacePSq)^{-1} f_0,\transB{\nabla \pi}),
\end{equation*}
we obtain a solution of \eqref{eq:Stokes perfect slip}, due to Lemma~\ref{thm:resolvent}\eqref{thm:resolvent 2}.
Conversely, if there exists a solution $(u,\transB{\nabla p})$ of \eqref{eq:Stokes perfect slip} with right-hand side $f \in L_q(\Omega)^n$, then Lemma~\ref{thm:resolvent}\eqref{thm:resolvent 1} gives that $f \in L_{q,\sigma}(\Omega) + \Gradq$.

Thanks to Lemma~\ref{thm:resolvent}\eqref{thm:resolvent 2}
a solution of the homogeneous problem \eqref{eq:Stokes perfect slip} is given by
$\big( (\lambda - \LaplacePSq)^{-1} \transB{\nabla \pi}, -\transB{\nabla
\pi} \big)$ for every $\transB{\nabla \pi} \in U_q(\Omega)$.
If, conversely, $(u,\transB{\nabla p}) \in [W^2_q(\Omega)^n \cap
L_{q,\sigma}(\Omega)] \times \Gradq$ solves \eqref{eq:Stokes perfect
slip} with $f = 0$, then we have $(\lambda - \Delta)u = -\transB{\nabla p} \in \Gradq$.
On the other hand, Lemma~\ref{thm:resolvent}\eqref{thm:resolvent 1}
yields $(\lambda - \Delta)u \in L_{q,\sigma}(\Omega)$,
hence $\transB{\nabla p} = - (\lambda - \Delta)u \in U_q(\Omega)$.
This yields
\begin{equation*}
(u,\transB{\nabla p}) = ((\lambda - \LaplacePSq)^{-1} \transB{\nabla
\pi},-\transB{\nabla \pi})
\end{equation*}
if we set $\nabla \pi := - \nabla p \in U_q(\Omega)$,
and \eqref{thm:Stokes perfect slip 2} is proved.

Now, let Assumption~\ref{thm:Assumption A}\eqref{thm:Assumption A 1} 
be satisfied and let $f \in L_{q,\sigma}(\Omega) + \Gradq$.
Using the direct decomposition \eqref{eq:decomposition modGrad b}, we
can decompose $f = f_0 + \transB{\nabla p}$ with $f_0 \in L_{q,\sigma}(\Omega)$ and $\transB{\nabla p} \in \modGradq$.
The solution
\begin{equation*}
(u,\transB{\nabla p}) := \big( (\lambda - \LaplacePSq)^{-1}f_0,\transB{\nabla p} \big)
\end{equation*}
of \eqref{eq:Stokes perfect slip} is contained in $[W^2_q(\Omega)^n \cap
L_{q,\sigma}(\Omega)] \times \modGradq$. So, we only have to prove that there is at most one solution in this space to obtain uniqueness.
To this end, let $(v,\transB{\nabla \pi}) \in [W^2_q(\Omega)^n \cap L_{q,\sigma}(\Omega)] \times \modGradq$ be a solution of the homogeneous problem \eqref{eq:Stokes perfect slip}.
Lemma~\ref{thm:resolvent}\eqref{thm:resolvent 1} then yields $(\lambda -
\Delta)v \in L_{q,\sigma}(\Omega)$. On the other hand, we also have
\begin{equation*}
(\lambda - \Delta)v = - \transB{\nabla \pi} \in \modGradq.
\end{equation*}
The fact that $\modGradq \cap L_{q,\sigma}(\Omega) = \{ 0 \}$ implies $\transB{\nabla \pi} = 0$ and $v = -(\lambda - \LaplacePSq)^{-1} \transB{\nabla \pi} = 0$.
Hence, solutions of \eqref{eq:Stokes perfect slip} in $[W^2_q(\Omega)^n \cap L_{q,\sigma}(\Omega)] \times \modGradq$ are unique
and sufficiency in \eqref{thm:Stokes perfect slip 3} is proved.

Conversely, for any right-hand side function $f \in L_q(\Omega)^n$ the
condition $f \in L_{q,\sigma}(\Omega) + \Gradq$ is also necessary to
obtain existence of the solution in \eqref{thm:Stokes perfect slip 3}.
This follows by the fact that
$[W^2_q(\Omega)^n \cap L_{q,\sigma}(\Omega)] \times \modGradq$ is a
subspace of $[W^2_q(\Omega)^n \cap L_{q,\sigma}(\Omega)] \times \Gradq$
and since for the latter space we have seen necessity in
\eqref{thm:Stokes perfect slip 1} already. Altogether, 
we proved \eqref{thm:Stokes perfect slip 3}.

To see \eqref{thm:Stokes perfect slip 4}, let Assumption~\ref{thm:Assumption A} be valid.
Then, due to to Lemma~\ref{thm:complemented subspaces}\eqref{thm:complemented subspaces 3} there is 
a constant $C' = C'(n,q,\Omega) > 0$ so that for the decomposition $f = f_0 + \transB{\nabla p}$ we have
\begin{equation} \label{eq:decomposition modGrad estimate a}
\| (f_0,\nabla p) \|_q \le C' \| f \|_q.
\end{equation}
Thanks to this estimate and Theorem~\ref{thm:HE with PS}
the solution $(u,\nabla p) = \bigl((\lambda - \LaplacePSq)^{-1}
f_0,\nabla p\bigr)$ of \eqref{eq:Stokes perfect slip} satisfies
the resolvent estimate \eqref{eq:resolvent estimate perfect slip}.
\end{proof}

\begin{proof}[Proof of Theorem~\thref{thm:Stokes perfect slip inhom}]
Fix $\lambda_0 = \lambda_0(n,q,\theta,\Omega)$ such that the conditions of Theorem~\ref{thm:HE with PS} and Lemma~\ref{thm:Neumann Laplace weak} are satisfied and let $\lambda \in \Sigma_\theta$, $|\lambda| \ge \lambda_0$.

Let $g \in W^1_q(\Omega)^n$ and assume initially $f \in L_{q,\sigma}(\Omega)$.
Denote by $\tilde{u} \in W^2_q(\Omega)^n$ the unique solution of
\begin{equation*}
\left\{
\begin{array}{rll}
\lambda \tilde{u} - \Delta \tilde{u} & = 0 & \text{in } \Omega \\
\Dm(\tilde{u})\nu & = \Pi_\tau g & \text{on } \partial \Omega \\
\nu \cdot \tilde{u} & = 0 & \text{on } \partial \Omega
\end{array}
\right.
\end{equation*}
(see Theorem~\ref{thm:HE with PS}).
By Assumption~\ref{thm:Assumption B} we can decompose
\begin{equation} \label{eq:weak Neumann pre}
	\nabla \Div \tilde{u} - \transB{\Delta \tilde{u}}
	=h-\nabla p
\end{equation}
with $h\in L_{q,\sigma}(\Omega)$ and $\nabla p\in G_q(\Omega)$.
With regard to \eqref{eq:L_q,sigma} this yields 
\begin{equation} \label{eq:weak Neumann}
\dualq{-\nabla p}{\nabla \varphi}
= \dualq{\nabla \Div \tilde{u} - \transB{\Delta \tilde{u}}}{\nabla
\varphi} \quad \forall \varphi \in \widehat{W}^1_{q'}(\Omega).
\end{equation}
Utilizing Theorem~\ref{thm:HE with PS} again, we define $u \in W^2_q(\Omega)^n$ as the unique solution of
\begin{equation} \label{eq:HE with pressure}
\left\{
\begin{array}{rll}
\lambda u - \Delta u & = f - \transB{\nabla p} & \text{in } \Omega \\
\Dm(u)\nu & = \Pi_\tau g & \text{on } \partial \Omega \\
\nu \cdot u & = 0 & \text{on } \partial \Omega.
\end{array}
\right.
\end{equation}

With the help of the representation of $L_{q,\sigma}(\Omega)$ 
from Lemma~\ref{thm:characterization L_q,sigma}
we now aim to prove that $u \in L_{q,\sigma}(\Omega)$.
First, applying $\Tr_\nu$ to the first line of \eqref{eq:HE with pressure} gives
\begin{equation} \label{eq:Laplace and pressure}
\Tr_\nu \Delta u = \Tr_\nu \transB{\nabla p}.
\end{equation}
Note that $\Tr_\nu \nabla p$ is well-defined, since $\Div \nabla p =
-\Div(\nabla \Div \tilde{u} - \transB{\Delta \tilde{u}}) = 0$ in the
sense of distributions by \eqref{eq:weak Neumann pre}.
Also $\Tr_\nu \nabla \Div u$ is well defined, since \eqref{eq:HE with pressure} yields
$\Div \nabla \Div u = \Div \Delta u = \lambda \Div u \in L_q(\Omega)^n$.
In order to see that $\Tr_\nu \nabla \Div u = 0$,
let $k \in W^{1-1/q'}_{q'}(\partial \Omega)$ and fix any $w \in
W^1_{q'}(\Omega)$ so that $\Tr w = k$ (Lemma~\ref{thm:Trace}). Based on
\eqref{eq:Laplace and pressure}, Lemma~\ref{thm:Extended Gauss theorem}
and $\Div \nabla p = 0$, as in \eqref{eq:applying Gauss A} we obtain
\begin{equation} \label{eq:normal trace zero}
\begin{split}
\dualb{k}{\Tr_\nu \nabla \Div u}
&= \dualb{w}{\nu \cdot (\nabla p + \nabla \Div u - \transB{\Delta u})} \\
&= \int_\Omega \Div (w (\nabla p + \nabla \Div u - \transB{\Delta u})) \D \lambda_n \\
&= \int_\Omega \nabla w \cdot (\nabla p + \nabla \Div u - \transB{\Delta
u}) \D \lambda_n.
\end{split}
\end{equation}
Now, in the last term of \eqref{eq:normal trace zero} we can replace
$\nabla \Div u - \transB{\Delta u}$ by $\nabla \Div \tilde{u} -
\transB{\Delta \tilde{u}}$. In fact, using Lemma~\ref{thm:Dminus}\eqref{thm:Dminus 2} and~\eqref{thm:Dminus 3} and Lemma~\ref{thm:Gauss in
W^1_1}, we can calculate for $w \in C_c^\infty(\overline{\Omega})$,
\begin{equation*}
\begin{split}
\int_\Omega \nabla w \cdot (\nabla \Div u - \transB{\Delta u}) \D \lambda_n
&= \int_\Omega \Div (\Dm(u) \transB{\nabla w}) \D \lambda_n \\
&= \int_{\partial \Omega} \nu \cdot (\Dm(u) \transB{\nabla w}) \D \sigma \\
&= - \int_{\partial \Omega} \nabla w \cdot (\Dm(u) \nu) \D \sigma \\
&= - \int_{\partial \Omega} \nabla w \cdot (\Pi_\tau g) \D \sigma.
\end{split}
\end{equation*}
The same calculation holds true with $u$ replaced by $\tilde{u}$
which then implies
\begin{equation} \label{eq:replace u by tilde(u)}
\int_\Omega \nabla w \cdot (\nabla \Div u - \transB{\Delta u}) \D \lambda_n
= \int_\Omega \nabla w \cdot (\nabla \Div \tilde{u} - \transB{\Delta \tilde{u}}) \D \lambda_n
\end{equation}
for $w \in C_c^\infty(\overline{\Omega})$. The density of $C_c^\infty(\overline{\Omega}) \subset W^1_{q'}(\Omega)$ yields that \eqref{eq:replace u by tilde(u)} holds for $w \in W^1_{q'}(\Omega)$ as well and therefore \eqref{eq:weak Neumann} gives that the right-hand side of \eqref{eq:normal trace zero} vanishes.
Consequently, $\Tr_\nu \nabla \Div u = 0$.

Next, applying $\Div$ to the first line of \eqref{eq:HE with pressure}
results in $(\lambda - \Delta) \Div u = 0$. Thus, $\Div u$ satisfies
\begin{equation*}
\left\{
\begin{array}{rll}
(\lambda - \Delta) \Div u & = 0 & \text{in } \Omega \\
\partial_\nu \Div u & = 0 & \text{on } \partial \Omega.
\end{array}
\right.
\end{equation*}
Lemma~\ref{thm:Neumann Laplace weak} yields $\Div u = 0$.
Summarizing, we conclude that $u \in L_{q,\sigma}(\Omega)$
and that $(u,\transB{\nabla p}) \in [W^2_q(\Omega)^n \cap L_{q,\sigma}(\Omega)] \times \Gradq$ is a solution of \eqref{eq:Stokes perfect slip inhom}.

For the general case $f \in L_q(\Omega)^n$, by 
Assumption~\ref{thm:Assumption B} we can decompose $f = f_0 +
\transB{\nabla \pi}$ with $f_0 \in L_{q,\sigma}(\Omega)$ and $\transB{\nabla \pi} \in \Gradq$.
By what we have shown already, there exists a solution
$(u,\transB{\nabla p}) \in [W^2_q(\Omega)^n \cap L_{q,\sigma}(\Omega)]
\times \Gradq$ of \eqref{eq:Stokes perfect slip inhom} for the 
right-hand side $(f_0,g)$. Then, obviously $(u,\transB{\nabla p} +
\transB{\nabla \pi})$ solves \eqref{eq:Stokes perfect slip inhom} for
the right-hand side $(f,g)$.
Thus, \eqref{thm:Stokes perfect slip inhom 1} is proved.

Let now Assumptions~\ref{thm:Assumption A} and~\ref{thm:Assumption B} be valid.
Again, let initially $f \in L_{q,\sigma}(\Omega)$.
As in the proof of \eqref{thm:Stokes perfect slip inhom 1} let
$\tilde{u} \in W^2_q(\Omega)^n$ be the unique solution of
\begin{equation*}
\left\{
\begin{array}{rll}
\lambda \tilde{u} - \Delta \tilde{u} & = 0 & \text{in } \Omega \\
\Dm(\tilde{u})\nu & = \Pi_\tau g & \text{on } \partial \Omega \\
\nu \cdot \tilde{u} & = 0 & \text{on } \partial \Omega.
\end{array}
\right.
\end{equation*}
Theorem~\ref{thm:HE with PS} then yields
\begin{equation} \label{eq:resolvent estimate preparation a}
\| (\lambda \tilde{u},\sqrt{\lambda} \nabla \tilde{u},\nabla^2 \tilde{u}) \|_q \le C \| (\sqrt{\lambda} g,\nabla g) \|_q
\end{equation}
with a constant $C = C(n,q,\theta,\Omega) > 0$.
The direct decomposition \eqref{eq:decomposition modGrad a} gives that
\[
\transB{\nabla \Div \tilde{u}} - \Delta \tilde{u} = v_0 - \transB{\nabla
p},
\]
this time with unique $v_0 \in L_{q,\sigma}(\Omega)$ and 
$-\transB{\nabla p} \in \modGradq$.
Relation \eqref{eq:L_q,sigma} also here implies that
\begin{equation*}
\dualq{-\nabla p}{\nabla \varphi}
= \dualq{\nabla \Div \tilde{u} - \transB{\Delta \tilde{u}}}{\nabla
\varphi} \quad \forall \varphi \in \widehat{W}^1_{q'}(\Omega).
\end{equation*}
Also note that $-\transB{\nabla p} \in \modGradq$ in this case 
is the unique solution of this weak Neumann problem.

Next, by virtue of \eqref{eq:decomposition modGrad a} and 
Lemma~\ref{thm:complemented subspaces}\eqref{thm:complemented subspaces 3} there exists a constant 
$C' = C'(n,q,\Omega) > 0$ so that
\begin{equation} \label{eq:resolvent estimate preparation b}
\| \nabla p \|_q \le C' \| \nabla \Div \tilde{u} - \transB{\Delta \tilde{u}} \|_q.
\end{equation}
Again, we define $u \in W^2_q(\Omega)^n$ as the unique solution of
\begin{equation*}
\left\{
\begin{array}{rll}
\lambda u - \Delta u & = f - \transB{\nabla p} & \text{in } \Omega \\
\Dm(u)\nu & = \Pi_\tau g & \text{on } \partial \Omega \\
\nu \cdot u & = 0 & \text{on } \partial \Omega
\end{array}
\right.
\end{equation*}
and obtain $u \in L_{q,\sigma}(\Omega)$ as in the proof of \eqref{thm:Stokes perfect slip inhom 1}.
Hence, $(u,\transB{\nabla p}) \in [W^2_q(\Omega)^n \cap L_{q,\sigma}(\Omega)] \times \modGradq$ is the unique solution of \eqref{eq:Stokes perfect slip inhom}.
In addition, Theorem~\ref{thm:HE with PS} yields
\begin{equation} \label{eq:resolvent estimate preparation c}
\| (\lambda u,\sqrt{\lambda} \nabla u,\nabla^2 u) \|_q
\le C \| (f - \transB{\nabla p}, \sqrt{\lambda}g,\nabla g) \|_q
\end{equation}
with a constant $C = C(n,q,\theta,\Omega) > 0$.
The estimates \eqref{eq:resolvent estimate preparation a}, \eqref{eq:resolvent estimate preparation b} and \eqref{eq:resolvent estimate preparation c} imply \eqref{eq:resolvent estimate perfect slip inhom}.

Finally, let $f \in L_q(\Omega)^n$. Decomposition \eqref{eq:decomposition modGrad a} gives
$f = f_0 + \transB{\nabla \pi}$ with two unique functions $f_0 \in
L_{q,\sigma}(\Omega)$ and $\transB{\nabla \pi} \in \modGradq$ as well as
a constant $C' = C'(n,q,\Omega) > 0$ such that
\begin{equation} \label{eq:resolvent estimate preparation d}
\| (f_0,\nabla \pi) \|_q \le C' \| f \|_q.
\end{equation}
We have proved that \eqref{eq:Stokes perfect slip inhom} with right-hand side 
$(f_0,g)$ admits a unique solution $(u,\transB{\nabla p}) \in [W^2_q(\Omega)^n \cap L_{q,\sigma}(\Omega)] \times \modGradq$ satisfying \eqref{eq:resolvent estimate perfect slip inhom} with $f_0$ instead of $f$.
Thus,
$(u,\transB{\nabla p} + \transB{\nabla \pi}) \in [W^2_q(\Omega)^n \cap
L_{q,\sigma}(\Omega)] \times \modGradq$ is the unique solution of
\eqref{eq:Stokes perfect slip inhom} with right-hand side $(f,g)$.
Furthermore, \eqref{eq:resolvent estimate preparation d} yields the
corresponding resolvent estimate \eqref{eq:resolvent estimate perfect slip inhom} with $\transB{\nabla p} + \transB{\nabla \pi}$ instead of $\transB{\nabla p}$.
The proof of \eqref{thm:Stokes perfect slip inhom 2} is now completed.
\end{proof}

\section{Partial slip type boundary conditions: Proof of Theorem~\ref{thm:Stokes partial slip}}
\label{sec:proofpartial}

We first show that partial slip type boundary conditions as considered
in Theorem~\ref{thm:Stokes partial slip} can be obtained by 
perturbing perfect slip boundary conditions.
As before, the underlying domain $\Omega \subset \mathbb{R}^n$ has 
a uniform $C^{2,1}$-boundary, and we assume $n \ge 2$ as well as 
$1 < q < \infty$.

\begin{lemma} \label{thm:matrix different partial slip b.c.}
There exists a matrix $A \in W^1_\infty(\Omega)^{n \times n}$
such that for all $u \in W^2_q(\Omega)^n$ satisfying 
$\nu \cdot u = 0$ on $\partial \Omega$ we have
\begin{equation*}
\Pi_\tau \Dp(u) \nu = \Dm(u) \nu + \Pi_\tau A u
\quad \text{on } \partial \Omega.
\end{equation*}
\end{lemma}

\begin{proof}
Let $T_x \partial \Omega \subset \mathbb{R}^{n-1}$ be the tangent space in some fixed point $x \in \partial \Omega$. Let $\tau_1,\dots,\tau_{n-1}$ be a basis of $T_x \partial \Omega$. Then, with the outer unit normal $\tau_n := \nu = \nu(x)$, let $\tau^1,\dots,\tau^n$ be the dual basis of $\tau_1,\dots,\tau_n$ in $\mathbb{R}^n$ (i.e., $\tau_i \cdot \tau^j = \delta_{ij}$ for $i,j = 1,\dots,n$).
Note that then $\tau^n = \nu$.

We first observe for the tangential projection $\Pi_\tau u = (I - \nu \nu^T) u$, the change of basis matrix $S := (\tau_1,\dots,\tau_{n-1},\nu)^T$ and the vector $[u]_{1,\dots,n}$ of covariant components
$[u]_i := u \cdot \tau_i = (Su)_i$ that
\begin{enumerate}[\upshape (a)]
\item \label{eq:boundary a} $\Pi_\tau u = \sum_{k=1}^{n-1} (u \cdot \tau_k) \tau^k$,
\item \label{eq:boundary b} $S^{-1} = (\tau^1,\dots,\tau^{n-1},\nu)$ and
\item \label{eq:boundary c} $\Pi_\tau S^{-1} [u]_{1,\dots,n} = \Pi_\tau S^{-1} ([u]_1,\dots,[u]_{n-1},0)^T$.
\end{enumerate}
It is obvious that $(\tau^1,\dots,\tau^{n-1},\nu)$ is a right inverse of
$S$ and since $S$ has full rank, it must be the left inverse, too. 
Thus, \eqref{eq:boundary b} is true.
From \eqref{eq:boundary b} we infer, using the representation $u = S^{-1} [u]_{1,\dots,n}$, that
\begin{equation*}
\begin{split}
\Pi_\tau u
& = (I - \nu \nu^T) (\tau^1,\dots,\tau^{n-1},\nu) [u]_{1,\dots,n} \\
& = (\tau^1,\dots,\tau^{n-1},0) [u]_{1,\dots,n} \\
& = \sum_{k=1}^{n-1} (u \cdot \tau_k) \tau^k.
\end{split}
\end{equation*}
Hence, \eqref{eq:boundary a} is true.
Based on \eqref{eq:boundary a} and \eqref{eq:boundary b} we obtain
\eqref{eq:boundary c} by computing
\begin{equation*}
\begin{split}
\Pi_\tau S^{-1} ([u]_1,\dots,[u]_{n-1},0)^T
&= (\tau^1,\dots,\tau^{n-1},0) ([u]_1,\dots,[u]_{n-1},0)^T \\
&= \sum_{k=1}^{n-1} (u \cdot \tau_k) \tau^k = \Pi_\tau u \\
&= \Pi_\tau S^{-1} [u]_{1,\dots,n}.
\end{split}
\end{equation*}

Next, we choose a concrete basis of $T_x \partial \Omega$ in an arbitrary point $x \in \partial \Omega$. For this purpose, let $\phi_l$, $l \in \Gamma_1$ be the parametrization of the boundary $\partial \Omega$ chosen in \eqref{eq:Parametrization}.
If, for some $l \in \Gamma_1$, the point $x \in \partial \Omega$ is contained in the part $\partial \Omega \cap B_l$ of the boundary, the functions $\partial_i \phi_l$, $i = 1,\dots,n-1$ form a basis of $T_x \partial \Omega$. More precisely, we can define $\tau_i = \tau_i(x) := \partial_i \phi_l(\phi_l^{-1}(x))$ for $i = 1,\dots,n-1$. Let $l \in \Gamma_1$ be fixed now.
For a function $v: \partial \Omega \cap B_l \rightarrow \mathbb{R}$ and $i = 1,\dots,n-1$ we define the $i$-th tangential derivative as
$\partial_{\tau_i} v := \partial_i (v \circ \phi_l) \circ \phi_l^{-1}$
and, if $v$ is a vector field, 
then $\partial_{\tau_i} v$ is defined componentwise.
Note that for $v \in W^1_q(\Omega \cap B_l)$ by the chain rule we have 
$\partial_{\tau_i} v = \nabla v \cdot \tau_i$, as usual.
In case $v \in W^1_q(\Omega \cap B_l)^n$, we have $\partial_{\tau_i} v = (\transA{\nabla v}) \tau_i$. Therefore, for $u \in W^2_q(\Omega)^n$ we have
\begin{equation} \label{eq:tangential derivative a}
[(\transB{\nabla u}) \nu]_i
= \tau_i \cdot (\transB{\nabla u}) \nu
= \nu \cdot (\transA{\nabla u}) \tau_i
= \nu \cdot \partial_{\tau_i} u
\quad \text{on } \partial \Omega \cap B_l
\end{equation}
for $i = 1,\dots,n-1$.
For $u \in W^2_q(\Omega)^n$ satisfying $\nu \cdot u = 0$ on $\partial \Omega$ we obtain
\begin{equation} \label{eq:tangential derivative b}
0
= \partial_{\tau_i} (\nu \cdot u)
= u \cdot \partial_{\tau_i} \nu + \nu \cdot \partial_{\tau_i} u
\quad \text{on } \partial \Omega \cap B_l.
\end{equation}
Utilizing \eqref{eq:tangential derivative a} and \eqref{eq:tangential
derivative b} and writing $(\transA{\nabla u}) \nu = \partial_\nu u$, we
deduce
\begin{equation} \label{eq:representation Dpm}
[\Dpm(u) \nu]_i = [\partial_\nu u]_i \mp (\partial_{\tau_i} \nu) \cdot u
\quad \text{on } \partial \Omega \cap B_l \text{ for } i = 1,\dots,n-1.
\end{equation}

Now, \eqref{eq:boundary c} and \eqref{eq:representation Dpm} yield
\begin{equation*}
\begin{split}
& \Pi_\tau \Dpm(u) \nu \\
& = \Pi_\tau S^{-1} [\Dpm(u) \nu]_{1,\dots,n} \\
& = \Pi_\tau S^{-1} \big( [\Dpm(u) \nu]_1,\dots,[\Dpm(u) \nu]_{n-1},0 \big)^T \\
& = \Pi_\tau S^{-1} \big( \big( [\partial_\nu u]_1,\dots,[\partial_\nu u]_{n-1},0 \big)^T
\mp \big( (\partial_{\tau_1} \nu) \cdot u,\dots,(\partial_{\tau_{n-1}} \nu) \cdot u,0 \big)^T \big) \\
& = \Pi_\tau (\partial_\nu u \mp S^{-1} R u)
\quad \text{on } \partial \Omega \cap B_l,
\end{split}
\end{equation*}
where $R := (\partial_{\tau_1} \nu,\dots,\partial_{\tau_{n-1}} \nu,0)$.
Applying again $\nu \cdot u = 0$ on $\partial \Omega$, in combination
with \eqref{eq:Dm invariant under tangential projection} this gives
\begin{equation*}
\Pi_\tau \Dp(u) \nu
= \Dm(u) \nu - 2 \Pi_\tau S^{-1} R u
\quad \text{on } \partial \Omega.
\end{equation*}

It remains to prove that there exists an extension $A \in W^1_\infty(\Omega)^{n \times n}$ of $-2 S^{-1} R$.
Therefore, we first consider the entries of $S^{-1}$. We have shown in \eqref{eq:extension nu} that there exists an extension
$\widebar{\nu} \in W^2_\infty(\Omega)^n$ of $\nu$. In the same way we can establish an extension $\widebar{\tau}_i \in W^2_\infty(\Omega)^n$ of $\tau_i$ for $i = 1,\dots,n-1$
and the corresponding extension $\partial_{\widebar{\tau}_i}$ of the tangential derivative operator $\partial_{\tau_i}$.

A representation of $\tau^i$ is given by $\tau^i = \sum_{k=1}^{n-1} g^{ik} \tau_k$, where for the Gram matrix
$G := (\tau_j \cdot \tau_k)_{j,k=1,\dots,n-1}$
we define
$(g^{jk})_{j,k=1,\dots,n-1} := G^{-1}$
as its inverse (cf.~\cite{PS16}).
In \eqref{eq:estimate Gram matrix} we have established a uniform upper bound for $\| G^{-1} \|_{1,\infty}$, so we also have an extension $\widebar{\tau}^i \in W^2_\infty(\Omega)^n$ of $\tau^i$ for $i = 1,\dots,n-1$.
Now, considering the entries of $R$, we obtain that $\partial_{\tau_i} \nu$ for $i=1,\dots,n-1$ can be written as the directional derivative of the extension $\widebar{\nu}$ in direction of $\tau_i$.
Since $\widebar{\nu} \in W^2_\infty(\Omega)^n$, we conclude $\partial_{\widebar{\tau}_i} \widebar{\nu} \in W^1_\infty(\Omega)^n$.
Summarizing, we have extensions of $S^{-1}$ and $R$, hence also of $-2 S^{-1} R$, in $W^1_\infty(\Omega)^{n \times n}$.
\end{proof}

\begin{proof}[Proof of Theorem~\thref{thm:Stokes partial slip}]
We start with proving \eqref{thm:Stokes partial slip 2}.
Let initially $f \in L_{q,\sigma}(\Omega)$, choose $\lambda_0 = \lambda_0(n,q,\theta,\Omega)$ and $C = C(n,q,\theta,\Omega)$ such that the conditions of Theorem~\ref{thm:Stokes perfect slip inhom}\eqref{thm:Stokes perfect slip inhom 2} are satisfied and let $\lambda \in \Sigma_\theta$, $|\lambda| \ge \lambda_0$.
Let $A \in W^1_\infty(\Omega)^{n \times n}$ be the matrix 
constructed in Lemma~\ref{thm:matrix different partial slip b.c.}.

We define the Banach spaces
\begin{equation*}
\begin{split}
X &:= \big\{ (u,\transB{\nabla p}) \in [W^2_q(\Omega)^n \cap L_{q,\sigma}(\Omega)] \times \modGradq : (\lambda - \Delta)u + \transB{\nabla p} \in L_{q,\sigma}(\Omega) \big\}, \\
Y &:= L_{q,\sigma}(\Omega) \times \big\{ \Pi_\tau \Tr g : g \in W^1_q(\Omega)^n \big\} \\
\end{split}
\end{equation*}
with $\lambda$-dependent norms 
\begin{equation*}
\begin{split}
\| (u,\transB{\nabla p}) \|_X &:= \| (\lambda u,\sqrt{\lambda} \nabla u,\nabla^2 u,\nabla p) \|_q, \\
\| (f,a) \|_Y &:= \| f \|_q + \inf \{ \| (\sqrt{\lambda} g,\nabla g) \|_q : g \in W^1_q(\Omega)^n, a = \Pi_\tau \Tr g \}.
\end{split}
\end{equation*}
We further define the operators
\begin{equation*}
\begin{split}
S &: X \longrightarrow Y, \quad (u,\transB{\nabla p}) \longmapsto ((\lambda - \Delta)u + \transB{\nabla p},\Tr \Dm(u) \nu),
\\
P_- &: X \longrightarrow Y, \quad (u,\transB{\nabla p}) \longmapsto
(0,\Pi_\tau \Tr \alpha u),
\\
P_+ &: X \longrightarrow Y, \quad (u,\transB{\nabla p}) \longmapsto
(0,\Pi_\tau \Tr (Au + \alpha u)).
\end{split}
\end{equation*}
The statement of Theorem~\ref{thm:Stokes partial slip} for $f \in
L_{q,\sigma}(\Omega)$, $g \in W^1_q(\Omega)^n$ then means that
\begin{equation} \label{eq:operator S+P}
S + P_\pm: X \longrightarrow Y
\end{equation}
is bijective such that $(S + P_\pm)^{-1}$ is bounded, uniformly in $\lambda$.
More precisely, the related continuity constant of $(S + P_\pm)^{-1}$ is
only allowed to depend on $n,q,\theta,\Omega$, but not on 
$|\lambda| \ge \lambda_0$.
Besides, in \eqref{eq:operator S+P} the operator $S + P_-$ relates to $\eqref{eq:Stokes partial slip Dpm}_-$ while $S + P_+$ corresponds to $\eqref{eq:Stokes partial slip Dpm}_+$.

Theorem~\ref{thm:Stokes perfect slip inhom}\eqref{thm:Stokes perfect slip inhom 2} gives that $S$ is bijective and for $(f,a) \in Y$, $(u,\transB{\nabla p}) := S^{-1} (f,a)$ and any $g \in W^1_q(\Omega)^n$ satisfying $a = \Pi_\tau \Tr g$ we have
\begin{equation*}
\| (\lambda u, \sqrt{\lambda} \nabla u, \nabla^2 u, \nabla p) \|_q
\le C \| (f,\sqrt{\lambda} g,\nabla g) \|_q.
\end{equation*}
Consequently,
\begin{equation} \label{eq:partial slip resolvent estimate}
\| (u,\transB{\nabla p}) \|_X
\le C \| (f,a) \|_Y.
\end{equation}
Next, we prove that the two operators $P_\pm$ are continuous with
\begin{equation} \label{eq:Perturbation estimate}
\| P_\pm \|_{X \rightarrow Y} \le \frac{C' + |\alpha|}{\sqrt{|\lambda|}},
\end{equation}
where $C' = C'(n,q,\Omega) > 0$.
The definition of the norm in $Y$ directly gives
\begin{equation} \label{eq:Perturbation estimate B}
\| P_-(u,\transB{\nabla p}) \|_Y \le |\alpha| \| (\sqrt{\lambda}u,\nabla u) \|_q,
\end{equation}
for all $(u,\transB{\nabla p}) \in X$.
We obtain the same for the second part of $P_+$, i.e.,
\begin{equation*}
\| (0,\Pi_\tau \Tr \alpha u) \|_Y \le | \alpha | \| (\sqrt{\lambda} u,\nabla u) \|_q.
\end{equation*}
Lemma~\ref{thm:matrix different partial slip b.c.} yields for the first part
\begin{equation*}
\| (0,\Pi_\tau \Tr A u) \|_Y
\le \| A \|_{1,\infty} \| (\sqrt{\lambda} u,\nabla u) \|_q
\end{equation*}
if $|\lambda| \ge 1$.
In total we obtain \eqref{eq:Perturbation estimate}.

We now increase $\lambda_0 = \lambda_0(n,q,\theta,\Omega)$
to some $\lambda_0 = \lambda_0(n,q,\theta,\Omega,\alpha)$ so that we have
$\lambda_0 \ge \max\{1,(2 C)^2 (C' + |\alpha|)^2\}$, where $C$ and $C'$
are the constants from \eqref{eq:partial slip resolvent estimate} and
\eqref{eq:Perturbation estimate} respectively.
Then \eqref{eq:Perturbation estimate} yields for $|\lambda| \ge \lambda_0$
\begin{equation*}
\| P_\pm \|_{X \rightarrow Y} \le \frac{1}{2C}.
\end{equation*}
By a standard Neumann series argument $S + P_\pm$ is an isomorphism and we have
\begin{equation} \label{eq:XY estimate}
\| (u,\transB{\nabla p}) \|_X \le C \frac{1}{1 - C \| P_\pm \|_{X \rightarrow Y}} \| (f,a) \|_Y
\le 2 C \| (f,a) \|_Y
\end{equation}
for all $(f,a) \in Y$ and $(u,\transB{\nabla p}) = (S + P_\pm)^{-1}
(f,a)$. For any $g \in W^1_q(\Omega)^n$ we have $(f,a) \in Y$, where 
$a := \Pi_\tau \Tr g$. In other words, \eqref{eq:XY estimate} implies 
\eqref{eq:resolvent estimate partial slip}
for the special case $f \in L_{q,\sigma}(\Omega)$.

The general case $f \in L_q(\Omega)^n$ is completely 
analogous to the last part of the proof of 
Theorem~\ref{thm:Stokes perfect slip inhom}.

In order to prove \eqref{thm:Stokes partial slip 1}, similar to the proof of \eqref{thm:Stokes partial slip 2}, consider the Banach spaces
\begin{equation}
\begin{split}
& X' := [W^2_q(\Omega)^n \cap L_{q,\sigma}(\Omega)] \times \Gradq, \\
& Y' := L_q(\Omega)^n \times \{ \Pi_\tau \Tr g : g \in W^1_q(\Omega)^n \}
\end{split}
\end{equation}
with the same norms as for $X$ and $Y$ and the operators $S: X' \rightarrow Y'$ and $P_-: X' \rightarrow Y'$ as defined above.
Then estimate \eqref{eq:Perturbation estimate B} is still valid so we have
\begin{equation} \label{eq:Perturbation estimate C}
\| P_- \|_{X' \rightarrow Y'} \le |\alpha|
\end{equation}
if $\lambda_0 \ge 1$.
Theorem~\ref{thm:Stokes perfect slip inhom}\eqref{thm:Stokes perfect
slip inhom 1} yields that $S: X' \rightarrow Y'$ is surjective. Hence,
there exists $\epsilon = \epsilon(n,q,\Omega,\lambda) > 0$ so that in
case $\| P_- \|_{X' \rightarrow Y'} < \epsilon$ the operator $S + P_-:
X' \rightarrow Y'$ is surjective as well. Estimate \eqref{eq:Perturbation
estimate C} then yields the assertion.
\end{proof}

\section{The Stokes semigroup: Proof of Theorem~\ref{thm:Stokes semigroup}}
\label{sec:instat}

Let $\Stokesqpm$ be the Stokes operator as defined in \eqref{eq:Stokes operator def}.

\begin{proof}[Proof of Theorem~\thref{thm:Stokes semigroup}]
For $\theta := \omega + \frac{\pi}{2}$ Theorem~\ref{thm:Stokes partial slip} yields some constants $\lambda_0 = \lambda_0(n,q,\theta,\Omega,\alpha)$ and $C = C(n,q,\theta,\Omega)$ so that for $\lambda \in \Sigma_\theta$, $|\lambda| \ge \lambda_0$ the resolvent $(\lambda - \Stokesqpm)^{-1}: L_{q,\sigma}(\Omega) \rightarrow L_{q,\sigma}(\Omega)$ exists and fulfills the resolvent estimate
\begin{equation*}
\| \lambda (\lambda - \Stokesqpm)^{-1} f \|_q
\le C \| f \|_q \quad \forall f \in L_{q,\sigma}(\Omega).
\end{equation*}
Choosing $\delta = \delta(\theta) \in (0,1)$ small enough, we obtain that
$- \frac{1}{\delta} \lambda_0 + \Stokesqpm$ is the generator of a strongly continuous bounded analytic semigroup with angle $\omega$.

Now, we prove estimates \eqref{thm:estimate Stokes semigroup 1} and \eqref{thm:estimate Stokes semigroup 2} of Theorem~\ref{thm:Stokes semigroup}.
Let $\beta := n(\frac{1}{p} - \frac{1}{q}) \in [0,2)$. Then for the Bessel-potential space
\begin{equation} \label{eq:Sobolev embedding}
H^\beta_p(\mathbb{R}^n) = [L_p(\mathbb{R}^n),W^2_p(\mathbb{R}^n)]_\frac{\beta}{2}
\end{equation}
we have the Sobolev embedding
$H^\beta_p(\mathbb{R}^n) \subset L_q(\mathbb{R}^n)$
(see, e.g.,~\cite{Tri78})
with some embedding constant $C_e = C_e(n,q,p) > 0$,
since the condition $p \le q$, $\frac{n}{p} - \beta \le \frac{n}{q}$ is satisfied.
Let $t \in (0,T)$ and $f \in L_{p,\sigma}(\Omega)$ and denote by $E$ the extension operator from Lemma~\ref{thm:Stein-Extension}.
For vector-valued functions $v$, by $Ev$ we mean componentwise application of the extension operator $E$.
Then we conclude
\begin{equation} \label{eq:Stokes semigroup estimate A}
\begin{split}
\| e^{t \Stokesppm} f \|_{L_q(\Omega)^n}
& \le \| E e^{t \Stokesppm} f \|_{L_q(\mathbb{R}^n)^n} \\
& \le C_e \| E e^{t \Stokesppm} f \|_{H^\beta_p(\mathbb{R}^n)^n} \\
& \le C_e \| E e^{t \Stokesppm} f \|_{L_p(\mathbb{R}^n)^n}^{1 - \frac{\beta}{2}}
        \| E e^{t \Stokesppm} f \|_{H^2_p(\mathbb{R}^n)^n}^{\frac{\beta}{2}} \\
& \le C_e \| E \| \| e^{t \Stokesppm} f \|_{L_p(\Omega)^n}^{1 - \frac{\beta}{2}}
        \| e^{t \Stokesppm} f \|_{H^2_p(\Omega)^n}^{\frac{\beta}{2}},
\end{split}
\end{equation}
where $\| E \|$ denotes the maximum of the operator norms of \eqref{eq:Stein-Extension} for $k \in \{ 0,2 \}$.
Fix any $0 < \theta < \pi$ and choose $\lambda_0 = \lambda_0(n,p,\theta,\Omega,\alpha) \ge 1$ and $C = C(n,p,\theta,\Omega) > 0$ such that the conditions of Theorem~\ref{thm:Stokes partial slip} are satisfied.
Theorem~\ref{thm:Stokes partial slip} and Lemma~\ref{thm:Stokes operator} 
then yield
\begin{equation*}
\Big\| \Big( \frac{1}{\delta} \lambda_0 - \Stokesppm \Big)^{\!\! -1} \Big\|_{L_p(\Omega)^n \rightarrow H^2_p(\Omega)^n} \le C
\end{equation*}
(since $\frac{1}{\delta} \lambda_0 \ge \lambda_0 \ge 1$).
By the fact that $\bigl(e^{t (\Stokesppm - \frac{1}{\delta}
\lambda_0)}\bigr)_{t \ge 0}$
is a bounded analytic strongly continuous semigroup we deduce
\begin{equation} \label{eq:Stokes semigroup estimate C}
\begin{split}
\| e^{t \Stokesppm} f \|_{H^2_p(\Omega)^n}
& = \Big\| \Big( \frac{1}{\delta} \lambda_0 - \Stokesppm \Big)^{\!\! -1} \Big( \frac{1}{\delta} \lambda_0 - \Stokesppm \Big) e^{t \Stokesppm} f \Big\|_{H^2_p(\Omega)^n} \\
& \le C e^{\frac{1}{\delta} \lambda_0 T} \Big\| \Big( \frac{1}{\delta} \lambda_0 - \Stokesppm \Big) e^{t (\Stokesppm - \frac{1}{\delta}\lambda_0)} f \Big\|_{L_p(\Omega)^n} \\
& \le C C' e^{\frac{1}{\delta} \lambda_0 T} \frac{1}{t} \| f \|_{L_p(\Omega)^n}
\end{split}
\end{equation}
with $C' = C'(n,p,\theta,\Omega,\alpha)>0$.
Now, \eqref{eq:Stokes semigroup estimate A} 
and \eqref{eq:Stokes semigroup estimate C} imply
\begin{equation*}
\| e^{t \Stokesppm} f \|_{L_q(\Omega)^n}
\le C e^{\frac{1}{\delta} \lambda_0 T} t^{-\frac{\beta}{2}} \| f \|_{L_p(\Omega)^n}
\end{equation*}
for $t\in(0,T)$ with $C=C(n,p,q,\theta,\Omega,\alpha)>0$.
Hence, \eqref{thm:estimate Stokes semigroup 1} is proved.

In order to see \eqref{thm:estimate Stokes semigroup 2}, let $t \in (0,T)$ and $f \in L_{p,\sigma}(\Omega)$ again, where we have
$\beta = n(\frac{1}{p} - \frac{1}{q}) \in [0,1)$ this time.
The condition $p \le q$, $\frac{n}{p} - \beta \le \frac{n}{q}$ for
Sobolev's embedding is still satisfied. So, we have
$H^{\beta + 1}_p(\mathbb{R}^n) \subset H^1_q(\mathbb{R}^n)$
with some embedding constant $C_e = C_e(n,q,p) > 0$ as above.
Furthermore, the condition $\beta < 1$ gives that \eqref{eq:Sobolev
embedding} is fulfilled with $\beta + 1$ instead of $\beta$. This yields
\begin{equation*}
\begin{split}
\| \nabla e^{t \Stokesppm} f \|_{L_q(\Omega)^{n^2}}
& \le \| e^{t \Stokesppm} f \|_{H^1_q(\Omega)^n} \\
& \le \| E e^{t \Stokesppm} f \|_{H^1_q(\mathbb{R}^n)^n} \\
& \le C_e \| E e^{t \Stokesppm} f \|_{H^{\beta + 1}_p(\mathbb{R}^n)^n} \\
& \le C_e \| E e^{t \Stokesppm} f \|_{L_p(\mathbb{R}^n)^n}^{1 - \frac{\beta + 1}{2}}
        \| E e^{t \Stokesppm} f \|_{H^2_p(\mathbb{R}^n)^n}^{\frac{\beta + 1}{2}} \\
& \le C_e \| E \| \| e^{t \Stokesppm} f \|_{L_p(\Omega)^n}^{1 - \frac{\beta + 1}{2}}
        \| e^{t \Stokesppm} f \|_{H^2_p(\Omega)^n}^{\frac{\beta + 1}{2}}.
\end{split}
\end{equation*}
Analogously as above by applying \eqref{eq:Stokes semigroup estimate C} 
we conclude
\begin{equation*}
\begin{split}
\| \nabla e^{t \Stokesppm} f \|_{L_q(\Omega)^{n^2}}
& \le C e^{\frac{1}{\delta} \lambda_0 T}t^{-\frac{\beta + 1}{2}} 
\| f \|_{L_p(\Omega)^n}.
\end{split}
\end{equation*}
for $t\in(0,T)$ with $C=C(n,p,q,\theta,\Omega,\alpha)>0$.
\end{proof}

\section{The Navier-Stokes equations: Proof of Theorem~\ref{thm:NSE}}
\label{sec:ns}

Based on the $L_p$-$L_q$ estimates derived in 
Theorem~\thref{thm:Stokes semigroup},
we can apply a standard fixed point argument.

\begin{proof}[Proof of Theorem~\thref{thm:NSE}]
Let $u_0 \in L_{q,\sigma}(\Omega)$.
For $M > 0$ and $T > 0$ we define
$X_{M,T}$ as the space of functions $u$ satisfying \eqref{eq:mild solution condition} and
$\| u \|_T \le M \| u_0 \|_q$,
where $\| u \|_T
:= \sup_{t \in [0,T]} \| u(t) \|_q + \sup_{t \in [0,T]} \sqrt{t} \| \nabla u(t) \|_q$
and
\begin{equation*}
H(u)(t)
:= e^{t \Stokesqpm} u_0 - \int_0^t e^{(t - s) \Stokes{q/2}} \Proj_{q/2} \sum_{j=1}^n \partial_j (u^{j}(s) u(s)) \D s
\end{equation*}
for $u \in X_{M,T}$ and $t \in [0,T]$.
To prove that $H$ is a contraction is standard, so we will be brief in details.
Assuming $T < 1$, we can apply Theorem~\ref{thm:Stokes
semigroup}\eqref{thm:estimate Stokes semigroup 1} with $T = 1$ and $p =
\frac{q}{2}$ (resp.\ with $p = q$ for the first term) to obtain a constant $C = C(n,q,\Omega,\alpha) > 0$ so that
\begin{equation} \label{eq:Contraction A}
\| H(u)(t) \|_q
\le C \Big( \| u_0 \|_q + \sum_{j=1}^n \int_0^t (t - s)^{-\frac{n}{2q}} \| \Proj_{q/2} \partial_j (u^{j}(s) u(s)) \|_\frac{q}{2} \D s \Big).
\end{equation}
The continuity of $\Proj_{q/2}$ on $L_{q/2}(\Omega)^n$, the fact that $\frac{n}{2q} < 1$ and Hölder's estimate yield a constant $C' = C'(n,q,\Omega) > 0$ so that
{\allowdisplaybreaks
\begin{align}
\sum_{j=1}^n & \int_0^t (t - s)^{-\frac{n}{2q}} \| \Proj_{q/2} \partial_j (u^{j}(s) u(s)) \|_\frac{q}{2} \D s \\
& \le C' \left( \sup_{\tau \in [0,T]} \| u(\tau) \|_q \right) \left( \sup_{\tau \in [0,T]} \sqrt{\tau} \| \nabla u(\tau) \|_q \right)
\int_0^t \frac{ (t - s)^{-\frac{n}{2q}} }{\sqrt{s}} \D s \\
& \le C' \left( \sup_{\tau \in [0,T]} \| u(\tau) \|_q + \sup_{\tau \in [0,T]} \sqrt{\tau} \| \nabla u(\tau) \|_q \right)^2
\int_0^t \frac{ (t - s)^{-\frac{n}{2q}} }{\sqrt{s}} \D s \\
& = C' \| u \|_T^2 \int_0^t \frac{ (t - s)^{-\frac{n}{2q}} }{\sqrt{s}} \D s \\
& \le C'' M^2 \| u_0 \|_q^2 T^{\frac{1}{2} - \frac{n}{2q}}.
\end{align}
}
Therefore
\begin{equation} \label{eq:Fixpoint estimate 1}
\| H(u)(t) \|_q
\le C \| u_0 \|_q + C C'' M^2 T^{\frac{1}{2} - \frac{n}{2q}} \| u_0 \|_q^2
\end{equation}
for $M > 0$, $0 < T < 1$, $u \in X_{M,T}$ and $t \in [0,T]$.

Next, we derive a similar estimate for
\begin{equation*}
\nabla H(u)(t)
= \nabla e^{t \Stokesqpm} u_0 - \int_0^t \nabla e^{(t - s) \Stokes{q/2}} \Proj_{q/2} \sum_{j=1}^n \partial_j (u^{j}(s) u(s)) \D s
\end{equation*}
for $u \in X_{M,T}$ and $t \in [0,T]$.
We assume $T < 1$ again and apply Theorem~\ref{thm:Stokes semigroup}\eqref{thm:estimate Stokes semigroup 2} with $T = 1$ and $p = \frac{q}{2}$ (resp.\ with $p = q$ for the first term) to receive a constant $C = C(n,q,\Omega,\alpha) > 0$ so that
\begin{equation} \label{eq:Fixpoint estimate 2}
\begin{split}
\sqrt{t} \| \nabla H(u)(t) \|_q
& \le C \Big( \| u_0 \|_q + \sum_{j=1}^n \int_0^t \sqrt{t} (t - s)^{-\frac{n}{2q} - \frac{1}{2}} \| \Proj_{q/2} \partial_j (u^{j}(s) u(s)) \|_\frac{q}{2} \D s \Big) \\
& \le C \| u_0 \|_q + C C''' M^2 T^{\frac{1}{2} - \frac{n}{2q}} \| u_0 \|_q^2
\end{split}
\end{equation}
for $M > 0$, $u \in X_{M,T}$ and $t \in [0,T]$.
The constant $C''' = C'''(n,q,\Omega) > 0$ results from the 
continuity of $\Proj_{q/2}$, Hölder's estimate and the 
fact that $\frac{n}{2q} < \frac{1}{2}$.

Let $0 < T < T'$ with
$T' := \min \left\{ 1 , \left( \frac{1}{4 C^2 C'' \| u_0 \|_q} \right)^\frac{2}{1 - \frac{n}{q}} , \left( \frac{1}{4 C^2 C''' \| u_0 \|_q} \right)^\frac{2}{1 - \frac{n}{q}} \right\}$
and let $M \ge 2C$.
Then \eqref{eq:Fixpoint estimate 1} and \eqref{eq:Fixpoint estimate 2} yield
\begin{equation} \label{eq:map into itself}
H: X_{M,T} \rightarrow X_{M,T},
\end{equation}
i.e., $H$ maps $X_{M,T}$ into intself.

We proceed to prove that $H: X_{M,T} \rightarrow X_{M,T}$ satisfies a contraction estimate for $T > 0$ small enough.
Let $u,v \in X_{M,T}$ and $t \in [0,T]$.
For $0 \le s \le t$ we can estimate
\begin{equation*}
\| \partial_j \big( u^{j}(s) u(s) - v^{j}(s) v(s) \big) \|_\frac{q}{2}
\le \frac{4 M}{\sqrt{s}} \| u_0 \|_q \| u - v \|_T.
\end{equation*}
Similar to \eqref{eq:Fixpoint estimate 1} and \eqref{eq:Fixpoint estimate 2} this gives
\begin{equation*}
\| H(u)(t) - H(v)(t) \|_q
\le C C'''' M T^{\frac{1}{2} - \frac{n}{2q}} \| u_0 \|_q \| u - v \|_T
\end{equation*}
and
\begin{equation*}
\| \sqrt{t} \nabla \big( H(u)(t) - H(v)(t) \big) \|_q
\le C C'''' M T^{\frac{1}{2} - \frac{n}{2q}} \| u_0 \|_q \| u - v \|_T
\end{equation*}
for $t \in [0,T]$,
where $C'''' = C''''(n,q,\Omega) > 0$.

Thus, we have
\begin{equation} \label{eq:Contraction estimate}
\| H(u) - H(v) \|_T
\le \frac{1}{2} \| u - v \|_T
\end{equation}
if $0 < T < T''$, where
$T'' := \min \left\{ 1 , \left( \frac{1}{4 C^2 C'''' \| u_0 \|_q} \right)^\frac{2}{1 - \frac{n}{q}} \right\}$.

Summarizing, for $M \ge 2C$ and $0 < T < T_0$ with $T_0 := \min \{
	T',T'' \}$, \eqref{eq:map into itself} and \eqref{eq:Contraction
	estimate} yield that $H: X_{M,T} \rightarrow X_{M,T}$ is a
	contraction. The contraction mapping principle implies
	the assertion.
\end{proof}

\section{Appendix A}

Since we could not find an appropriate result on the resolvent problem
for the heat equation subject to Neumann boundary conditions in the case
of general noncompact boundaries for the whole scale $1 < q < \infty$ in 
existing literature, we state the result here.
The proof is very similar to the proof of Theorem~\ref{thm:HE with PS}.
In fact, it is somewhat easier, since the Neumann boundary condition is a condition for scalar functions $u: \Omega \longrightarrow \mathbb{R}$ instead of vector fields.
In particular, there is no distinction between boundary conditions in tangential and normal direction.
As a consequence, the multiplication of the matrix $\transA{\nabla
\Phi}$ to the boundary terms in the proof of Theorem~\ref{thm:HE bent
rotated shifted half space} is not required.
By this fact, similar to the localization technique for Dirichlet boundary conditions (see~\cite{Kun03}), uniform $C^{1,1}$-boundary regularity is sufficient.

\begin{lemma} \label{thm:Neumann Laplace}
Let $\Omega \subset \mathbb{R}^n$ be a uniform $C^{1,1}$-domain,
$n \ge 2$, $1 < q < \infty$ and $0 < \theta < \pi$.
Then there exist $\lambda_0 = \lambda_0(n,q,\theta,\Omega) > 0$ and $C = C(n,q,\theta,\Omega) > 0$ such that for $\lambda \in \Sigma_\theta$, $|\lambda| \ge \lambda_0$ the problem
\begin{equation} \label{eq:HE Neumann}
\left\{
\begin{array}{rll}
\lambda u - \Delta u & = f & \text{in } \Omega \\
\partial_\nu u & = g & \text{on } \partial \Omega,
\end{array}
\right.
\end{equation}
for all $f \in L_q(\Omega)$ and $g \in W^1_q(\Omega)$, has a unique solution $u \in W^2_q(\Omega)$ and this solution fulfills the resolvent estimate
\begin{equation} \label{eq:HE Neumann resolvent estimate}
\| (\lambda u, \sqrt{\lambda} \nabla u, \nabla^2 u) \|_q
\le C \| (f, \sqrt{\lambda} g, \nabla g) \|_q.
\end{equation}
\end{lemma}

\begin{lemma} \label{thm:Stein-Extension}
Let $\Omega \subset \mathbb{R}^n$ be a uniform $C^{0,1}$-domain (i.e., a uniform Lipschitz domain) and $n \ge 2$. Then there exists a linear operator $E$ mapping real-valued functions on $\Omega$ to real-valued functions on $\mathbb{R}^n$ such that $Ef|_\Omega = f$ holds for any function $f$ on $\Omega$ (i.e., $E$ is an extension operator) and such that
\begin{equation} \label{eq:Stein-Extension}
E: W^k_q(\Omega) \longrightarrow W^k_q(\mathbb{R}^n)
\end{equation}
is continuous for all $1 \le q \le \infty$ and all $k \in \mathbb{N}_0$.
\end{lemma}

\begin{proof}
See~\cite{Ste70}, Thm.~VI.3.1/5. The condition for $\Omega$ to be a uniform $C^{0,1}$-domain is exactly the condition in~\cite{Ste70} for $\partial \Omega$ to be minimally smooth.
\end{proof}

\begin{definition} \label{perturbed cones}
For $n \ge 2$ we call a domain $\Omega \subset \mathbb{R}^n$, satisfying the segment property
(cf.~\cite{Ada75}), a perturbed cone if
there exists a (convex or concave) cone $\Omega_C \subset \mathbb{R}^n$ (where we assume the apex to be at the origin, w.l.o.g.) and $R > 0$ so that $\Omega \setminus B_R(0) = \Omega_C \setminus B_R(0)$, where the maximal cone $\Omega_C = \mathbb{R}^n$ and the minimal cone $\Omega_C = \emptyset$ are admitted.
\end{definition}

\begin{lemma} \label{thm:perturbed cones}
Let $n \ge 2$ and let $\Omega \subset \mathbb{R}^n$ be a perturbed cone.
Then
$C_c^\infty(\overline{\Omega}) \subset \widehat{W}^1_q(\Omega)$ is dense
for all $1 \le q < \infty$.
Hence, Assumption~\thref{thm:Assumption C} is valid for $\Omega$ and for all $1 < q < \infty$.
\end{lemma}

\begin{proof}
We first convince ourselves that it is sufficient to prove that
$\widehat{W}_{c,q}^1(\Omega)$, consisting of those functions in $\widehat{W}_q^1(\Omega)$ having compact support in $\overline{\Omega}$, is a dense subspace of $\widehat{W}_q^1(\Omega)$. In fact, the (algebraic) inclusion $\widehat{W}_{c,q}^1(\Omega) \subset W^1_q(\Omega)$ and the density of $C_c^\infty(\overline{\Omega}) \subset W^1_q(\Omega)$ (see~\cite{Ada75}, Thm.\ 3.18; 
note that $\Omega$ is assumed to have the segment property) yield that
$C_c^\infty(\overline{\Omega}) \subset \widehat{W}_{c,q}^1(\Omega)$ is dense.
Hence, for some given function $p \in \widehat{W}_q^1(\Omega)$ it remains to find a sequence $(\psi_k)_{k \in \mathbb{N}}$ in $\widehat{W}_{c,q}^1(\Omega)$ such that
$\| \nabla \psi_k - \nabla p \|_q \xrightarrow{k \to \infty} 0$.

Let $\Chi \in C^\infty(\mathbb{R}^n)$ so that $\Chi = 1$ in $\widebar{B}_{1/2}(0)$, $\Chi = 0$ in $\mathbb{R}^n \setminus B_1(0)$ and $0 \le \Chi \le 1$.
Let $\Chi_k(x) := \Chi(\frac{x}{k})$ for $x \in \mathbb{R}^n$ and $k \in \mathbb{N}$.
Then we have $\Chi_k = 1$ in $\widebar{B}_{k/2}(0)$, $\Chi = 0$ in $\mathbb{R}^n \setminus B_k(0)$ and $0 \le \Chi \le 1$.
Setting $M := \| \nabla \Chi \|_\infty$, we further have
\begin{equation} \label{eq:estimate cutoff function}
\| \nabla \Chi_k \|_\infty \le \frac{M}{k}.
\end{equation}
Let $R_k := B_k(0) \setminus \widebar{B}_{k/2}(0)$ be the $k$-th annulus.
Due to the assumption on $\Omega$ there exists $N \in \mathbb{N}$ so that for the scaling
$\phi_k: \Omega \cap R_N \rightarrow \Omega \cap R_{k N}$, $x \mapsto k x$
we have
\begin{equation} \label{eq:scaling domain}
\phi_k(\Omega \cap R_N) = \Omega \cap R_{k N}
\end{equation}
for all $k \in \mathbb{N}$.

Now for $p \in \widehat{W}_q^1(\Omega)$ we define
$\psi_k := \Chi_{k N} \big( p - \frac{1}{\lambda_n(\Omega \cap R_{k N})} \int_{\Omega \cap R_{k N}} p \D \lambda_n \big)$. Then $\psi_k$ is a function in $\widehat{W}_{c,q}^1(\Omega)$ for all $k \in \mathbb{N}$ and we have
\begin{equation*}
\begin{split}
\| \nabla \psi_k - \nabla p \|_q
& \le \| \nabla \Chi_{k N} \|_\infty \Big\| p - \frac{1}{\lambda_n(\Omega \cap R_{k N})} \int_{\Omega \cap R_{k N}} p \D \lambda_n \Big\|_{q,\Omega \cap R_{k N}} \\
& + \| 1 - \Chi_{k N} \|_\infty \| \nabla p \|_{q,\Omega \setminus B_{k N /2}(0)}.
\end{split}
\end{equation*}
Now, using \eqref{eq:estimate cutoff function}, we can estimate
$\| 1 - \Chi_{k N} \|_\infty \le 1$ and $\| \nabla \Chi_{k N} \|_\infty \le \frac{M}{k N}$ as well as
{\allowdisplaybreaks
\begin{align*}
& \Big\| p - \frac{1}{\lambda_n(\Omega \cap R_{k N})} \int_{\Omega \cap R_{k N}} p \D \lambda_n \Big\|_{q,\Omega \cap R_{k N}}^q \\
& = k^n \int_{\Omega \cap R_N} \Big| p \circ \phi_k - \frac{k^n}{\lambda_n(\Omega \cap R_{k N})} \int_{\Omega \cap R_N} p \circ \phi_k \D \lambda_n \Big|^q \D \lambda_n \\
& = k^n \Big\| p \circ \phi_k - \frac{k^n}{\lambda_n(\Omega \cap R_{k N})} \int_{\Omega \cap R_N} p \circ \phi_k \D \lambda_n \Big\|_{q,\Omega \cap R_N}^q \\
& \le k^n C^q \| \nabla (p \circ \phi_k) \|_{q,\Omega \cap R_N}^q \\
& = k^n C^q \int_{\Omega \cap R_N} |k(\nabla p \circ \phi_k)|^q \D \lambda_n \\
& = k^n k^q \frac{1}{k^n} C^q \int_{\Omega \cap R_{k N}} |\nabla p|^q \D \lambda_n \\
& = k^q C^q \| \nabla p \|_{q,\Omega \cap R_{k N}}^q,
\end{align*}
}%
using \eqref{eq:scaling domain},
where $C = C(n,q,\Omega \cap R_N) > 0$ is the constant from the Poincaré inequality (see~\cite{Gal11}, Thm. II.5.4).
This results in
\begin{equation*}
\| \nabla \psi_k - \nabla p \|_q
\le \frac{M C}{N} \| \nabla p \|_{q,\Omega \cap R_{k N}} + \| \nabla p \|_{q,\Omega \setminus B_{k N /2}} \xrightarrow{k \to \infty} 0,
\end{equation*}
since $\nabla p \in L_q(\Omega)^n$.
\end{proof}

\begin{lemma} \label{thm:(epsilon,infty)-domains}
Let $n \ge 2$, $1 \le q < \infty$ and let $\Omega \subset \mathbb{R}^n$ be an $(\epsilon,\infty)$-domain for some $\epsilon > 0$, i.e., for all $x,y \in \Omega$ there exists a rectifiable curve $\gamma$ in $\Omega$ with length $l(\gamma)$, connecting $x$ and $y$, such that
\begin{equation*}
l(\gamma) < \frac{|x - y|}{\epsilon}
\end{equation*}
and
\begin{equation} \label{eq:Tube condition}
\mathrm{dist}(z,\partial \Omega) > \frac{\epsilon |x - z| |y - z|}{|x - y|} \quad \forall z \in \gamma.
\end{equation}
Condition \eqref{eq:Tube condition} says that there is a tube around
$\gamma$, lying in $\Omega$, such that in some point $z \in \gamma$ the
tube's width is of order $\min \{ |x - z|,|y - z| \}$ (cf.~\thcite{Chu92} and~\thcite{Jon81}).
Then $C_c^\infty(\overline{\Omega}) \subset \widehat{W}^1_q(\Omega)$ is dense.
\end{lemma}

\begin{proof}
Due to~\cite{Chu92}, Thm.\ 1.2, the conditions on $\Omega$ yield a continuous extension operator
$\Lambda: \widehat{W}^1_q(\Omega) \longrightarrow \widehat{W}^1_q(\mathbb{R}^n)$,
where we choose the weight $w = 1$.
Now, using the density of $C_c^\infty(\mathbb{R}^n) \subset
\widehat{W}^1_q(\mathbb{R}^n)$, the assertion is proved.
\end{proof}

\section{Appendix B: Traces and Gauß's theorem} \label{sec:Appendix Gauss}

\begin{lemma} \label{thm:density E_q}
Let $\Omega \subset \mathbb{R}^n$ be a domain satisfying the segment property (cf.~\thcite{Ada75}), $n \ge 2$ and $1 < q < \infty$.
Then $C_c^\infty(\overline{\Omega})^n \subset E_q(\Omega)$ is dense.
\end{lemma}

\begin{proof}
\emph{Step 1.}
Let $J_\epsilon \in C_c^\infty(\mathbb{R}^n)$ be the mollifier
from~\cite{Ada75}, Sec.~2.17, that is, $J_\epsilon(x) := \frac{1}{\epsilon^n} J(\frac{x}{\epsilon})$ for $\epsilon > 0$ and a function $J \in C_c^\infty(\mathbb{R}^n)$ satisfying $J(x) \ge 0$ for all $x \in \mathbb{R}^n$, $J(x) = 0$ for $|x| \ge 1$ and $\int_{\mathbb{R}^n} J(x) dx = 1$.
Following the arguments in the proof of~\cite{Ada75}, Lem.~3.15 (in particular, using~\cite{Ada75}, Lem.\ 2.18(c)), we obtain that for $u \in E_q(\Omega)$ and any subdomain $\Omega' \subset \subset \Omega$ (i.e., $\overline{\Omega'}$ is compact and $\overline{\Omega'} \subset \Omega$)
\begin{equation} \label{eq:mollifier property Eq}
J_\epsilon * u \xrightarrow{\epsilon \searrow 0} u \quad \text{in } E_q(\Omega')
\end{equation}
holds, where $J_\epsilon * u$ means convolution of $J_\epsilon$ with the trivial extension of $u$ to $\mathbb{R}^n$.

\emph{Step 2.}
Following the arguments in the proof of~\cite{Ada75}, Thm.~3.16, we establish density and continuity of the embedding
\begin{equation} \label{eq:density Eq}
W^1_q(\Omega)^n \cap C^\infty(\Omega)^n \subset E_q(\Omega)
\end{equation}
by using \eqref{eq:mollifier property Eq}.
Continuity of \eqref{eq:density Eq} is obvious.
Now let $u \in E_q(\Omega)$ and $\delta > 0$. Set
$\Omega_k := \{ x \in \Omega : |x| < k, ~\dist(x,\partial \Omega) > \frac{1}{k} \}$ for $k \in \mathbb{N}$ as well as $\Omega_0 = \Omega_{-1} = \emptyset$ and
$U_k := \Omega_{k+1} \cap (\overline{\Omega}_{k-1})^c$. Then the $U_k$, $k \in \mathbb{N}$ form an open cover of $\Omega$.
Let $(\psi_k)_{k \in \mathbb{N}}$ be a subordinated partition of unity, i.e., $\psi_k \in C_c^\infty(U_k)$, $0 \le \psi_k \le 1$ and $\sum_{k=1}^\infty \psi_k = 1$ on $\Omega$.
For $0 < \epsilon < \frac{1}{(k+1)(k+2)}$ we have
\begin{equation*}
\Supp(J_\epsilon * (\psi_k u)) \subset \Omega_{k+2} \cap (\Omega_{k-2})^c
=: V_k \subset \subset \Omega.
\end{equation*}
Now, we apply \eqref{eq:mollifier property Eq} to $\Omega' = V_k$:
Starting with some $k \in \mathbb{N}$, let $0 < \epsilon_k < \frac{1}{(k+1)(k+2)}$ such that
\begin{equation*}
\| J_{\epsilon_k} * (\psi_k u) - \psi_k u \|_{E_q(\Omega)}
= \| J_{\epsilon_k} * (\psi_k u) - \psi_k u \|_{E_q(V_k)}
< \frac{\delta}{2^k}.
\end{equation*}
Set $\Phi := \sum_{k=1}^\infty J_{\epsilon_k} * (\psi_k u)$ and
note that on any $\Omega' \subset \subset \Omega$ there is only a finite number of nonzero summands. For $x \in \Omega_k$ we have
\begin{equation*}
u(x) = \sum_{j=1}^{k+2} \psi_j(x) u(x)
\quad \text{and} \quad
\Phi(x) = \sum_{j=1}^{k+2} J_{\epsilon_j} * (\psi_j u)(x).
\end{equation*}
Hence $\Phi \in C^\infty(\Omega)$ and
\begin{equation*}
\| u - \Phi \|_{E_q(\Omega_k)}
\le \sum_{j=1}^{k+2} \| J_{\epsilon_j} * (\psi_j u) - \psi_j u \|_{E_q(\Omega)}
\le \delta.
\end{equation*}
By use of the monotone convergence theorem we conclude
\begin{equation*}
\| u - \Phi \|_{E_q(\Omega)} = \lim_{k \to \infty} \| u - \Phi \|_{E_q(\Omega_k)} \le \delta,
\end{equation*}
so embedding \eqref{eq:density Eq} is dense.

\emph{Step 3.}
The embedding $C_c^\infty(\overline{\Omega}) \subset W^1_q(\Omega) \cap
C^\infty(\Omega)$ is dense, due to~\cite{Ada75}, Thm.\ 3.18. Combining
this with the density and continuity of 
\eqref{eq:density Eq} yields the result.
\end{proof}

\begin{lemma}[Trace] \label{thm:Trace}
Let $\Omega \subset \mathbb{R}^n$ be a domain with uniform $C^{2,1}$-boundary, $n \ge 2$ and $1 \le q < \infty$.
Then the trace
\begin{equation*}
\Tr = \Tr_{\partial \Omega}: W^1_q(\Omega) \rightarrow W^{1-\frac{1}{q}}_q(\partial \Omega), \quad
\Tr u = u|_{\partial \Omega} ~\forall u \in C_c^\infty(\overline{\Omega})
\end{equation*}
is continuous. For $q > 1$ the trace is surjective with a continuous linear right inverse $\R: W^{1-1/q}_q(\partial \Omega) \rightarrow W^1_q(\Omega)$.
\end{lemma}

\begin{proof}
In case $1 < q < \infty$ we refer to~\cite{Mar87}, Thm.~2. Note that there 
$\Omega$ is merely assumed to be a uniform Lipschitz domain.
In case $q = 1$ we make use of the trace for bounded $C^1$-domains, constructed in~\cite{Eva10}, Thm.\ 5.5/1.
Choosing for all parts of the boundary $\partial \Omega \cap B_l$, $l \in \Gamma_1$ a bounded $C^1$-domain $U_l$ such that $\partial U_l \cap B_l = \partial \Omega \cap B_l$ and denoting by $\Tr_l$ the trace operator for $U_l$, we can define the trace of $u \in W^1_q(\Omega)$ as
\begin{equation*}
\Tr u := \sum_{l \in \Gamma_1} \Tr_l (\varphi_l^2 u).
\end{equation*}
Looking at the construction of $\Tr_l$ in the proof of~\cite{Eva10}, Thm.\ 5.5/1, we observe that the uniformity of the boundary $\partial \Omega$ yields that the continuity of $\Tr_l$ is uniform in $l \in \Gamma_1$.
Therefore we obtain a uniform estimate of the operators $\Tr_l$ in their operator
norm. This leads to
\begin{equation*}
\| \Tr u \|_{L_1(\partial \Omega)}
= \int_{\partial \Omega} \Big| \sum_{l \in \Gamma_1} \Tr_l(\varphi_l^2 u) \Big| \D \sigma \\
\le C \| u \|_{W^1_1(\Omega)}
\end{equation*}
with a constant $C = C(n,\Omega) > 0$, using \eqref{eq:uniformity} and
the condition that at most $\bar{N}$ of the balls $B_l$ have nonempty 
intersection.
\end{proof}

\begin{lemma}[Gauß's theorem in $W^1_1$] \label{thm:Gauss in W^1_1}
Let $\Omega \subset \mathbb{R}^n$ be a domain with uniform $C^{2,1}$-boundary, $n \ge 2$ and
let $u \in W^1_1(\Omega)^n$. Then we have
\begin{equation} \label{eq:Gauss in W^1_1}
\int_\Omega \Div u \D \lambda_n
= \int_{\partial \Omega} \nu \cdot u \D \sigma.
\end{equation}
\end{lemma}

\begin{proof}
In case $u \in C_c^\infty(\overline{\Omega})^n$ see, e.g.,~\cite{AE09}.
Since $C_c^\infty(\overline{\Omega}) \subset W^1_1(\Omega)$ is dense (see~\cite{Ada75}, Thm.\ 3.18), starting with some $u \in W^1_1(\Omega)^n$, we can find a sequence $(u_k)_{k \in \mathbb{N}} \subset C_c^\infty(\overline{\Omega})^n$ converging to $u$ in $W^1_1(\Omega)^n$.
Now, replacing $u$ in \eqref{eq:Gauss in W^1_1} by $u_k$, we see that the left-hand side converges to $\int_\Omega \Div u \D \lambda_n$ and, thanks to Lemma~\ref{thm:Trace}, the right-hand side converges to $\int_{\partial \Omega} \nu \cdot u \D \sigma$.
\end{proof}

\begin{lemma}[Green's formula in $W^1_q$] \label{thm:Greens formula in W^1_q}
Let $\Omega \subset \mathbb{R}^n$ be a domain with uniform $C^{2,1}$-boundary, $n \ge 2$ and $1 < q < \infty$.
For $u \in W^1_q(\Omega)$ and $v \in W^1_{q'}(\Omega)^n$ we have
\begin{equation*}
\int_\Omega u (\Div v) \D \lambda_n
= \int_{\partial \Omega} u (\nu \cdot v) \D \sigma - \int_\Omega \nabla u \cdot v \D \lambda_n.
\end{equation*}
\end{lemma}

\begin{proof}
Lemma~\ref{thm:Gauss in W^1_1} yields
$\int_\Omega \Div (u v) \D \lambda_n
= \int_{\partial \Omega} \nu \cdot (u v) \D \sigma$.
Using the representation
$\Div (u v) = \nabla u \cdot v + u(\Div v)$,
we obtain the statement.
\end{proof}

\begin{lemma}[Trace of the normal component] \label{thm:Trace of the normal component}
Let $\Omega \subset \mathbb{R}^n$ be a domain with uniform $C^{2,1}$-boundary, $n \ge 2$ and $1 < q < \infty$.
There exists a bounded linear operator
\begin{equation*}
\Tr_\nu: E_{q'}(\Omega) \longrightarrow W^{-\frac{1}{q'}}_{q'}(\partial \Omega)
\end{equation*}
such that for any $v \in W^1_{q'}(\Omega)^n$ we have
$\Tr_\nu v = \nu \cdot v|_{\partial \Omega}$ in
$W^{-\frac{1}{q'}}_{q'}(\partial \Omega)$ in the sense that
\begin{equation*}
\Tr_\nu v = \Big[ W^{1-\frac{1}{q}}_q(\partial \Omega) \ni g \mapsto \int_{\partial \Omega} g(\nu \cdot v) \D \sigma \Big].
\end{equation*}
For $v \in E_{q'}(\Omega)$, we denote by $\langle u , \nu \cdot v \rangle_{\partial \Omega}
:= \langle \Tr u , \Tr_\nu v \rangle_{\partial \Omega}$ the application of $\Tr_\nu v$ to some $g = \Tr u \in W^{1-1/q}_q(\partial \Omega)$, $u \in W^1_q(\Omega)$.
\end{lemma}

\begin{proof}
We can simply follow the arguments in~\cite{Soh01}, II.1.2, where the domain $\Omega$ is assumed to be bounded, to construct the trace of the normal component.
Let $g \in W^{1-1/q}_q(\partial \Omega)$ and $v \in W^1_{q'}(\Omega)^n$.
Then we have $\R g \in W^1_q(\Omega)$, so, using Lemma~\ref{thm:Greens formula in W^1_q}, we obtain
\begin{equation*}
\dualq{\R g}{\Div v}
= \dualb{g}{\nu \cdot v}
  - \dualq{\nabla \R g}{v}.
\end{equation*}
Therefore, we can estimate
\begin{equation*}
\begin{split}
| \dualb{g}{\nu \cdot v} |
&\le | \dualq{\nabla \R g}{v} |
+ | \dualq{\R g}{\Div v} | \\
&\le \| \nabla \R g \|_q \| v \|_{q'}
+ \| \R g \|_q \| \Div v \|_{q'}\\
&\le \| \R g \|_{W^1_q(\Omega)} \| v \|_{E_{q'}(\Omega)} \\
&\le C \| g \|_{W^{1-1/q}_q(\partial \Omega)} \| v \|_{E_{q'}(\Omega)},
\end{split}
\end{equation*}
where $C = C(n,q,\Omega) > 0$. We obtain
\begin{equation*}
\Tr_\nu v := \Big[ W^{1-\frac{1}{q}}_q(\partial \Omega) \ni g \mapsto \dualb{g}{\nu \cdot v} \Big]
\in W^{-\frac{1}{q'}}_{q'}(\partial \Omega)
\end{equation*}
with
$\| \Tr_\nu v \|_{W^{-1/q'}_{q'}(\partial \Omega)} \le C \| v \|_{E_q(\Omega)}$.
Now, Lemma~\ref{thm:density E_q} gives the assertion.
\end{proof}

\begin{lemma}[Green's formula in $E_q$] \label{thm:Greens formula in E_q}
Let $\Omega \subset \mathbb{R}^n$ be a domain with uniform $C^{2,1}$-boundary, $n \ge 2$ and $1 < q < \infty$.
We have for $u \in W^1_q(\Omega)$ and $v \in E_{q'}(\Omega)$
\begin{equation} \label{eq:Green's formula in E_q}
\int_\Omega u (\Div v) \D \lambda_n
= \dualb{u}{\nu \cdot v} - \int_\Omega \nabla u \cdot v \D \lambda_n.
\end{equation}
\end{lemma}

\begin{proof}
Due to Lemma~\ref{thm:density E_q} we can choose a sequence $(v_k)_{k \in \mathbb{N}} \subset W^1_{q'}(\Omega)^n$ converging to $v$ in $E_{q'}(\Omega)$. Now Lemma~\ref{thm:Greens formula in W^1_q} gives that \eqref{eq:Green's formula in E_q} is true for $v_k$ instead of $v$.
It is not hard to see that, for $k \to \infty$, the two terms $\int_\Omega u
(\Div v_k) \D \lambda_n$ and $\int_\Omega \nabla u \cdot v_k \D
\lambda_n$ converge to $\int_\Omega u (\Div v) \D \lambda_n$ and
$\int_\Omega \nabla u \cdot v \D \lambda_n$ respectively.
Using the continuity of
$\Tr: W^1_q(\Omega) \rightarrow W^{1-1/q}_q(\partial \Omega)$ and
$\Tr_\nu: E_{q'}(\Omega) \rightarrow W^{-1/q'}_{q'}(\partial \Omega)$, we obtain the third term $\dualb{u}{\nu \cdot v_k}$ converging to $\dualb{u}{\nu \cdot v}$ as well, for $k \to \infty$.
\end{proof}

\begin{lemma}[Extended Gauß theorem] \label{thm:Extended Gauss theorem}
Let $\Omega \subset \mathbb{R}^n$ be a domain with uniform $C^{2,1}$-boundary, $n \ge 2$ and $1 < q < \infty$.
For $u \in W^1_q(\Omega)$ and $v \in E_{q'}(\Omega)$ we have
\begin{equation*}
\int_\Omega \Div (u v) \D \lambda_n
= \dualb{u}{\nu \cdot v}.
\end{equation*}
\end{lemma}

\begin{proof}
Via approximation (Lemma~\ref{thm:density E_q}) we obtain that
$\Div (u v) = \nabla u \cdot v + u (\Div v)$ is a function in $L_1(\Omega)$, so the left-hand side of the formula is well-defined.
Lemma~\ref{thm:Greens formula in E_q} yields the assertion.
\end{proof}



\end{document}